\documentclass[a4paper,10pt]{amsart}
\usepackage[utf8]{inputenc}
\usepackage{latexsym, amsfonts, amsthm, amsmath, amssymb, tikz, enumitem, booktabs, float,color,multicol,epsf, url,epsfig,epstopdf, array, setspace, graphicx, csquotes, textcmds}
\usepackage{tikz-cd}
\usetikzlibrary{matrix, calc, arrows}

\theoremstyle{plain}
\newtheorem{thm}{Theorem}[section]
\newtheorem{cor}{Corollary}[section]

\newtheorem{prop}{Proposition}[section]

\newtheorem{lem}{Lemma}[section]

\newtheorem{rmk}{Remark}[section]
\title[Holomorphic curves in the symplectizations of lens spaces]{Holomorphic curves in the symplectizations of\\ lens spaces: an elementary approach}
\author{Murat Sa\u glam}

\begin{document}

\begin{abstract}
We present an elementary computational scheme for the moduli spaces of rational pseudo-holomorphic curves in the symplectizations of 3-dimensional lens spaces, which are equipped with Morse-Bott contact forms induced by the standard Morse-Bott contact form on $S^3$. As an application, we prove that for $p$ prime and $1<q,q'<p-1$, if there is a contactomorphism between lens spaces $L(p,q)$ and $L(p,q')$, where both spaces are equipped with their standard contact structures, then $q\equiv (q')^{\pm 1}$ in$\mod p$. For the proof we study the moduli spaces of pair of pants with two non-contractible ends in detail and establish that the standard almost complex structure that is used is regular. Then the existence of a contactomorphism enables us to follow a neck-stretching process, by means of which we compare the homotopy relations encoded at the non-contractible ends of the pair of pants in the symplectizations of $L(p,q)$ and $L(p,q')$. Combining our proof with the result of Honda on the classification of universally tight contact structures on lens spaces, we provide a purely symplectic/contact topological proof of the diffeomorphism classification of lens spaces in the class mentioned above. 
\end{abstract}

\maketitle
\tableofcontents

\section{Introduction}
The objects of the study in this paper are 3-dimensional lens spaces and the standard Reeb dynamics on them. These dynamical systems are given as finite quotients of the Hopf fibration on $S^3$. To be more precise, we consider \emph{the standard contact form on $S^3$} given by
\begin{eqnarray}\label{alpha0}
\alpha_0(z)[w]={\rm Im} \langle w,z\rangle_{\mathbb{C}}
\end{eqnarray}
where $z=(z_1,z_2)\in S^3\hookrightarrow \mathbb{C}^2$, $w=(w_1,w_2)\in T_z S^3\subset T_z\mathbb{C}^2\cong \mathbb{C}^2$ and $\langle \cdot,\cdot\rangle_{\mathbb{C}}$ is the standard hermitian product. Then \emph{the standard contact structure} $\xi_0$ at $z\in S^3$ is given by the hermitian complement of the complex line along $z\in \mathbb{C}^2$. The Reeb vector field and its flow are given by 
\begin{eqnarray}\label{R0}
R_0(z_1,z_2)=(iz_1,iz_2),\;\phi_t(z_1,z_2)=(e^{it}z_1,e^{it}z_2).
\end{eqnarray}
We see that all Reeb orbits are closed and we have $2\pi$ as the common minimal period. The fibration via Reeb orbits leads to the Hopf fibration
\begin{eqnarray}\label{hopf}
\pi: S^3\rightarrow \mathbb{C}P^1, (z_1,z_2)\mapsto (z_1:z_2).
\end{eqnarray}
We consider the $\mathbb{Z}_p$-action on $S^3$ generated by 
\begin{eqnarray}\label{sigma}
\sigma(z_1,z_2)=(e^{i\theta}z_1,e^{iq\theta}z_2)
\end{eqnarray}
where $0<q<p$ are integers with $(p,q)=1$ and $\theta=2\pi/p$. This is a free action and gives rise to the lens space $L(p,q)$ as the quotient space, where
\begin{eqnarray}\label{P}
\texttt{p}: S^3\rightarrow L(p,q)
\end{eqnarray}
is the quotient/covering map. We note that since $\texttt{p}$ is a universal covering,
 $$\pi_1(L(p,q))\cong H_1(L(p,q),\mathbb{Z})\cong \mathbb{Z}_p.$$ For later purposes we fix the integer $0<v<p$ such that 
\begin{equation}\label{defnofv}
vq\equiv 1 \mod p.
\end{equation}
Since $\sigma$ is a complex linear map on $\mathbb{C}^2$, it preserves the standard
contact form. Hence we have an induced \textit{standard contact  form} $\alpha$ on $L(p,q)$, given by \begin{eqnarray}\label{alpha}
\texttt{p}^*\alpha=\alpha_0,
\end{eqnarray}
which defines the \textit{standard contact structure} 
\begin{eqnarray}\label{xi}
\xi=\ker \alpha =\texttt{p}_*\xi_0.
\end{eqnarray}
We also note that $\texttt{p}_*R_0=R$, where $R$ is the Reeb field of $\alpha$.
\begin{rmk}\label{pqconditions}
(1) In this paper, we restrict ourselves to the case where  $1<q<p-1$. In the remaining case, the Reeb flow still induces an $S^1$-bundle structure and it is relatively less interesting, see \cite{Bourgeois}.

(2) We also assume that $p$ is prime. This assumption simplifies the computational aspect of our work but we claim that most of the statements we present in this paper can be generalized after some extra care.

(3) In what follows, the expressions like $k\equiv l$ and $k\not\equiv l$ should be understood in$\mod p$.
\end{rmk}
The dynamics of the Reeb field $R$ is of Morse-Bott type and is easy to describe since $\sigma$ commutes with the flow of $R_0$. There are two closed orbits of $R_0$ in $S^3$ which are invariant under the group action,
namely
\begin{eqnarray}\label{gammainfty}
\gamma_\infty=\{(z_1,0)\,|\,|z_1|=1\}=\{(e^{it},0)\,|\,t\in[0,2\pi]\}
\end{eqnarray}
and 
\begin{eqnarray}\label{gamma0}
\gamma_0=\{(0,z_2)\,|\,|z_2|=1\}=\{(0,e^{it})\,|\,t\in[0,2\pi]\}.
\end{eqnarray}
These orbits are mapped to the points $\infty$ and $0$ in $\mathbb{C}P^1\cong \mathbb{C}\cup\{\infty\}$ respectively via (\ref{hopf}) and  they project
to $p$-fold covers of simple orbits $\overline{\gamma}_\infty$ and $\overline{\gamma}_0$ in $L(p,q)$. 
These simple orbits have period $2\pi/p$ and they represent generators of the fundamental group.
We note that the simple orbits of $R_0$ besides $\gamma_{0}$ and $\gamma_\infty$ are permuted by $\sigma$. In fact we have an induced action on the orbit space 
\begin{eqnarray}\label{sigmaflat}
\sigma_\flat:\mathbb{C}P^1\rightarrow \mathbb{C}P^1, (z_1:z_2)\mapsto (e^{i\theta}z_1:e^{iq\theta}z_2),
\end{eqnarray}
which gives an orbifold structure to the orbit spaces of $R$ via the quotient map
\begin{eqnarray}\label{Pflat}
\texttt{p}_\flat:\mathbb{C}P^1\rightarrow \overline{\mathbb{C}P^1}.
\end{eqnarray}
The fixed points 0 and $\infty$ correspond to the (iterations of) orbits
$\overline{\gamma}_0$ and $\overline{\gamma}_{\infty}$ respectively. 
The isotropy group of both singularities is $\mathbb{Z}_p$.

Our aim is to study punctured holomorphic curves in the symplectizatios of lens spaces. These curves are the main objects studied in symplectic field theory (SFT) see \cite{SFT}. As pointed above, the standard Reeb flow on $S^3$ is perfectly symmetric and if one chooses an almost complex structure on $\xi_0$, which is invariant under the Reeb flow, the resulting SFT-type almost complex structure on $\mathbb{R}\times S^3$ leads to a description of punctured curves in terms of the closed curves in the orbit space $\mathbb{C}P^1$ paired with meromorphic sections above them since $\mathbb{R}\times S^3$ may be viewed as a complex (in fact holomorphic) line bundle without its zero section, see section 2.  This is a particular example of a phenomenon observed for the pre-quantization bundles, see \cite{Bourgeois}. 

If such a symmetric almost complex structure on  $\mathbb{R}\times S^3$ is also invariant under the $\mathbb{Z}_p$- action, then it descends to  $\mathbb{R}\times L(p,q)$ and one may study punctured curves in $\mathbb{R}\times L(p,q)$ explicitly in terms of their lifts to $\mathbb{R}\times S^3$, which are obtained after precomposing the curves in $\mathbb{R}\times L(p,q)$ with a suitable covering of their domains. This idea goes back to \cite{Hind} and executed in detail in \cite{murthesis} for lens spaces and their unit cotangent bundles. It turns out that it is easy to determine whether a given moduli space of rational curves is non-empty and has the correct dimension. We present a systematic treatment of this idea in the next section. For an application of the resulting computational scheme, we prove the following. 
\begin{thm}\label{teoatintro} Let $p$ be a prime number and $1<q,q'<p-1$. Let $\alpha$ and $\alpha'$ be the standard contact forms on $L(p,q)$ and $L(p,q')$ respectively. 
Suppose that there is a positive contactomorphism $$\varphi: \left(L(p,q),\xi=\ker \alpha\right)\rightarrow \left(L(p,q'),\xi'=\ker \alpha'\right).$$ 
Then $
q\equiv (q')^{\pm1}$.
\end{thm}
Here \emph{the positivity of $\varphi$} means that $\varphi^* \alpha'=f\alpha$ for some positive smooth function $f$ on $L(p,q)$. 
For the proof we study the moduli spaces of pair of pants with non-contractible positive ends in detail. We show that non-empty components of such moduli spaces are cut out transversally by comparing the Fredholm index of the pair of pants with the dimension of the equivariant pseudo-holomorphic perturbations of their lifts. This enables us to perturb the data and carry on a standard neck-stretching argument after the given contactomorphism. Noting the homotopy/homology relation $[\overline{\gamma}_{\infty}]=q[\overline{\gamma}_0]$, the result follows from the comparison of the homotopy relation that is encoded by the positive ends of certain pair of pants in $\mathbb{R}\times L(p,q)$ with the corresponding relation encoded by the pair of pants in $\mathbb{R}\times L(p,q')$, which is obtained after stretching the neck, see Figure \ref{outline}. 

The lens spaces were introduced by W. Dyck in 1884 and since then the problem of their classification had attracted many great mathematicians like H. Tietze, H. Poincar\'{e}, J.W. Alexander, H. Seifert, see \cite{hist} for a more precise historical account. In 1935, the first complete classification of lens spaces was given by Reidemeister in the category of piecewise linear (PL) homeomorphisms \cite{Reidemeister} and later in 1960, Brody showed that the homeomorphism classification coincides with the classification in PL category, see \cite{Brody}. After many other contributions, we finally have the following classification statement for 3-dimensional lens spaces, see \cite{Cohen} for a modern proof. 
\begin{thm}\label{class} The lens spaces $L(p,q)$ and $L(p,q')$ are
\begin{itemize}
\item homotopy equivalent if and only if 
$$qq'\equiv \pm a^2$$
 for some $a\in \mathbb{N}$,
\item simple homotopy equivalent/homeomorphic/diffeomorphic if and only if $$q\equiv \pm (q')^{\pm 1}.$$
\end{itemize} 
\end{thm} 
One of the important aspects of the above theorem is providing examples of manifolds which are homotopy equivalent but not simple homotopy equivalent/ homeomorphic/ diffeomorphic.  We also note that if $q\equiv \pm (q')^{\pm 1}$ is assumed then it is easy to define an explicit diffeomorphism between $L(p,q)$ and $L(p,q')$ for each case of the assumption. Hence the \qq{only if} part of the second statement of Theorem \ref{class} is the non-trivial part and it requires sophisticated topological tools like the Reidemeister-Franz torsion, which was introduced by Reidemeister in \cite{Reidemeister} for his classification result and later generalized by Franz in \cite{Franz}. 

We note that in topological point of view Theorem \ref{teoatintro} is a straightforward consequence of Theorem \ref{class}. If we orient lens spaces via the standard contact forms, then due to the dimensional reasons any contactomorphism, independent of being positive or negative, is an orientation preserving diffeomorphism in the first place. Moreover after paying more attention to the smooth classification, one realizes that the existence of an orientation preserving diffeomorphism between $L(p,q)$ and $L(p,q')$ implies that $q\equiv(q')^{\pm 1}$.

On the other hand once combined with the classification of universally tight contact structure on 3-dimensional lens spaces given by Honda in \cite{Honda}, whose proof is purely contact topological and does not rely on any classification statement on lens spaces, Theorem \ref{teoatintro} gives a pure symplectic/contact topological proof of the following classification statement. 
\begin{thm} Assume that $p$ is prime and $1<q,q'<p-1$. If $L(p,q)$ and $L(p,q')$ are diffeomorphic, then $q\equiv \pm(q')^{\pm 1}$. 
\end{thm}
\begin{proof} We orient $L(p,q)$ and $L(p,q')$ via the volume forms $\alpha\wedge d\alpha$ and $\alpha'\wedge d\alpha'$ respectively, where $\alpha$ and $\alpha'$ are the standard contact forms. Let $\varphi: L(p,q)\rightarrow L(p,q')$ be an orientation preserving diffeomorphism. We put $\beta:=\varphi^*\alpha'$. Since $\varphi$ is orientation preserving 
$\beta \wedge d\beta$ is a positive volume form on $L(p,q)$ and consequently $(\ker \beta, d\beta)$ is a positive contact form in the sense of \cite{Honda}. It is also clear that $\ker \beta$ is universally tight. By Proposition 5.1 in \cite{Honda}, there are precisely two universally tight positive contact structures on $L(p,q)$ since $q\neq p-1$. We note that $(\xi, d\alpha)$ and $(\xi, -d\alpha)$ are two positive universally tight contact structures on $L(p,q)$ and they are not isotopic as oriented contact structures since they are distinguished by their Euler classes. In fact an easy computation shows that the Poincar\' e dual of the Euler class of $(\xi, d\alpha)$ is given by $(q+1)[\overline{\gamma}_0]\in H_1(L(p,q),\mathbb{Z})$. Note that the Poincar\' e dual of the Euler class of $(\xi, -d\alpha)$ is then given by $(-q-1)[\overline{\gamma}_0]$ and  $q+1\not\equiv -q-1$ since $q\neq p-1$.  Hence $(\ker \beta, d\beta)$ is isotopic to either $(\xi, d\alpha)$ or $(\xi, -d\alpha)$ as positive contact structures. Hence we either have a positive contactomorphism 
$$\psi^+:\left(L(p,q),\xi= \ker \alpha\right)\rightarrow\left(L(p,q), \ker \beta\right),$$
or a positive contactomorphism 
$$\psi^-:\left(L(p,q),\xi= \ker (-\alpha)\right)\rightarrow\left(L(p,q),\ker \beta\right).$$
In the first case, $\varphi\circ \psi^+$ is the desired positive contactomorphism. In the second case  we consider the diffeomorphism $\psi$ of $L(p,q)$, which is induced by the map $(z_1,z_2)\rightarrow(\overline{z}_1,\overline{z}_2)$ on $S^3$. It is easy to see that $\psi^*\alpha=-\alpha$ and therefore $\varphi\circ \psi^+\circ \psi$ is a positive contactomorphism. By Theorem \ref{teoatintro} we get $q\equiv (q')^{\pm 1}$ in either case. 

Now we assume that  $\varphi: L(p,q)\rightarrow L(p,q')$ is an orientation reversing diffeomorphism. Then we consider the diffeomorphism $\psi: L(p,-q)\rightarrow L(p,q)$, which is induced by the map $(z_1,z_2)\rightarrow(z_1,\overline{z}_2)$ on $S^3$. It is easy to see that $\psi$ is orientation reversing. Hence $\varphi\circ\psi$ is an orientation preserving diffeomorphism. Applying the above argument to $\varphi\circ \psi$ leads to $-q\equiv (q')^{\pm 1}$. 
\end{proof}

The study of holomorphic curves in the symplectization of lens spaces with respect to the Morse-Bott data is not new. We note that these curves can be studied directly as orbicurves in the orbit space $\overline{\mathbb{C}P^1}$ paired with meromorphic sections of  orbibundles. This leads to a formal relation between the orbifold Gromov-Witten potential of the orbit space $\overline{\mathbb{C}P^1}$ and the SFT hamiltonian of $L(p,q)$, see \cite{Rossi}. Compared to \cite{Rossi} our approach is rather elementary and serves us well since we do not aim to compute the SFT hamiltonian. We just provide a computational scheme and applying this scheme to a small portion of the SFT hamiltonian, namely pair of pants with positive non-contractible ends, gives us enough information to prove the classification statement.

Another work that aligns with our paper is presented in \cite{Chen}. In \cite{Chen} it is proven that if there is a symplectic homology cobordism between two lens spaces, which are equipped with the standard contact structures, then the cobordism is necessarily trivial. This statement is an application of the intersection theory for pseudo-holomorphic curves in 4-dimensional symplectic orbifolds, which is the main result of \cite{Chen}. Given a symplectic homology cobordism between two lens spaces, one compactifies both ends of the symplectic cobordism in a particular way and studies a moduli space of holomorphic spheres with certain properties in the resulting closed symplectic orbifold. The moduli space turns out to be a closed 2-dimensional orbifold and one shows that after suitably compactified, the total space of a particular line bundle over this moduli space is diffeomorphic to the compactification of the symplectic cobordism by means of an evaluation map. On the other hand the diffeomorphism type of the compactified line bundle turns out to be the same as the compactification of the trivial cobordism over one of the given lens spaces. Although our treatment of holomorphic curves stay in the smooth category and we study punctured curves in the SFT setting, due to the Morse-Bott data the punctured curves that appear in this work can be also thought as closed curves in suitable singular compactifications of the symplectizations or symplectic cobordisms, which are essentially different than the symplectic orbifold used in \cite{Chen}. On the other hand, in order to achieve the necessary  control on the neck-stretching process, we lift punctured curves to suitable coverings and apply the classical intersection theory to the extensions of these lifts, which are now smooth curves.  

\textbf{Acknowledgements.} This paper is the outcome of both my Ph.D. study in Leipzig and my postdoctoral research in Bochum. To this end I am grateful to Matthias Schwarz, Richard Hind, Alberto Abbondandolo and Barney Bramham for their support during these periods. I thank Hansj\" org Geiges and Kai Zehmisch for pointing out the classification statement on contact structures on lens spaces. I thank Richard Siefring for sharing his work in preparation, for very helpful discussions and for reading parts of this manuscript. I also thank Felix Schm\"aschke for fruitful discussions. This work is part of a project in the SFB/TRR 191 `Symplectic Structures in Geometry,
Algebra and Dynamics', funded by the DFG.
\section{Holomorphic curves in the symplectizations of lens spaces: the general scheme}
 Let $(M,\alpha)$ be a closed, $(2n-1)$-dimensional contact manifold and let $R_\alpha$ be the associated Reeb vector field with the flow $\phi^t$, which is of \textit{Morse-Bott} type. Namely, the action spectrum of $\alpha$ is discrete and for any action value (period) $T$, $N_T:=\{x\in M \:|\: \phi^T(x)=x\}\subseteq M$ is a closed submanifold  such that the rank of $d\alpha_{|N_T}$ is locally constant and $T_xN_T=\ker (d\phi^T-{\rm id})_x$ for all $x\in N_T
$. In the case that $N_T$ consists of a single orbit 
$\gamma:\mathbb{R}\rightarrow M$, we say that $\gamma$ is \textit{non-degenerate}. We note that in this case $d\phi^T(\gamma(0))_{|\xi_{\gamma (0)}}$ does not admit 1 as an eigenvalue. We note that for any period $T$, the Reeb flow defines an $S^1$-action on $N_T$ and the resulting orbit space $S_T=N_T/S^1$ is an orbifold, which consists of a single point if $N_T$ is geometrically a single orbit.    

Given an almost complex structure $J$ on $\xi=\ker \alpha$ that is compatible with $d\alpha$, one extends it to the \textit{symplectization} $(\mathbb{R}\times M,d(e^a\lambda))$ in such a way that $J$ is invariant under translations along $\partial_a$ and $J\partial_a=R_\alpha$. 

Let $(\Sigma,j)$ be a closed Riemann surface and let $\Gamma\subset \Sigma$ be a finite set of punctures. We consider the maps that satisfy the non-linear Cauchy-Riemann equation, namely 
$$u:\Sigma\setminus\Gamma\rightarrow \mathbb{R}\times M,\; du\circ j=J\circ du.$$
We call such a map as a \emph{(punctured) pseudo-holomorphic curve} or a \emph{holomorphic curve} in short. We say a holomorphic curve $u$ is \emph{rational} if $\Sigma=\mathbb{C}P^1$. Note that if $\gamma:\mathbb{R}\rightarrow M$ is a periodic orbit of $R_\alpha$ with period $T$, then \textit{the trivial cylinder over $\gamma$}
\begin{eqnarray}\label{trivialcylinder}
u: \mathbb{R}\times S^1\rightarrow \mathbb{R}\times M,\;(s,t)\mapsto (Ts,\gamma(Tt))
\end{eqnarray}
is a holomorphic curve, where $\mathbb{R}\times S^1:=\mathbb{C}/i\mathbb{Z}$. It turns out that finiteness of a suitable notion of energy forces a pseudo-holomorphic curve to behave as meromorphic objects in complex analysis. The \textit{Hofer energy} of a holomorphic curve $u$ is defined to be
$$E(u):=\sup\limits_{f\in \mathcal{F}}\int_{\Sigma\setminus\Gamma}u^*d(e^f\alpha),$$
where $\mathcal{F}=\{f:\mathbb{R}\rightarrow (-1,1)\;|\; f'>0 \}$.  Under the Morse-Bott assumption, a holomorphic curve $u$ with finite Hofer energy converges to a trivial cylinder near a puncture unless its image is bounded near that puncture and in the latter case $u$ extends holomorphically over a puncture \cite{Bourgeois}. More precisely, for each honest puncture $z\in \Gamma$, one fixes a holomorphic coordinate chart identified with an open disk $\mathbb{D}$ centered at $z$. The puncture set splits as $\Gamma=\Gamma^+\cup \Gamma^-$ and one fixes the following cylindrical coordinates.
\begin{itemize}
\item $Z_+:=[0,+\infty)\times S^1\rightarrow \mathbb{D}\setminus \{0\}\subset\Sigma\setminus\{z\},\;(s,t)\mapsto e^{-2\pi(s+it)}$ if $z\in \Gamma^+$.
\item $Z_-:=(-\infty,0]\times S^1\rightarrow \mathbb{D}\setminus \{0\}\subset\Sigma\setminus\{z\},\;(s,t)\mapsto e^{2\pi(s+it)}$ if $z\in \Gamma^-$.
\end{itemize}
Then for every $z\in\Gamma^{\pm}$, there is a $T$-periodic Reeb orbit $\gamma$ such that 
$$u(s,t)=\exp_{(Ts,\gamma(Ts))}h(s,t),\;(s,t)\in Z_\pm$$
for $|z|$ large, where $h$ is a vector field along the trivial cylinder (\ref{trivialcylinder}) such that $h(s,.)\rightarrow 0$ uniformly as $|s|\rightarrow \infty$ and $\exp$ is the exponential map defined via an $\mathbb{R}$-invariant metric on $\mathbb{R}\times M$. We say that 
\begin{itemize}
    \item $z$ is a \textit{positive puncture} if $z\in \Gamma^+$. We write  $u(z)=(+\infty,\gamma)$ if the the asymptotic end is a non-degenerate orbit $\gamma$ and $u(z)\in \{+\infty\}\times S_T$ if the asymptotic end belongs to a non-trivial orbit space $S_T$.
    
    \item $z$ is a \textit{negative puncture}  if $z\in \Gamma^-$. We write  $u(z)=(-\infty,\gamma)$ if the the asymptotic end is a non-degenerate orbit $\gamma$ and $u(z)\in \{-\infty\}\times S_T$ if the asymptotic end belongs to a non-trivial orbit space $S_T$.  
\end{itemize}

We want to study the moduli space of holomorphic curves for a fixed asymptotic data. We pick a collection of orbit spaces $S^\pm_1,...,S^\pm_{n^\pm}$ and let 
\begin{eqnarray}\label{generalmoduli}
\mathcal{M}=\left\{(\Sigma,j,\Gamma,u)\;|\; u(z^\pm_i)\in\{\pm \infty\}\times S^\pm_i \;\; {\rm for }\;{\rm all}\;\; z^\pm_i\in \Gamma^\pm\right\}/\sim
\end{eqnarray}
 be the moduli space of equivalence classes $[\Sigma,j,\Gamma,u]$ of  holomorphic curves, where $(\Sigma,j,\Gamma,u)\sim (\Sigma',j',\Gamma',u')$ if there exists a biholomorphism $h:\Sigma \rightarrow \Sigma'$ such that $h$ restricts to a sign and ordering preserving bijection on the corresponding puncture sets and $u=u'\circ h$.
\begin{rmk}\label{paramunparam}
In what follows, we will be mostly interested in moduli spaces of rational curves with at least three punctures. We note that due to the uniqueness of the complex structure on $\mathbb{C}P^1$, when $\Sigma=\mathbb{C}P^1$ and $|\Gamma|=3+k$, $k\geq0$, the above moduli space $\mathcal{M}$ is identified with the space of pairs $(u,(z_1,...,z_k))$ where
$$u:\mathbb{C}P^1\setminus \{0,1,\infty,z_1,...,z_k\}\rightarrow \mathbb{R}\times M $$
is a holomorphic curve with prescribed sign of punctures and asymptotic ends.
\end{rmk}

In general it is hard to carry a hand on study of moduli spaces of holomorphic curves. Under nice circumstances, a generic choice of $J$ on $\xi$ leads to a smooth structure on the moduli space. But such a generic choice makes hard to grasp the moduli space itself even in the Morse-Bott case. But in certain perfectly symmetric Morse-Bott situations like pre-quantization bundles, these moduli spaces can be described as rather elementary objects \cite{Bourgeois}.

The aim of this section is to study holomorphic curves in $\mathbb{R}\times L(p,q)$, where the setting is a finite quotient of the perfect Morse-Bott setting on $\mathbb{R}\times S^3$. Our strategy is to study curves in  $\mathbb{R}\times L(p,q)$  through their lifts to $\mathbb{R}\times S^3$.  We want to describe spaces of solutions of the Cauchy-Riemann equation on $\mathbb{R}\times L(p,q)$ as the subspaces of equivariant solutions of lifted problems on $\mathbb{R}\times S^3$. To this end we need first to understand the setting of $\mathbb{R}\times S^3$. 

With real coordinates $z_j=x_j+iy_j$, (\ref{alpha0}) reads as
$$\alpha_0=-y_1dx_1+x_1dy_1-y_2dx_2+x_2dy_2.$$
The symplectic form $d\alpha_0$ on $\xi_0$ given by
$$d\alpha_0=2(dx_1\wedge dy_1+dx_2\wedge dy_2)|_{\xi_0}.$$
The principal $S^1$-bundle given by (\ref{hopf}) is in fact the $S^1$-bundle associated to the tautological line bundle 
\begin{eqnarray}\label{tautbundle}
L\rightarrow \mathbb{C}P^1,\; L_{(z_1:z_2)}={\rm span}_\mathbb{C}\{(z_1,z_2)\}\subset \mathbb{C}^2
\end{eqnarray} 
and the hermitian metric on $L$, which is induced by 
the standard hermitian metric on $\mathbb{C}^2$. Hence the Euler class of (\ref{hopf}) is given by
$$e(\pi)=c_1(L)=-[\pi^{-1}\omega_{FS}].$$  
where $\omega_{FS}$ is the Fubini-Study form on $\mathbb{C}P^1$ s.t. $<\omega_{FS}, [\mathbb{C}P^1]>=\pi$ and
$$\pi^*\omega_{FS}=(dx_1\wedge dy_1+dx_2\wedge dy_2)|_{S^3}.$$

Since the $S^1$-action is generated by the Reeb field $R_0$ and the period is $2\pi$, being the associated contact form, $\alpha_0$ satisfies 
$\mathcal{L}_{R_0}\alpha_0=0$ and $\alpha_0(R_0)=1$. So $\alpha_0$ is a connection 1-form and from the equality above we have $\pi^*\omega_{FS}=\frac{1}{2}d\alpha_0$, 
that is $2\omega_{FS}$ is the curvature form. 

Consider the symplectization $\mathbb{R}\times S^3$. We pull the standard complex structure on $\mathbb{C}P^1$ back to the contact distribution
as an $S^1$-invariant complex structure via the map (\ref{hopf}) and we extend it to the symplectization as described above. We call this almost complex structure as \textit{the standard almost complex structure} and denote it by $J_0$. We note that the extension $\pi:\mathbb{R}\times S^3\rightarrow \mathbb{C}P^1$
of (\ref{hopf}) is $J_0$- holomorphic.  We consider the diffeomorphism
\begin{eqnarray}\label{bigidentifi}
\Phi: \mathbb{R}\times S^3\rightarrow L^*,\; (a,(z_1,z_2))\mapsto e^a(z_1,z_2) 
\end{eqnarray}
where $L^*$ is the total space of the line bundle without the zero section.  
We note that once conjugated by $\Phi$, $J_0$ coincides the complex structure on the fibres of $L$. Since $\Phi$ also covers holomorphic bundle projections on both its domain and target, it is a biholomorphism. 

For any punctured curve 
$$u:\Sigma \setminus\Gamma \rightarrow \mathbb{R}\times S^3$$  
with finite energy, the holomorphic map $c=\pi\circ u$ extends over the punctures and 
gives a closed holomorphic curve in $\mathbb{C}P^1$. The map $u$ then corresponds to
a meromorphic section $f$ of the holomorphic line bundle $c^*L\rightarrow \Sigma$, see \cite{Bourgeois}. The positive ends of $u$ correspond to the poles of $f$ and negative
ends of $u$ corresponds to the zeros of $f$ since 
the complex structure on the symplectization fits to the complex structure on the fibres of $L$. 
Note that the first Chern number of the bundle $c^*L$ is given by $-d$ where $d\geq 0$ is the degree of the map $c$. If $\Sigma=\mathbb{C}P^1$ then necessary and sufficient condition for the existence of a meromorphic section $f$ is both the divisor of the section and the bundle to have the same degree, namely
$$\#f^{-1}(0)-\#f^{-1}(\infty)=-d.$$
We note that the above formula says that there is no punctured curve with only negative ends, which is consistent with the maximum principle.

We now want to add the $\mathbb{Z}_p$ action into the setting. 
This action is free on $L^*$ but has two fixed points on the zero section $\mathbb{C}P^1$, see (\ref{sigmaflat}).
Note that the standard almost complex structure $J_0$ on $\mathbb{R}\times S^3$ is $\mathbb{Z}_p$-invariant so 
we identify $\mathbb{R}\times L(p,q)$ with $\overline{L}^*:=L^*/\mathbb{Z}_p$ where the former space is equipped with 
the quotient almost complex structure denoted by $J_{\alpha}$. We call $J_{\alpha}$ as the \textit{standard almost complex structure on $\mathbb{R}\times L(p,q)$}.

We consider a rational $J_{\alpha}$- holomorphic curve $$\overline{u}:\mathbb{C}P^1\setminus\Gamma \rightarrow \mathbb{R}\times L(p,q)\cong \overline{L}^*.$$
Such a curve lifts to the cover $L^*$ if and only if 
$\overline{u}_*=0$ on $\pi_1(\mathbb{C}P^1\setminus\Gamma)$. We note that $\mathbb{C}P^1\setminus\Gamma$ is homotopy equivalent to the bouquet of 
$(\#\Gamma-1)$-many 
circles so its fundamental group is the free group with $(\#\Gamma-1)$ generators. 
Hence the image of the fundamental group is trivial if and only if the image of each generator is trivial, 
that is, if all of the asymptotics are contractible. In this paper, our main concern is about curves with non-contractible ends. But since a curve $\overline{u}$ with a non-contractible end does not lift 
to the cover immediately, we need to 
precompose the map $\overline{u}$ with a suitable covering map 
\begin{eqnarray}\label{fracp}
\mathfrak{p}: \Sigma\setminus\tilde{\Gamma}\rightarrow \mathbb{C}P^1\setminus\Gamma.
\end{eqnarray}
Once we pick a suitable cover (\ref{fracp}), we have the  
commutative diagram given by Figure \ref{commute1}.
The problem is then to study the lifted $J_0$- holomorphic curves 
$$u:\Sigma\setminus\tilde{\Gamma}\rightarrow \mathbb{R}\times S^3\cong L^*$$
which are equivariant with respect to  the action of the group of Deck transformations $G$ on the punctured surface
$\Sigma\setminus\tilde{\Gamma}$ and $\mathbb{Z}_p$-action on $L^*$.
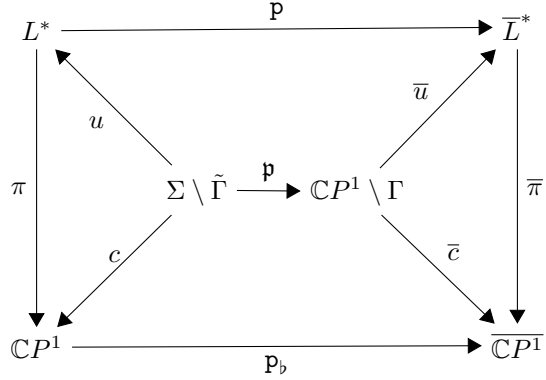
\begin{figure}
\begin{tikzpicture}[>=triangle 60,]
  \matrix[matrix of math nodes,column sep={60pt,between origins},row
    sep={60pt,between origins},nodes={rectangle}] (s)
  {
    |[name=A]|  L^* & & &|[name=D]| \overline{L}^* \\ 
    & |[name=L]| \Sigma\setminus\tilde{\Gamma} &|[name=K]| \mathbb{C}P^1\setminus\Gamma & \\ 
    |[name=B]| \mathbb{C}P^1 & & &|[name=E]| \overline{\mathbb{C}P^1}  \\
  };
  \draw[->] (A) edge node[auto] {\(\texttt{p}\)} (D)
	    (A) edge node[left] {\(\pi\)} (B)
	    (B) edge node[below] {\(\texttt{p}_\flat\)} (E)
	    (D) edge node[auto] {\(\overline{\pi}\)} (E)
	    (L) edge node[auto] {\(u\)} (A)
	    (L) edge node[above] {\(c\)} (B)
	    (L) edge node[auto] {\(\mathfrak{p}\)} (K)
	    (K) edge node[auto] {\(\overline{u}\)} (D)
	    (K) edge node[auto] {\(\overline{c}\)} (E)
  ;
\end{tikzpicture}
    \caption{Lifting diagram for a holomorphic curve in $\mathbb{R}\times L(p,q)$.}
    \label{commute1}
\end{figure}
\subsection{Equivariant curves: the necessary conditions for the existence}
In this section, our aim is to determine a suitable minimal covering (\ref{fracp}) and establish a 
correspondence between a given moduli problem in $\overline{L}^*$ and a lifted moduli 
problem in $L^*$, which is determined by (\ref{fracp}).

We fix a sets of punctures $\Gamma\subset \mathbb{C}P^1$, which partitions as follows.
\begin{itemize}
    \item $\Gamma=\Gamma_{nc}\cup \Gamma_c$ with cardinalities $n_{nc}$ and $n_{c}$.
    \item $\Gamma_{nc}=\Gamma_0\cup \Gamma_{\infty}$ with cardinalities $n_0$ and $n_\infty$ so that $n_{nc}=n_0+n_\infty$. 
    \item $\Gamma_0=\Gamma^{+}_0\cup \Gamma^{-}_0$ with cardinalities $n^{+}_{0}$ and $n^{-}_{0}$ so that 
    $n_0=n^{+}_{0}+n^{-}_{0}$. 
    \item $\Gamma_\infty=\Gamma^+_{\infty}\cup \Gamma^{-}_{\infty}$ with cardinalities $n^{+}_{\infty}$ and $n^{-}_{\infty}$ so that $n_\infty=n^{+}_{\infty}+n^{-}_{\infty}$.
    \item $\Gamma_c=\Gamma_c^+\cup \Gamma_c^-$ with cardinalities 
    $n^+_c$ and $n_c^-$ so that $n_c=n_c^++n_c^-$.
    \item $\Gamma^{\pm}_{0}=\{z^{0,\pm}_1,...,z^{0,\pm}_{n^\pm_0}\}$, $ \Gamma^{\pm}_{\infty}=\{z^{\infty,\pm}_1,...,z^{\infty,\pm}_{n^\pm_{\infty}}\}$ and $\Gamma^\pm_c=\{w^\pm_1,...,w^\pm_{n_c^\pm}\}$.
\end{itemize}


Suppose that we have a holomorphic curve 
\begin{eqnarray}\label{ubarwithends}
\overline{u}:
\mathbb{C}P^1\setminus\Gamma
\rightarrow \overline{L}^*\cong \mathbb{R}\times L(p,q)
\end{eqnarray}
with asymptotics
\begin{enumerate}[label=\textrm{(a\arabic*)}]
    \item \label{ncat0}$\overline{u}(z^{0,\pm}_i)=(\pm\infty, k^{0,\pm}_i\overline{\gamma}_0)$ where $k^{0,\pm}_i\not\equiv 0$ for $i=1,...,n^\pm_0$, that is $\overline{u}$ has a positive/negative puncture at $z^{0,\pm}_i$ with positive/negative non-contractible asymptotic end $k^{0,\pm}_i\overline{\gamma}_0$ for $i=1,...,n^\pm_0$.
    \item \label{ncatinfty}$\overline{u}(z^{\infty,\pm}_i)=(\pm\infty, k^{\infty,\pm}_i\overline{\gamma}_{\infty})$ where $k^{\infty,\pm}_i\not\equiv 0$ for $i=1,...,n^\pm_\infty$, that is $\overline{u}$ has a positive/negative puncture at $z^{\infty,\pm}_i$ with positive/negative non-contractible asymptotic end $k^{\infty,\pm}_i\overline{\gamma}_{\infty}$ for $i=1,...,n^\pm_\infty$.
    \item \label{cont} $\overline{u}(w^{\pm}_i)\in \{\pm \infty\}\times S^\pm_{k^\pm_i}$ for $i=1,...,n^\pm_c$, that is $\overline{u}$ has a positive/negative puncture at $w^{\pm}_i$ with positive/negative contractible asymptotic end, which lies in the orbit space $S^\pm_{k^\pm_i}$ of orbits with action $2\pi k^\pm_i$.
\end{enumerate}

We fix a point $z\in\mathbb{C}P^1\setminus \Gamma$ and consider the map 
$$\overline{u}_*:\pi_1(\mathbb{C}P^1\setminus \Gamma,z)\rightarrow \pi_1(\overline{L}^*, x_{\overline{u}})\cong\mathbb{Z}_p,$$
where $x_{\overline{u}}=\overline{u}(z)$.
Since $p$ is prime It is clear that $\overline{u}_*$ is surjective if and only if  $n_{nc}\geq 1$. 
We know that $K:=\ker\,\overline{u}_*$ is a normal subgroup of $\pi_1(\mathbb{C}P^1\setminus \Gamma,z)$ and there exists  a covering space 
$$\mathfrak {p}:\Sigma \setminus\tilde{\Gamma}\rightarrow \mathbb{C}P^1\setminus \Gamma,$$
where $\Sigma \setminus\tilde{\Gamma}$ is smooth punctured surface and
$\mathfrak{p}_*(\pi_1(\Sigma \setminus\tilde{\Gamma},\tilde{z}))=K$ where $\tilde{z}$ is a fixed lift of $z$.
This covering is (up to isomorphism) determined by $K$ and the group of Deck transformations $G$ is given by 
\begin{eqnarray}\label{deck}
G\cong \pi_1(\mathbb{C}P^1\setminus \Gamma,z)/K\cong \mathbb{Z}_p.
\end{eqnarray}
We endow $\Sigma \setminus\tilde{\Gamma}$ with the pull-back complex structure  and get a punctured Riemann surface so 
that 
$\mathfrak{p}$ is holomorphic. By the removal of singularities theorem, $\mathfrak{p}$ extends to a holomorphic branched covering
$\mathfrak{p}:\Sigma\rightarrow \mathbb{C}P
^1$.

We note that $\overline{u}\circ \mathfrak {p}:\Sigma \setminus\tilde{\Gamma}\rightarrow \overline{L}^*$ satisfies 
$$(\overline{u}\circ \mathfrak {p})_*(\pi_1(\Sigma \setminus\tilde{\Gamma},\tilde{z}))=\overline{u}_*\mathfrak {p}_*(\pi_1(\Sigma \setminus\tilde{\Gamma},\tilde{z}))=\{0\}=p_*(\pi_1(L^*,\tilde{x}_{\overline{u}}))$$
where $\tilde{x}_{\overline{u}}$ is a fixed lift of $x_{\overline{u}}$. Hence we have the unique lift 
\begin{eqnarray}\label{thelift}
u:=\widetilde{\overline{u}\circ \mathfrak {p}}:\Sigma\setminus\tilde{\Gamma}\rightarrow L^*,
\; u(\tilde{z})=\tilde{x}_{\overline{u}}.
\end{eqnarray}
We note that $u$ is equivariant with respect to the action of $G$ on $\Sigma \setminus\tilde{\Gamma}$ and the action of 
$\mathbb{Z}_p$ on $L^*$. We fix a generator $\tau_\flat$ of $G$ such that
\begin{eqnarray}\label{tauflat}
u\circ\tau_\flat=\sigma\circ u.
\end{eqnarray}
Now we take a closer look at the branched covering $\mathfrak{p}:\Sigma\rightarrow \mathbb{C}P^1$. 
It is clear that $\tilde{\Gamma}=\mathfrak{p}^{-1}(\Gamma)$. We put 
$\tilde{\Gamma}_{nc}:=\mathfrak{p}^{-1}(\Gamma_{nc})$
and $\tilde{\Gamma}_{c}= \mathfrak{p}^{-1}(\Gamma_c)$.
\begin{lem} \label{genuslemma}Each point in $\Gamma_{nc}$ is a branch point with exactly one preimage and 
each point in $\Gamma_{c}$ is regular and has exactly p preimages. $\tilde{\Gamma}_{nc}$ forms the fixed point set 
of the extended action of $G$ over $\Sigma$. Moreover, the genus of $\Sigma$ is given by
\begin{eqnarray}\label{genus}
g=\frac{(p-1)(n_{nc}-2)}{2}.
\end{eqnarray}
\end{lem}
\begin{proof} For any $z^{i,\pm}_j\in \Gamma_{nc}$, the cardinality of the set $\mathfrak{p}^{-1}(z^{i,\pm}_j)$ 
is either $p$ or $1$. In fact, any element in $\mathfrak{p}^{-1}(z^{i,\pm}_j)$ 
has a local isotropy group with respect to the extended action of $G$ on $\Sigma$, 
which is a subgroup of 
$G\cong \mathbb{Z}_p$. Since $p$ is prime, this subgroup is either trivial or $\mathbb{Z}_p$. 
We note also that at least one point in $\mathfrak{p}^{-1}(z^{i,\pm}_j)$ is 
a puncture of $u$ with an asymptotic, which descends to a non contractible asymptotic $\overline{u}$ 
at $z^{i,\pm}_j$. Hence at least one point in $\mathfrak{p}^{-1}(z^{i,\pm}_j)$ has 
isotropy group $\mathbb{Z}_p$ 
and therefore  $\mathfrak{p}^{-1}(z^{i,\pm}_j)$ consists of a single point. 
We conclude that $\mathfrak{p}$ 
branches points over $\Gamma_{nc}$ so that each branch point has a single preimage with ramification 
number $p$. 
Similarly, for each $w^\pm_j\in \Gamma_c$  there is a contractible end of $\overline{u}$, which 
lifts to $p$-many contractible ends of $u$. Hence 
the punctures corresponding to these ends are precisely the preimages of $w^\pm_j$. 
Applying the Riemann-Hurwitz formula, we get
$$2-2g=2p-n_{nc}(p-1)\Rightarrow g=\frac{(p-1)(n_{nc}-2)}{2}.$$
It is clear that $\tilde{\Gamma}_{nc}$ the fixed point set of the extended $G$-action.
\end{proof}
We abuse the notation and call preimages of $z^{0,\pm}_i$, $z^{\infty,\pm}_i$ under $\mathfrak{p}$ by 
the same letters. For the preimages of $w^\pm_i$, which is given by $G=\langle \tau_\flat \rangle$-orbit of any point in the preimage, we write $w^\pm_{i,j}$, $j=1,..,p$. 
Now we have an equivariant curve
\begin{eqnarray}\label{theliftwithends}
u:\Sigma\setminus\tilde{\Gamma}\rightarrow L^*
\end{eqnarray}
with asymptotics
\begin{enumerate}[label=\textrm{(la\arabic*)}]
    \item \label{la1} $u(z^{0,\pm}_i)=(\pm\infty, k^{0,\pm}_i \gamma_{0})$ for $i=1,...,n_0^\pm$,
    \item \label{la2}
    $u(z^{\infty,\pm}_i)=(\pm\infty, k^{\infty,\pm}_i\gamma_\infty)$ for $i=1,...,n_\infty^\pm$,
    \item \label{la3}
    $u(w^\pm_{i,j})\in\{\pm\infty\}\times S^\pm_{k^\pm_i}$ for $i=1,...,n_c^\pm$, $j=1,...,p$, 
\end{enumerate}
where now $S^\pm_{k^\pm_i}$ denotes the space of orbits in $S^3$ with the action $2\pi k^\pm_i$.
\begin{figure}[h]
\includegraphics[scale=0.2]{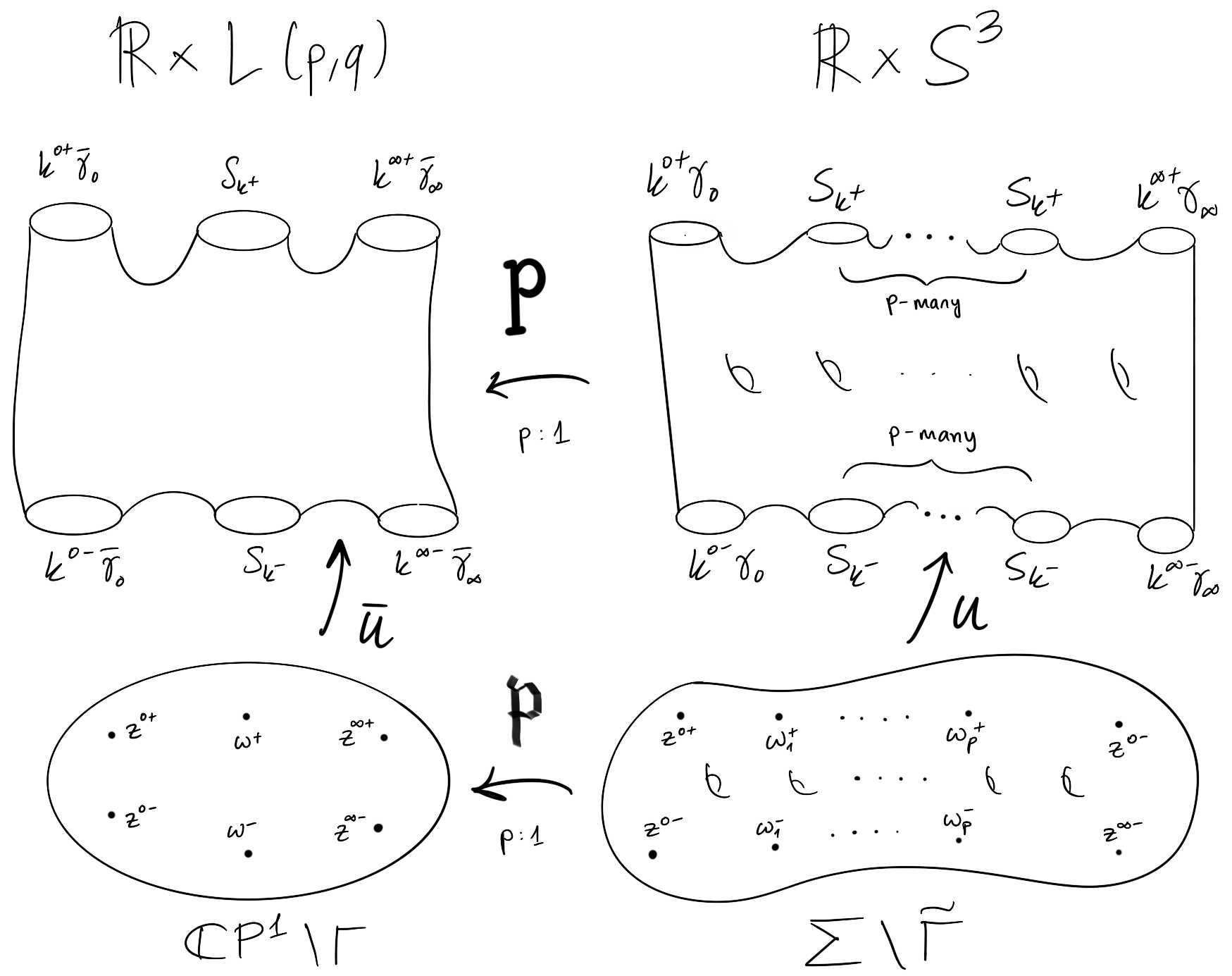}
\centering
\caption{An example of a curve in $\mathbb{R}\times L(p,q)$ and its lift to $\mathbb{R}\times S^3$.} 
\label{liftingscheme}
\end{figure}

\begin{rmk} \label{manylifts} We note the lift $u$ depends on the choice of the lift $\tilde{x}_{\overline{u}}$. In fact, 
imposing  that the lift maps $\tilde{z}$ to $\sigma^k\tilde{x}_{\overline{u}}$,
leads to the unique lift, say $u'$, which satisfies $u'=u\circ \tau_\flat^k$. More precisely, there are $p$ 
lifts of $\overline{u}$, which are distinct as maps, given by $u\circ \tau_\flat^k$, $k=0,...,p-1$ or alternatively by 
$\sigma^k\circ u$, $k=0,...,p-1$. We note that any such lift $u'$ 
satisfies $u'\circ \tau_\flat=\sigma \circ u'$. In particular the image of any lift is invariant under $\sigma$ and the images of all lifts coincide.
\end{rmk}
Now we want to understand to what extent the above description depends on $\overline{u}$. 
\begin{rmk}\label{dependence} By (\ref{genus}) the topology of $\Sigma$ depends only on the number of the 
non-contractible ends of $\overline{u}$.

Since $\pi_1(\overline{L}^*,\ast)$ is abelian the homomorphism $\overline{u}_*$ is invariant under conjugations and therefore depends only on non-contractible asymptotic ends of $\overline{u}$. In fact, given any representation of $\pi_1(\mathbb{C}P^1\setminus\Gamma, z)$ via loops, each being the concatenation of a small simple loop around a puncture and a path joining a point on it to  the base point $z$, image of a generator (associated to a puncture) under $\overline{u}_*$ is equal to the image of the generator (associated to the same puncture) represented by the small simple loop itself, which comes from a representation where the base point $z$ is taken on that very same loop. Hence $\overline{u}_*$ and therefore $K$ is determined up to conjugacy by the data \ref{ncat0} and \ref{ncatinfty}. In particular, if $\overline{u}$
is contained in a continuous family of parametrized curves, one can fix the same subgroup $K$ for the whole family. We also note that the generator $\tau_\flat$ of $G$ can be chosen constantly for such a family of curves.   
\end{rmk}
\begin{rmk} \label{cpxstr} We note that the set of punctures $\Gamma$ is fixed so far but in general, one 
needs to let punctures move. Although the covering $\mathfrak{p}$ topologically depends only 
on $n_{nc}$, the complex structure on $\Sigma$ does 
depend on the positions of the punctures in $\mathbb{C}P^1$. 
\end{rmk}

Now we want to understand the picture for \textit{the base curves}. We fix a lift $u$ given by (\ref{thelift}) 
and get a holomorphic map 
\begin{eqnarray*}
c:=\pi\circ u : \Sigma\setminus \tilde{\Gamma}\rightarrow \mathbb{C}P^1,
\end{eqnarray*}
see Figure \ref{commute1}. After extending $c$ over $\tilde{\Gamma}$, we have a closed curve 
\begin{eqnarray}\label{liftedbasecurve}
c:\Sigma\rightarrow \mathbb{C}P^1\; {\rm s.t. }\; c\circ \tau_\flat=\sigma_\flat\circ c
\end{eqnarray}
in the view of (\ref{sigmaflat}) and (\ref{tauflat}). We call such a curve $c$ as a \textit{lifted base curve}. Since $c$ is equivariant,  
one can pull back the action on the total space  of $L\rightarrow \mathbb{C}P^1$ uniquely to the total spaces of $c^*L\rightarrow \Sigma$
in such a way that the bundle map $c^*L\rightarrow L$, which covers the map $c$, is equivariant. We fix a  generator $\tau$ of this $\mathbb{Z}_p$-action on $c^*L$ so that $u$ correspond to a meromorphic section $u:\Sigma \rightarrow c^*L$ such that 
\begin{eqnarray}\label{equaivariantsectionu}
u\circ \tau_\flat=\tau \circ u.
\end{eqnarray}
Moreover, the section $u$ has poles/zeros at $z^{0,\pm}_i$'s, $z^{\infty,\pm}_i$'s and $w^\pm_{i,j}$'s. We note that 
$$c(z^{0,\pm}_i)=(0:1)=0\in \mathbb{C}P^1=\mathbb{C}\cup \{\infty\},$$
$$\; c(z^{\infty,\pm}_i)=(1:0)=\infty\in \mathbb{C}P^1=\mathbb{C}\cup \{\infty\}$$
for all $i$, which explains the notation for $z^{0/\infty,\pm}_i$.

We want to understand the consequences of (\ref{equaivariantsectionu}) in terms of the multiplicities $k_i^{0/\infty,\pm}$. To this end, we need to understand the $\langle \tau \rangle$- action on $c^*L$
around the fixed points of $\langle \tau_\flat \rangle$-action on $\Sigma$. We note that since we do not know much about the global behaviour of the map $c$, it is hard to describe $\langle \tau_\flat \rangle$-action globally. But since $c$ behaves as a monomial around any point, which is determined by the ramification number at that point, a local description is still easy to get. 

We first fix local trivializations of $L$ and compute 
the $\langle \sigma \rangle$-action on these trivializations.
\begin{itemize}

\item 
\underline{Near $0=(0:1)$}, we fix the local trivialization
\begin{eqnarray}\label{trivat0}
\tilde{\phi}_0:\mathbb{C}\times \mathbb{C}\rightarrow L,\; (z,\lambda)\mapsto \left(\lambda \frac{z}{\sqrt{1+|z|^2}}, \lambda \frac{1}{\sqrt{1+|z|^2}}\right)
\end{eqnarray}
which covers the chart $\phi_0:\mathbb{C}\rightarrow \mathbb{C}P^1,\; z\mapsto (z:1)$. Hence $\sigma$ reads as
\begin{eqnarray}\label{sigmaat0}
\sigma_0(z,\lambda):=(\tilde{\phi}_0^{-1}\sigma\tilde{\phi}_0)(z,\lambda)
=(e^{i(1-q)\theta}z, e^{iq\theta}\lambda).
\end{eqnarray}

\item 
\underline{Near $\infty=(1:0)$}, we fix the local trivialization
\begin{eqnarray}\label{trivatinfty}
\tilde{\phi}_{\infty}:\mathbb{C}\times \mathbb{C}\rightarrow L, \;(z,\lambda)\mapsto \left(\lambda \frac{1}{\sqrt{1+|z|^2}},\lambda \frac{z}{\sqrt{1+|z|^2}}\right)
\end{eqnarray}
which covers the chart $\phi_{\infty}:\mathbb{C}\rightarrow \mathbb{C}P^1, \;z\mapsto (1:z)$. Hence $\sigma$ reads as
\begin{eqnarray}\label{sigmaatinfty}
\sigma_{\infty}(z,\lambda):=(\tilde{\phi}_{\infty}^{-1}\sigma\tilde{\phi}_{\infty})(z,\lambda)
=(e^{i(q-1)\theta}z, e^{i\theta}\lambda).
\end{eqnarray}
\end{itemize}

We recall that $\tilde{\Gamma}_{nc}$ is the fixed point set of the $G$ action on $\Sigma$. 
So using the above trivializations, we may determine the local behaviour of an equivariant
meromorphic section of $c^*L$ near  $\tilde{\Gamma}_{nc}$.
\begin{itemize}

\item 
\underline{Around $z^{0,\pm}_j$:} We take holomorphic coordinates on $\Sigma$ centered at $z^{0,\pm}_j$ so that $c(z)=z^{r}$ for some $r$.  and 
we trivialize $c^*L$ above this coordinate neighbourhood using (\ref{trivat0}).  
Locally, $\tau_{\flat}=e^{im\theta}$ for some $0<m<p$. By (\ref{liftedbasecurve}) and (\ref{sigmaat0}), we have
$$c(\tau_\flat(z))=\sigma_\flat c(z)\;\Rightarrow\; e^{imr\theta}z^r=e^{i(1-q)\theta}z^r\;\Rightarrow\; m\equiv (1-q)r^{-1}.$$
Hence by (\ref{sigmaat0}),  $\tau$ reads locally as
$$\tau(z,\lambda)= (e^{im\theta}z,e^{iq\theta}\lambda).$$
For a local section $z\mapsto (z,f(z))$ of $c^*L$ to be equivariant, it has to satisfy
$$f(e^{im\theta}z)=e^{iq\theta}f(z).$$
Writing the meromorphic section $f$ as $f(z)=\sum\limits_{k} a_kz^k$ around  $z^{0,\pm}_j$, we get
$$\sum_k a_ke^{imk\theta}z^k=\sum_k a_ke^{iq\theta}z^k$$
and we note that
\begin{eqnarray}\label{kat0}
a_k\neq 0 \;\Leftrightarrow \; mk\equiv q\;\Leftrightarrow \; k\equiv qm^{-1}\;\Leftrightarrow\; k\equiv rq(1-q)^{-1}.
\end{eqnarray}

\item
\underline{Around $z^{\infty,\pm}_j$:} Similarly, we have $c(z)=z^{r}$ for some $r$ and we put 
$\tau_\flat=e^{im\theta}$ for some $0<m<p$. By the equivariance of $c$ and (\ref{sigmaatinfty}), we have
$$c(\tau_\flat(z))=\sigma_\flat c(z)\;\Rightarrow\; e^{imr\theta}z^r=e^{i(q-1)\theta}z^r\;\Rightarrow\; m\equiv (q-1)r^{-1}$$
and locally
$$\tau (z,\lambda)= (e^{im\theta}z,e^{i\theta}\lambda).$$
Hence for an equivariant meromorphic section $f=\sum\limits_{k} a_kz^k$, one gets
$$a_k\neq 0 \;\Leftrightarrow \; mk\equiv 1\;\Leftrightarrow \; k\equiv m^{-1}\;\Leftrightarrow\; k\equiv r(q-1)^{-1}.$$
\end{itemize}
The above observations lead to the following lemma.
\begin{lem}\label{localrem} Let $u$ be a lift of $\overline{u}$ and let  $c:\Sigma\rightarrow \mathbb{C}P^1$ be the corresponding lifted base curve. Let $r^{0/\infty,\pm}_i$ denote the local degree of $c$ at $z^{0/\infty,\pm}_i\in\tilde{\Gamma}_{nc}$
 and $m^{0/\infty,\pm}_i$ denote the the local representative of $\tau_\flat$ at $z^{0/\infty,\pm}_i$, that is 
 $\tau_\flat=e^{im^{0/\infty,\pm}_i\theta}$ near $z^{0/\infty,\pm}_i$. 
 Then we have the following relations.
 \begin{enumerate}
  \item for the positive punctures $z^{0/\infty,+}_i$, we have 
  \begin{itemize}
      \item $m^{0,+}_i\equiv (1-q)(r^{0,+}_i)^{-1}$ and $k^{0,+}_i\equiv r^{0,+}_i(1-v)^{-1}$,
      \item $m^{\infty,+}_i\equiv (q-1)(r^{\infty,+}_i)^{-1}$ and $ k^{\infty,+}_i\equiv r^{\infty,+}_i(1-q)^{-1}$. 
  \end{itemize}
  \item for the negative punctures $z^{0/\infty,-}_i$, we have 
  \begin{itemize}
      \item $m^{0,-}_i\equiv (1-q)(r^{0,-}_i)^{-1}$ and $ k^{0,-}_i\equiv r^{0,-}_i(v-1)^{-1}$,
      \item  $m^{\infty,-}_i\equiv (q-1)(r^{\infty,-}_i)^{-1}$ and $ k^{\infty,-}_i\equiv r^{\infty,-}_i(q-1)^{-1}$.
  \end{itemize}

 \end{enumerate}
\end{lem}
\begin{proof} As a meromorphic section of $c^*L$, u satisfies 
$u\circ \tau_\flat=\tau \circ u$. We write $u$ locally as $(z,f(z))$ so that $f$ has poles at $z^{0,+}_i$/$z^{\infty,+}_i$ of order $k^{0,+}_i$/$k^{\infty,+}_i$ and has zeros at $z^{0,-}_i$/$z^{\infty,-}_i$ of order $k^{0,-}_i$/$k^{\infty,-}_i$. Hence at any $z^{0,+}_i$ or $z^{\infty,+}_i$, 
the Laurent expansions terminates at degrees $-k^{0,+}_i$ and $-k^{\infty,+}_j$ respectively, while
at any $z^{0,-}_i$ and $z^{\infty,-}_j$ at degrees $k^{0,-}_i$ and $-k^{\infty,-}_j$ respectively. Making
these adjustments in above local descriptions, we get the required relations.  
\end{proof}
\begin{rmk}\label{localbehav} We note that the relations in the above lemma, depends only on the multiplicities of 
non-contractible ends of $\overline{u}$. Hence given a moduli space of curves 
(parametrized or unparametrized) with fixed non-contractible ends, the multiplicities of these ends 
determine, in mod $p$, local degrees of lifted base curves  at fixed points of $\tau_\flat$ on $\Sigma$ and the behaviour of $\tau_\flat$ around these fixed points. 
\end{rmk}
We point out two immediate necessary conditions for the existence of $\overline{u}$ in terms of the multiplicities of its asymptotic ends. 
\begin{lem}\label{RR} Given $\overline{u}$ as with asymptotics \ref{ncat0}, \ref{ncatinfty} and \ref{cont}, one has
$$-d=\sum\limits^{n^-_{0}}_{i=1} k^{0,-}_i+\sum\limits^{n^-_{\infty}}_{i=1} k^{\infty,-}_i+p\sum\limits^{n_c^-}_{i=1}k^-_i-
 \sum\limits^{n^+_{0}}_{i=1} k^{0,+}_i-\sum\limits^{n^+_{\infty}}_{i=1} k^{\infty,+}_i-p\sum\limits^{n_c^+}_{i=1}k^+_i$$
where $d$ is the degree of $c$. 
\end{lem}
\begin{proof}
 Let $u$ be a lift of $\overline{u}$ as in (\ref{thelift}). Then it has the asymptotics \ref{la1}-\ref{la3}. Viewing $u$ as a meromorphic section of the bundle
 $c^*L$ leads to the desired equation since the left hand side of the is the degree of the bundle $c^*L$ and the right hand side
 is the degree of the divisor of the section $u$.
\end{proof}
\begin{lem}\label{homotopy}Given $\overline{u}$ as with asymptotics \ref{ncat0}, \ref{ncatinfty} and \ref{cont}, one has
$$\sum\limits^{n^-_{0}}_{i=1} k^{0,-}_i+q\sum\limits^{n^-_{\infty}}_{i=1} k^{\infty,-}_i-
 \sum\limits^{n^+_{0}}_{i=1} k^{0,+}_i-q\sum\limits^{n^+_{\infty}}_{i=1} k^{\infty,+}_i\equiv 0.$$ 
\end{lem}
\begin{proof}
 Since $\overline{u}$ is a rational curve, the sum of the homotopy classes of positive ends is equal to 
 the sum of the homotopy classes of negative ends. Writing all homotopy classes in terms of $[\overline{\gamma}_0]$,
 leads to the above equation since $[\overline{\gamma}_{\infty}]=q[\overline{\gamma}_0]$ in $\pi_1$.
\end{proof}

We now have a closer look at the \textit{singular base curve} 
$$\overline{c}:=\overline{\pi}\circ\overline{u}: \mathbb{C}P^{1}\setminus \Gamma \rightarrow \overline{\mathbb{C}P^1},$$
where $\overline{\pi}:\mathbb{R}\times L(p,q)\rightarrow \overline{\mathbb{C}P^1}$ is the projection along the Reeb orbits. 
Using the extension of $c$, we get the orbicurve 
$\overline{c}: \mathbb{C}P^{1}\rightarrow \overline{\mathbb{C}P^1}$. 
We remove singularities $\{\overline{0},\overline{\infty}\}\subset \overline{\mathbb{C}P^1}$ and consider the map 
\begin{eqnarray}\label{cbarnonsingular}
\overline{c}:\mathbb{C}P^1\setminus P\rightarrow \overline{\mathbb{C}P^1}\setminus\{\overline{0},\overline{\infty}\},\; P:=\overline{c}^{-1}(\{\overline{0},\overline{\infty}\})
\end{eqnarray}
which is a holomorphic 
branched covering. One can biholomorphically identify 
$\overline{\mathbb{C}P^1}\setminus\{\overline{0},\overline{\infty}\}$ with $\mathbb{C}^*$ and 
$\overline{c}$ can be viewed as a holomorphic
branched covering of $\mathbb{C}^*$. Extending over $P$, 
we get a holomorphic branched covering
\begin{eqnarray}\label{chat}
\hat{c}:\mathbb{C}P^1\rightarrow \mathbb{C}P^1,
\end{eqnarray}
which we call as the \textit{smoothened base curve} of $\overline{u}$. We note that $P=\hat{c}^{-1}(\{0,\infty\})$ and $P$ is in general larger then $\Gamma_{nc}$. 
We put $\tilde{P}:=c^{-1}(\{0,\infty\})=\mathfrak{p}^{-1}(P)\subset \Sigma$. 
\begin{lem}
 The degrees of the maps (\ref{cbarnonsingular}), (\ref{chat}) and (\ref{liftedbasecurve}) coincide. 
\end{lem}
\begin{proof} \label{degree}
It is clear that degrees of $\overline{c}$ and $\hat{c}$ coincide. The last part of the 
statement follows since $\overline{c}\circ\mathfrak{p}=\texttt{p}_\flat\circ c$ and both $\mathfrak{p}$ and $\texttt{p}_\flat$ are 
$p:1$ coverings.
\end{proof}
\begin{lem}\label{ramification} The ramification profiles of $\hat{c}$ and $c$ satisfies the followings.
\begin{enumerate}[label=\textrm{(\arabic*)}]
    \item For all $i$, the ramification number of $\hat{c}$ at $z_i^{0/\infty,\pm}\in\Gamma_{nc}$ coincides with the ramification number $r_i^{0/\infty,\pm}$ of $c$ at $z_i^{0/\infty,\pm}\in \tilde{\Gamma}_{nc}$ (see Lemma \ref{localrem}).
    \item $|\tilde{P}\setminus\tilde{\Gamma}_{nc}|=p|P\setminus\Gamma_{nc}|$.
    \item For any $z\in P\setminus \Gamma_{nc}$, the ramification numbers of $c$ at $p$-many preimages of $z$ are all the same.
    \item For any $z\in P\setminus \Gamma_{nc}$, the ramification number of $\hat{c}$ at $z$ is $p$ times the ramification number of $c$ at any preimage of $z$.
\end{enumerate}
 
\end{lem}
\begin{proof} 
We note that the ramification number of $\hat{c}$ at any $z\in P$ corresponds to the
local covering number of (\ref{cbarnonsingular})
around $z$. Hence the remaining statements follow from the local description of the equation 
$\overline{c}\circ\mathfrak{p}=\texttt{p}_\flat\circ c$.  
\end{proof}

\subsection{Equivariant curves: the sufficient conditions for the existence}
In this section, we discuss the sufficient conditions for the existence of equivariant meromorphic sections of $L$. It turns out that when $n_{nc}=2$, due to Lemma \ref{genuslemma} the necessary conditions given above are also sufficient. But if $n_{nc}\geq 3$, then there may be an a priori obstruction due to the genus of $\Sigma$. We postpone the treatment of the first case to the next section and concerning the second case, we discuss the problem for $n_{nc}=3$ for the sake of presentation. We further assume $n_c=0$ since contractible ends do not essentially change the problem and we can omit the ambiguity between parametrized or unparametrized  curves, see Remark \ref{paramunparam}. 
At the end of the section, we point out the necessary
modifications for more general types of curves.
 
Let $\mathcal{M}$ denote the moduli space of $J_\alpha$- holomorphic curves 
\begin{eqnarray}\label{mathcalM}
\overline{u}:\mathbb{C}P^1\setminus \Gamma \rightarrow \overline{L}^*,\;\Gamma=\{z^{0,+},z^{\infty,+},z^{0,-} \}
\end{eqnarray}
with the asymptotics 
\begin{itemize}
\item $\overline{u}(z^{0,+})=(+\infty, k^{0,+}\overline{\gamma}_0)$,
\item $\overline{u}(z^{\infty,+})=(+\infty, k^{\infty,+}\overline{\gamma}_{\infty})$,
\item $\overline{u}(z^{0,-})=(-\infty, k^{0,-}\overline{\gamma}_0)$
\end{itemize}
where the necessary conditions given by Lemma \ref{homotopy} and Lemma \ref{RR} are satisfied, namely 
\begin{eqnarray}\label{nec}
d:=k^{0,+}+k^{\infty,+}-k^{0,-}>0,\; k^{0,+}+qk^{\infty,+}-k^{0,-}\equiv 0. 
\end{eqnarray}

Now given the above data, we define $\mathcal{C}_\mathcal{M}$ to be
the moduli space of curves 
\begin{eqnarray}\label{C_M}
\hat{c}:\mathbb{C}P^1\rightarrow \mathbb{C}P^1; \;\hat{c}(z^{0,+})=0,\;\hat{c}(z^{\infty,+})=\infty,\;\hat{c}(z^{0,-})=0
\end{eqnarray}
such that 
\begin{enumerate}[label=\textrm{(cm\arabic*)}]
 \item \label{CM1}
 The degree of $\hat{c}$ is $d$.
 \item \label{CM2}
 The ramification numbers of $\hat{c}$ satisfies \begin{itemize}
     \item at $z^{0,+}$: $r^{0,+}\equiv k^{0,+} (1-v)$
     \item at $z^{0,-}$: $r^{0,-}\equiv k^{0,-}(v-1)$
     \item st $z^{\infty,+}$: $r^{\infty,+}\equiv k^{\infty,+}(1-q)$
     \item for any $z\in P\setminus \Gamma$, the ramification number at $z$ is divisible by $p$, where $P:=\hat{c}^{-1}\{0,\infty\}$.
 \end{itemize}
\end{enumerate}
Given the moduli problem $\mathcal{M}$, we fix the covering $(\Sigma,\tilde{\Gamma},\mathfrak{p})$, which exists even if $\mathcal{M}$ is empty, see Remark (\ref{dependence}). 
We note that for each $\overline{u}\in \mathcal{M}$, 
there corresponds a smoothened base curve $\hat{c}$ and due to Lemma \ref{localrem}, Lemma \ref{degree} and 
 Lemma \ref{ramification} we know that $\hat{c}\in\mathcal{C}_\mathcal{M}$. Now the question is determine which curves in $\mathcal{C}_\mathcal{M}$ provide a curve in $\mathcal{M}$. Hence, we need to reverse the procedure given in the previous section. 
Now given $\hat{c}\in \mathcal{C}_\mathcal{M}$, it corresponds to a non-singular branched covering
$$\overline{c}:\mathbb{C}P^1\setminus P\rightarrow 
\overline{\mathbb{C}P^1}\setminus \{\overline{0},\overline{\infty}\}$$ 
where $P:= \overline{c}^{-1}(\{\overline{0},\overline{\infty}\})$.
We first need to construct the lifted base curve $c$. Namely, we should check the diagram in Figure \ref{commute2} is valid.
\begin{figure}
    \centering
\begin{tikzpicture}[>=triangle 60,]
 \matrix[matrix of math nodes,column sep={60pt,between origins},row sep={60pt,between origins},nodes={rectangle}] (s) 
 {
|[name=A]| \Sigma\setminus \tilde{P} &&    
|[name=D]| \mathbb{C}P^1\setminus\{0,\infty\} \\ 
|[name=B]| \mathbb{C}P^{1}\setminus\overline{c}^{-1}(\{\overline{0},\overline{\infty}\}) && 
|[name=E]| \overline{\mathbb{C}P^1}  \setminus \{\overline{0},\overline{\infty}\}\\
 };
\draw[->] (A) edge node[auto] {\(c\)} (D)
(A) edge node[left] {\(\mathfrak {p}\)} (B)
(B) edge node[below] {\(\overline{c}\)} (E)
(D) edge node[auto] {\(\texttt{p}_\flat\)} (E)
;
\end{tikzpicture}
    \caption{Lifting diagram for $\overline{c}$. }
    \label{commute2}
\end{figure}

\begin{lem}
 For any  $\hat{c}\in \mathcal{C}_\mathcal{M}$, the corresponding (non-singular) branched covering 
 $\overline{c}$
 lifts through $$\mathfrak{p}:\Sigma\setminus \tilde{P}\rightarrow \mathbb{C}P^1\setminus P$$ 
 where $ \tilde{P}:=\mathfrak{p}^{-1}(P)$.  
\end{lem}
\begin{proof}
As in the previous section, we fix $z\in \mathbb{C}P^1\setminus P$
and put $w:=\overline{c}(z)$. We have the induced homomorphism
$$\overline{c}_*: \pi_1(\mathbb{C}P^1\setminus P, z)\rightarrow 
\pi_1(\overline{\mathbb{C}P^1}\setminus \{\overline{0},\overline{\infty}\},w)\cong\mathbb{Z},$$
where we fix the generator $\eta:=\{we^{it}:t\in [0,\theta]\}$ for the latter group. 
The covering 
$\texttt{p}_\flat:\mathbb{C}P^1\setminus\{0,\infty\}\rightarrow \overline{\mathbb{C}P^1}\setminus \{\overline{0},\overline{\infty}\}$ induces the monomorphism
$$\mathbb{Z}\rightarrow \mathbb{Z},\; 1\rightarrow p$$
on the fundamental group
where we fix a generator of the former group as a lift of $p\eta$. Let   
$$\rho: \mathbb{Z}\rightarrow \mathbb{Z}/ {\rm im }\:(\texttt{p}_\flat)_*= {\rm coker}(\texttt{p}_\flat)_*\cong \mathbb{Z}_p=\langle[\eta]\rangle$$
denote the the quotient homomorphism. Then 
$\overline{c}\circ\mathfrak{p}$ lifts if and only if 
\begin{eqnarray}\label{needtrivial1}
\rho\circ \overline{c}_*\circ \mathfrak{p}_*: \pi_1(\Sigma\setminus\tilde{P},\tilde{z})\rightarrow \mathbb{Z}_p
\end{eqnarray}
is trivial for some $\tilde{z}\in \mathfrak{p}^{-1}(z)$. 
We note that by the last statement of \ref{CM2} and the fact that $\mathbb{Z}_p$ is abelian, we have the following commutative diagram
\begin{center}
 \begin{tikzcd}
   \pi_1(\mathbb{C}P^1\setminus P, z) \arrow{r}{\imath_*} \arrow[swap]{dr}{\rho\circ \overline{c}_*} & \pi_1(\mathbb{C}P^1\setminus \Gamma, z)
   \arrow{d}{\rho\circ \overline{c}_*} \\
     & \mathbb{Z}_p
  \end{tikzcd}
\end{center}
where the upper horizontal arrow is induced by the inclusion $\imath:\mathbb{C}P^1\setminus P\hookrightarrow \mathbb{C}P^1\setminus \Gamma$. Combining this with the with the commutative diagram induced by
\begin{center}
\begin{tikzcd}
\Sigma\setminus \tilde{P} \arrow{r}{\tilde{\imath}} \arrow[swap]{d}{\mathfrak{p}} & \Sigma\setminus \tilde{\Gamma} \arrow{d}{\mathfrak{p}} \\
\mathbb{C}P^1\setminus P  \arrow{r}{\imath} & \mathbb{C}P^1\setminus \Gamma
\end{tikzcd}
\end{center}
we conclude that if 
\begin{eqnarray}\label{needtrivial2}
\rho\circ \overline{c}_*\circ \mathfrak{p}_*: \pi_1(\Sigma\setminus\tilde{\Gamma},\tilde{z})\rightarrow \mathbb{Z}_p
\end{eqnarray}
vanishes then (\ref{needtrivial1}) vanishes as well. 
Once we fix an isomorphism $\varphi: \pi_1(\overline{L}^*)\rightarrow {\rm coker}\: (\texttt{p}_\flat)_*\cong \mathbb{Z}_p$ such that 
$\varphi ([\overline{\gamma_{0}}])=[\eta]$, then it is not hard to see that by \ref{CM2}, $\rho\circ \overline{c}_*$ coincides with the homomorphism $\pi_1(\mathbb{C}P^1\setminus \Gamma,z)\rightarrow \pi_1(\overline{L}^*)$ determined by $\mathcal{M}$ (see Remark \ref{dependence}) up to multiplication by a $(1-q)$. But the kernel of the latter homomorphism is precisely the image of $\mathfrak{p}_*$. Hence (\ref{needtrivial2}) is trivial.   
\end{proof}
\begin{rmk}\label{basemanylifts}
 As in Remark \ref{manylifts}, there are $p$- many lifts of given $\overline{c}$. Once we fix a 
 lift $c$, the other lifts are given by $\sigma_\flat^k \circ c$, $k=1,...,p-1$.
\end{rmk}
Given $\hat{c}\in \mathcal{C}_\mathcal{M}$, we consider the extension $c:\Sigma\rightarrow \mathbb{C}P^1$ of a lift 
of $\overline{u}$ such that 
$c\circ\tau_\flat=\sigma_\flat\circ c$, where $\tau_\flat$ is a fixed generator of the group $G$ acting on $\Sigma$. 
The question is now to determine whether
there is a meromorphic section $u: \Sigma \rightarrow c^*L$ such that $u\circ\tau_\flat=\tau\circ u$ where $\tau$ is the corresponding generator of the $\mathbb{Z}_p$ action on $c^*L$ and  zeros and poles of $u$ are given by the moduli problem $\mathcal{M}$. Namely we ask for a pole at $z^{0,+}$ of order $k^{0,+}$, a pole at 
$z^{\infty,+}$ of order $k^{\infty,+}$ and a zero at $z^{0,-}$ of order $k^{0,-}$ 
(as in the previous section, $z^{i,\pm}$ denotes the punctures in $\tilde{\Gamma}_{nc}$ 
corresponding to punctures in $\Gamma_{nc}$). 
\begin{lem} Assume that there is meromorphic section $u$ of $c^*L$ with 
zeros and poles determined by $\mathcal{M}$ and $c$ is equivariant. Then $u$ is also equivariant.
\end{lem}
\begin{proof} 
Given fixed zeros and poles, one has a $\mathbb{C}^*$-family of meromorphic sections given
by $\lambda u$, $\lambda\in \mathbb{C}^*$. Note that since the scaling along fibres commutes with $\tau$-
action, if there is some equivariant meromorphic section $u$, then all meromorphic sections of
the form $\lambda u$ are equivariant.

Let $u$ be a meromorphic section of $c^*L$ with fixed zeros and poles, where $c$ is equivariant. 
Then $u':=\tau^{-1}\circ u\circ \tau_\flat$ is also a meromorphic section of $c^*L$ having same zeros 
and poles with $u$. 
Hence there is some $\lambda\in \mathbb{C}^*$ such that $u'=\lambda u$. In fact, $\lambda\in S^1$ since all the actions we have are unitary.  
Both $u$ and $u'$ has a pole at $z^{\infty,+}$ of order $k^{\infty,0}$. 
We know that the ramification number of $c$ at $z^{\infty,+}$ coincides with $\hat{c}$ and satisfies
\begin{eqnarray}\label{13}
 k^{\infty,+}\equiv r^{\infty,+}(1-q)^{-1}.
\end{eqnarray}
by the definition of $\mathcal{C}_\mathcal{M}$.
We take the local trivialization of $c^*L$ around $z^{\infty,+}\in \Sigma$ as in the previous section  
(see Lemma \ref{localrem}) so that 
around $z^{\infty,+}$, we have $\tau(z,\lambda)=(e^{im\theta}z,e^{i\theta}\lambda)$ and $\tau_\flat(z)=e^{im\theta}z$ where 
\begin{eqnarray}\label{14}
 r^{\infty,+}m\equiv (q-1).
\end{eqnarray}
With respect to the above trivialization, we write $u(z)=(z,f(z))$ and $u'(z)=(z,f'(z))$ where $f$ and $f'$ are  meromorphic functions.  With polar coordinates $z=\rho e^{it}$, we have
$$\lim_{\rho\rightarrow 0}\frac{f(\rho e^{it})}{|f(\rho e^{it})|}=e^{-ik^{\infty,+}t},\;
\lim_{\rho\rightarrow 0}\frac{f'(\rho e^{it})}{|f'(\rho e^{it})|}=\lambda e^{-ik^{\infty,+}t}.$$
On the other hand, definition of $u'$ gives 
\begin{eqnarray*}
 \lim_{\rho\rightarrow 0}\frac{f'(\rho e^{it})}{|f'(\rho e^{it})|}
 &=&\lim_{\rho\rightarrow 0}\frac{\tau^{-1} \circ f(\tau_\flat(\rho e^{it}))}{|\tau^{-1} \circ f(\tau_\flat(\rho e^{it}))|}\\
 &=&\lim_{\rho\rightarrow 0}\frac{e^{-i\theta} f(\rho e^{it+im\theta})}{|f(\rho e^{it+im\theta})|}\\
 &=& e^{-i\theta}(e^{-ik^{\infty,+}(t+m\theta)})\\
 &=& e^{-ik^{\infty,+}t}e^{-i(k^{\infty,+}m+1)\theta}\\
 &=& e^{-ik^{\infty,+}t}
\end{eqnarray*}
where the last equation follows from (\ref{13}) and (\ref{14}).  Hence $\lambda=1$ and therefore $\tau \circ u=u\circ \tau_\flat$.
\end{proof}
We want to understand the moduli space $\mathcal{C}_\mathcal{M}$ and possible obstructions on the existence of meromorphic sections over the lifts of singular curves, which correspond to the elements of $\mathcal{C}_\mathcal{M}$. To this end we recall some basic notions in the theory of Riemann surfaces.

Given a closed Riemann surface $\Sigma$ with genus $g$ and $[c]\in H_1(X,\mathbb{Z})$, there is a functional on the (vector) space $\Omega^1(\Sigma)$ consisting of holomorphic $1$-forms on $\Sigma$ ,
$$\int_{[c]}:\Omega^1(\Sigma)\rightarrow \mathbb{C},\:\: \omega\mapsto \int_{c}\omega,$$
which is well-defined since any holomorphic 1-form on $\Sigma$ is necessarily closed. An element of $\Omega^1(\Sigma)^*$ is called a \textit{period} if it is of the above form. The  \textit{Jacobian} of $\Sigma$ is the the quotient space
$$J(\Sigma):=\Omega^1(\Sigma)^*/\Lambda,$$
where $\Lambda$ is the space of periods. It turns out that $J(\Sigma)$ is isomorphic to the complex torus of dimension $g$.

Let $z_0$ be a point in $\Sigma$. We consider the map 
$$A:\Sigma\rightarrow \Omega^1(\Sigma)^*;\; A(z)(\omega):=\int_{\gamma_z}\omega,\; \omega\in \Omega^1(\Sigma),$$
where $\gamma_z$ is some path connecting $z_0$ to $z$.  Although this map depends on $\gamma_z$, 
it descends to so called the \textit{Abel map}  $A:\Sigma\rightarrow J(\Sigma)$ which is independent of the choice of $\gamma_z$. 
Note that the Abel map naturally extends to a group homomorphism over the group $Div(\Sigma)$
of divisors via $A(\sum n_zz):=\sum n_zA(z)$. It turns out that once restricted to subgroup $Div_0(\Sigma)=\ker {\rm deg}$, where ${\rm deg}:Div(\Sigma)\rightarrow\mathbb{Z}$ is the degree map, the Abel map is independent of the point $z_0$. 

$Div_0(\Sigma)$ has a special subgroup $PDiv(\Sigma)$ consisting of \textit{principal divisors}, namely the divisors given by meromorphic functions. Given a meromorphic function $f$ on $\Sigma$, then its divisor is given by 
$D(f)=D_0(f)-D_\infty(f)$ where  
$$D_0(f):=\sum_{f(z)=0}o_z(f)z,\;D_\infty(f):=\sum_{f(z)=\infty}o_z(f)z$$
and $o_z(f)>0$ stands for the order of zeros/poles. 
It turns out that the map
\begin{eqnarray}\label{abel}
A:Div_0(\Sigma)\rightarrow J(\Sigma)
\end{eqnarray}
is surjective and its kernel is given by $PDiv(\Sigma)$. Hence one gets the isomorphism 
$$Pic(\Sigma)=Div_0(\Sigma)/PDiv(\Sigma)\cong J(\Sigma),$$
where the \textit{Picard group} $Pic(\Sigma)$ is the group of isomorphism classes of of degree zero line bundles 
over $\Sigma$. Moreover if $A([D])=0$ for some $[D]\in Pic(\Sigma)$, then the divisor $D$ defines the trivial bundle and we have the following consequence. If $D_1$ and $D_2$ are two divisors, which define a line bundle of the same degree, then these bundles are isomorphic if and only if $A(D_1-D_2)=0$ and this is trivially the case if $D_1-D_2=0$.  We let $\mathfrak{S}^n(X)$ denote the $n$-fold symmetric product of the set $X$.

Setting the ground, we first give a description of $\mathcal{C}_\mathcal{M}$, which is adopted from \cite{oldie}.
\begin{lem}\label{moduliCM}
 $\mathcal{C}_\mathcal{M}$ is biholomorphic to 
 $$\left((\mathfrak{S}^{n_0}(\mathbb{C}P^1\setminus\{z^{\infty,+}\})
 \times \mathfrak{S}^{n_\infty}(\mathbb{C}P^1\setminus\{z^{0,-},z^{0,+}\}))\setminus \Delta\right)\times \mathbb{C}^*,$$
 where $\Delta$ is the subset of pairs 
 $(D_0,D_\infty)$ in $\mathfrak{S}^{n_0}(\mathbb{C}P^1\setminus\{z^{\infty,+}\})\times \mathfrak{S}^{n_\infty}(\mathbb{C}P^1\setminus\{z^{0,-},z^{0,+}\})$
 with at least one common point and
$$n_0:=\frac{d-\overline{r}^{0,+}-\overline{r}^{0,-}}{p},\;n_\infty:=\frac{d-\overline{r}^{\infty,+}}{p}.$$
\end{lem}
\begin{proof}
The moduli space $\mathcal{C}_\mathcal{M}$ itself can be seen as the space of 
meromorphic functions on $\mathbb{C}P^1$, which has a particular distribution of its zeros and poles. We note that in this case the Abel map vanishes identically, that is any degree zero divisor defines a $\mathbb{C}^*$-family of meromorphic functions. Hence it is enough to characterize the the set of divisors satisfying the conditions \ref{CM1} and \ref{CM1}. 

The condition \ref{CM1} tells us that 
${\rm deg}(D_0(\hat{c}))={\rm deg}(D_\infty(\hat{c}))=d$. Namely we have complex $d$-dimensional freedom to choose the zeros or poles. The conditions given by \ref{CM2} translate as follows. We two zeros $z^{0,-}$, $z^{0,+}$  and a pole $z^{\infty,+}$ whose orders, which corresponds to the ramification numbers, are fixed in$\mod p$. We fix
We fix  
$$\overline{r}^{0,+}, \overline{r}^{\infty,-}, \overline{r}^{0,+}\in\{1,...,p-1\}$$ such that
$$\overline{r}^{0,+}\equiv r^{0,+},\; \overline{r}^{\infty,+}\equiv r^{\infty,+},\;\overline{r}^{0,-}\equiv r^{0,-}.$$
Now we can perturb away $( r^{0,+}-\overline{r}^{0,+})$-many zeros at $z^{0,+}$ but any zero different than $z^{0,\pm}$ has to have order divisible by $p$. The same reasoning applies to $z^{0,+}$ and $z^{\infty,+}$. Hence for zeros we have 
$(d-\overline{r}^{0,+}-\overline{r}^{0,-})/p$
many choices among the points in $\mathbb{C}P^1\setminus\{z^{\infty,+}\}$ and for poles we have $(d-\overline{r}^{\infty,+})/p$
many choices among the points in $\mathbb{C}P^1\setminus\{z^{0,+},z^{0,-}\}$. We note that $d\equiv\overline{r}^{0,+}+\overline{r}^{0,-}\equiv \overline{r}^{\infty,+}$. After taking suitable symmetric products of our domain and removing collections of points which do not not result in divisors, we get the above description.
\end{proof}
We note that  
 $$\mathfrak{S}^{n_0}((\mathbb{C}P^1\setminus\{z^{\infty,+}\})
 \times \mathfrak{S}^{n_\infty}((\mathbb{C}P^1\setminus\{z^{0,-},z^{0,+}\})$$
 is a complex manifold and  $\Delta$ is an irreducible subvariety of (complex) codimension one \cite{oldie}. In particular, the dimension of $\mathcal{C}_\mathcal{M}$ is given by 
\begin{eqnarray}
 \dim_\mathbb{R} \mathcal{C}_\mathcal{M}=2+ 2n_0+2n_\infty=2+\frac{2}{p}(2d-\overline{r}^{0,+}-\overline{r}^{0,-}-\overline{r}^{\infty,+}).
\end{eqnarray}

Now the question is that given $\hat{c}\in\mathcal{C}_\mathcal{M}$, does $c^*L$  admit a 
meromorphic section with zeros and poles determined by $\mathcal{M}$ for some lift $c$ of corresponding 
$\overline{c}$. It turns out that by the very nature of the equivariant picture there is no obstruction to the existence of such meromorhic sections. 
\begin{prop}\label{sectionsexist} Given $\hat{c}\in \mathcal{C}_{\mathcal{M}}$, there exists a meromorphic section of $c^*L$, which leads to a punctured curve in $\mathcal{M}$ where $c$ is a lift of the singular curve $\overline{c}$ corresponding to $\hat{c}$.  
\end{prop}
\begin{proof}
We define the divisor $D_\mathcal{M}$ of $\Sigma$ by
$$D_\mathcal{M}=k^{0,-}z^{0,-}-k^{0,+}z^{0,+}-k^{\infty,+}z^{\infty,+}.$$
We note that the isomorphism class of the line bundle $L$ is determined by the isomorphism class of divisors on $\mathbb{C}P^1$, whose degree is $-1$. We fix a divisor in this class of the form 
$$D_L:=l\cdot0-(l+1)\cdot\infty,\;\; l>0$$
where $l$ is to be chosen. Since the phase shift on $\mathcal{C}_\mathcal{M}$ is not relevant for our problem we ignore it in what follows. Then lifting scheme above provides the following continuous embedding 
\begin{eqnarray}\label{liftembed}
\daleth: \mathcal{C}_\mathcal{M}\rightarrow Div^0(\Sigma),\;\;\hat{c}\mapsto D_\mathcal{M}-c^*D_L
\end{eqnarray}
where $c$ is (up to phase shift) the unique lift of $\hat{c}$ corresponding to $\hat{c}$ satisfying 
$$c\circ\tau_\flat=\sigma_\flat\circ c$$ 
and $c^*D_L$ is the pull-back divisor. We note that ${\rm deg} (D_\mathcal{M})={\rm deg} (c^*D_L)=d$ so that the map is well-defined. 
Now we want to show that there exists some $l\geq 0$ such that $A\circ\daleth(\hat{c})=0$ for a curve $\hat{c}$ satisfying 
\begin{eqnarray}\label{abelvanishes}
D_0(\hat{c})=r^{0,+}z^{0,+}+r^{0,-}z^{0,-},\;\;D_\infty(\hat{c})=r^{\infty,+}z^{\infty,+}.
\end{eqnarray}
In this case, 
$$D_\mathcal{M}-c^*D_L=(k^{0,-}-lr^{0,-})z^{0,-}+(-k^{0,+}-lr^{0,+})z^{0,+}+((l+1)r^{\infty,+}-k^{\infty,+})z^{\infty,+}.$$
First we want to find $l$ such that all the coefficients above are zero in$\mod p$ and by \ref{CM2}, we see that this holds for 
$l\equiv (1-v)^{-1}$. We choose some $l>0$ such that $l\equiv (1-v)^{-1}$. Then 
$$D_\mathcal{M}-c^*D_L=p\left(m^{0,-}z^{0,-}+m^{0,+}z^{0,+}+m^{\infty,+}z^{\infty,+}\right)$$
for some $m^{0/\infty,\pm}\in \mathbb{Z}$. We note that the divisor $m^{0,-}z^{0,-}+m^{0,+}z^{0,+}+m^{\infty,+}z^{\infty,+}$ descends to a divisor on $\mathbb{C}P^1$ where $z^{0/\infty,\pm}$ stands for the images of $z^{0/\infty,\pm}$ under $\mathfrak{p}$. Moreover this divisor has degree 0. Then we know that there is a meromorphic function, say $f$ on $\mathbb{C}P^1$ which realizes this divisor. Then $D_\mathcal{M}-c^*D_L$ is precisely the divisor for the meromorphic function $f\circ\mathfrak{p}$ on $\Sigma$. Hence $A(D_\mathcal{M}-c^*D_L)=0$.
\end{proof}
\begin{cor}\label{generaldim}
We have $\mathcal{M}\cong \mathcal{C}_\mathcal{M}\times \mathbb{C}^*$ and in particular,
 $$\dim_\mathbb{R}\mathcal{M}=4+\frac{2}{p}(2d-\overline{r}^{0,+}-\overline{r}^{0,-}-\overline{r}^{\infty,+}).$$
\end{cor}
\begin{rmk}\label{dimension}
Note that given the moduli problem (\ref{mathcalM}), one reads of the degree $d$ and the quantities $\overline{r}^{0/\infty,\pm}$ from the multiplicities and immediately gets the dimension of the moduli space. Then one can check the regularity of the almost complex structure $J_\alpha$ immediately by comparing the above dimension with the virtual dimension of the moduli space. Such a comparison will be carried out for pair of pants in the next section and the arguments used for pair of pants case immediately generalizes to other configurations. In fact we claim that $J_\alpha$ is regular for any admissible moduli problem. 
\end{rmk}
\section{Computations: pair of pants, cylinders and others}
As we saw above, if $n_{nc}=2$ then one has $g=0$. In this case the covering $\mathfrak{p}$ can be studied more explicitly and one can study all possible lifted base curves, which lead to the different non-empty components of the moduli space. In this section, we study the pair of pants with two positive non-contractible ends in detail and comment on other kinds of moduli problems with $n_{nc}=2$ as well. 

After determining the non-empty moduli spaces of pair of pants, we compute the dimensions of these moduli spaces in terms of the dimensions of equivariant
moduli spaces in the lift, see Remark \ref{dimension}, and compare them with the virtual dimension of these moduli spaces given by the well-known index formula (\ref{Fredholmindex}). 
The observation is that the dimension of the equivariant moduli space
coincides with the index of the problem and this establishes the regularity 
of the almost complex structure $J_\alpha$, see Remark \ref{dimension}. 
\begin{center}
{\emph{The moduli space of pair of pants}}
\end{center}

We consider the following moduli problem $\mathcal{M}$ of $J_\alpha$- holomorphic curves
\begin{eqnarray}\label{popinlens}
 \overline{u}:\mathbb{C}P^1\setminus\{0,1,\infty\}\rightarrow \mathbb{R}\times L(p,q)
\end{eqnarray}
with asymptotics
\begin{eqnarray}\label{asypmofpopinlens}
 \overline{u}(0)=(+\infty,k^0\overline{\gamma}_0),\;\overline{u}(\infty)=(+\infty,k^\infty\overline{\gamma}_{\infty}),\;
\overline{u}(1)\in\{-\infty\}\times S_k
\end{eqnarray}
where $k^0,k^\infty\not \equiv 0$ and $S_k$ denotes the orbit space of contractible orbits of action $2\pi k$. 

In order to work with simpler terms, we choose a model for the domain of (\ref{popinlens}) as follows. We consider the $G\cong \mathbb{Z}_p$-action on $\mathbb{C}P^1$ given by 
\begin{eqnarray}\label{tauflatatpop}
\tau_\flat((z:1))=(e^{im\theta}z:1)
\end{eqnarray}
where $m\in \{1,...,p-1\}$ and $\theta=2\pi/p$. The quotient map leads to the covering map
$$\mathfrak{p}:\mathbb{C}P^1\setminus\{0,1,w_1,...,w_{p-1},\infty\}\rightarrow \overline{\mathbb{C}P^1}\setminus \{\overline{0},\overline{1},\overline{\infty}\}.$$
We identify the quotient space above with the our domain $\mathbb{C}P^1\setminus\{0,1,\infty\}$.

Now we need to determine lifted base curves, namely the  equivariant holomorphic maps 
$$c:\mathbb{C}P^1\rightarrow \mathbb{C}P^1 $$
where the $\langle\sigma_\flat\rangle$-action on the range is given by
$$\sigma_\flat((z_1:z_2))=(e^{i(1-q)\theta}z_1:z_2).$$ 
Once we parametrize the domain and the range via
 $z\mapsto (z:1)$,   any non-trivial holomorphic map $c$ is given by $c(z)=\lambda g(z)/h(z)$
where $\lambda\in \mathbb{C^*}$ and $g$ and $h$ are monic polynomials without a common root. Imposing the equivariance leads to the following characterization.
\begin{lem}\label{cforlens}
A non-trivial holomorphic map $c$ satisfies $c\circ\tau_\flat=\sigma_\flat\circ c$ if and only if it has the following form
\begin{eqnarray} 
c(z)=\lambda z^{r} g(z)/h(z), 
\end{eqnarray}
where
$$mr\equiv 1-q$$
and 
$$g(z)=\prod_{s=1}^{n}(z^p-a_s)^{k_s},\;h(z)=\prod_{t=1}^{m} (z^p-b_t)^{l_t}$$
such that 
$\lambda\in \mathbb{C^*}$, $r\in \{\mp 1,\mp 2,...,\mp (p-1)\}$, $k_s,l_t\in \mathbb{N}^+$ and 
$a^j_s,b^j_t\in \mathbb{C}$ such that 

\end{lem}
\begin{proof}
It is clear that such a map is equivariant. For the other direction, one sees that 
if $c$ admits $w\in \mathbb{C}P^1\setminus\{0,\infty\}$ 
as a zero or pole then it must admit $p-1$ many distinct zeros/poles which are given by the orbit of $w$. 
Hence these polynomials 
have to factorize through terms like $(z^p-a_s)^{k_s}$ and $(z^p-b_t)^{l_t}$ where $k_s$ and $l_t$ are any non zero complex numbers.
The only thing that requires attention then is the case where we have zero or infinity as zero or pole. Checking the equivariance 
one gets the relation above between $m$ and $r$.      
\end{proof}
Now we restate the equivariance conditions given by Lemma\ref{localrem} in this special case. By (\ref{asypmofpopinlens}) we have $c(0)=0$ and $c(\infty)=\infty$ and by (\ref{trivat0}) and (\ref{trivatinfty}), an equivariant section $u$ locally looks like $(z,f(z))$ where
$$f(e^{im\theta}z)=e^{iq\theta}\; \textrm{ around } 0$$
and 
$$f(e^{-im\theta}z)=e^{i\theta}\; \textrm{ around } \infty.$$
Now if $f(z)=\sum a_n z^n$ around 0, we have 
$$a_n\not =0 \Rightarrow n\equiv m^{-1}q.$$ 
Similarly, if $f(z)=\sum b_n z^n$ around $\infty$, we have 
$$b_n\not =0 \Rightarrow n\equiv -m^{-1}.$$ 
Since we require that $u$ has  poles of order $k^0$ at $0$  and of order $k^\infty$ at $\infty$,  we get
$$-k^0\equiv m^{-1}q \; \textrm{and}\; -k^\infty\equiv -m^{-1}.$$
This means that 
\begin{eqnarray}\label{multiplicitiesforlens}
 k^0\equiv -m^{-1}q\equiv rq(q-1)^{-1}  \textrm{ and } k^\infty\equiv m^{-1}\equiv r(1-q)^{-1}.
\end{eqnarray}
Since $c(0)=0$ and $c(\infty)=\infty$, we have $r>0$ and $r+p\sum k_s>\sum l_t$, so that the degree of $c$ is given by ${\rm deg}(c)=r+p\sum k_s$. We note that
\begin{equation}\label{multipanddegreeinmod}
-k^0-k^\infty\equiv m^{-1}q-m^{-1}\equiv -m^{-1}(1-q)\equiv -r\equiv -{\rm deg}(c)\equiv {\rm deg}(c^*L).    
\end{equation}
We also require that $u$ has $p$ many zeros of order $k$. Hence for suitable choice of $k$, we have 
\begin{equation}\label{multipanddegree}
pk-k^0-k^\infty=-r-p\sum k_s={\rm deg}(c^*L)
\end{equation}
so that the divisor of $u$ has the correct degree.  
Therefore, in the lift  we have a curve with two positive ends asymptotic to $k^0\gamma_{0}$ and $k^\infty\gamma_{\infty}$ and $p$ many
negative ends with multiplicity $k$. Passing to the quotient,
we get pair of pants in the moduli space we look for.
\begin{rmk}\label{comparetolocalremlemma}
We remark that in terms of Lemma \ref{localrem}, we have $m^{0,+}=m$ and $m^{\infty,+}=-m$. Moreover $r>0$  implies that the ramification numbers $r^{0/\infty,+}$ given in Lemma \ref{localrem} coincides with $r$ in$\mod p$ and (\ref{multiplicitiesforlens}) coincides with the conditions given in Lemma \ref{localrem}.

\end{rmk}
\begin{rmk}\label{componentsofmoduli} In the previous section, we constructed the covering $(\Sigma,\tilde{\Gamma},\mathfrak{p})$ together with a fixed generator $\tau_\flat$ for the Deck group so that given the moduli space with fixed asymptotics, the equivariant curves are given by $u\circ\tau_\flat=\sigma\circ u$. In the above treatment, the representation $m$ of $\tau_\flat$ determines $r$ and hence the homotopy classes of positive ends. Hence different choices of $\tau_\flat$, equivalently $r$,  lead to different homotopy classes of non-contractible ends and 
therefore different components of the moduli space.  
\end{rmk}

We want to compute the dimension of the moduli space $\mathcal{M}$ using the above description. 
We note that the formula given in Corollary \ref{generaldim} immediately applies here. Nevertheless, we 
repeat this computation by directly looking at the lifted base curves. 

Let $c$ be a curve as in Lemma \ref{cforlens}. 
As a lifted base curve, $c$ contributes to the index of the pair of pants in only two ways. 
One is the freedom of moving the roots of $g$ and $h$ and moving the constant $\lambda$. In fact, if one moves the ``roots`` of $z^{r}$, 
then the resulting non zero root have to appear $p$ many but this is not possible since 
this makes the degree of the map jump. 
The contribution of moving "roots" of  $g$ and $h$ is little delicate. Given a term like $(z^r-a_s)^{k_s}$ 
in the factorizations of these polynomials, one has to move roots of $a_s$ simultaneously to keep invariance, i.e. we
only have the freedom  moving $a_s$'s and $b_t$'s. More precisely, if $d$ is the degree of a given base curve $c$ then the contribution of the base curve to the index is $4\lfloor{ d/p}\rfloor+2$.

For the problem $\mathcal{M}$ given by (\ref{popinlens}) and (\ref{asypmofpopinlens}), the lifted base curves satisfies
$r>0$ and $r+p\sum k_s>\sum l_t$ so that the degree of $c$ is given by $r+p\sum k_s$ and therefore $\lfloor{ d/p}\rfloor=\sum k_s$. 
Hence for fixed $a>0$ and $d^I:=\lfloor{ d/p}\rfloor\geq 0$ the index of the problem is given by (compare to Corollary \ref{generaldim})
\begin{eqnarray}\label{equivindex}
4+4d^I. 
\end{eqnarray}
where we add 2-dimensional freedom of rescaling sections. 
\begin{center}
\emph{The index of holomorphic curves in the Morse-Bott setting}
\end{center}

We first recall generalities about the virtual dimension of the moduli space (\ref{generalmoduli}) where $(M,\xi=\ker \alpha)$ is a 3-dimensional Morse-Bott contact manifold. As in the non-degenerate case, the virtual dimension a moduli space is determined by the Conley-Zehnder indices of the asymptotic ends, where the Conley-Zehnder index is suitably generalized to the Morse-Bott situation. We consider the generalization given in \cite{RS}, which is axiomatically described as follows in dimension 2, see \cite{Siefring} and \cite{Gutt}.

Let $\Sigma(1)$ be the space of paths $\varphi: [0,1]\rightarrow Sp(1)$ with $\varphi(0)=I$, where $Sp(1)$ is the space of 2-by-2 symplectic matrices and $I$ is the identity matrix. The \textit{Conley-Zehnder index} is the is a unique map $\mu :\Sigma(1)\rightarrow \frac{1}{2}\mathbb{Z}$ characterized by the following axioms. 
\begin{enumerate}[label=\textrm{(CZ\arabic*)}]
    \item \label{homotopyinvariance} $\mu$ is constant on homotopies $\varphi_s\in \Sigma(1)$ for which $\dim \ker (\varphi_s(1)-I)$ is constant.
    \item \label{maslov} If $\varphi\in \Sigma(1)$ and $\psi:\mathbb{R}/\mathbb{Z}\rightarrow Sp(1)$ is a loop then 
    $$\mu(\psi\varphi)=\mu(\varphi)+2m(\psi)$$
    where $m(\psi)$ is the Maslov index of $\psi$.
    \item \label{inverse} If $\varphi\in \Sigma(1)$ and $\varphi^{-1}\in \Sigma(1)$ is the corresponding path of inverses, then
    $$\mu(\varphi)+\mu(\varphi^{-1})=0.$$
    \item \label{normalization} $\mu(e^{i\pi t})=1$ and if $\varphi (t)=\begin{bmatrix} 1 & -t  \\ 0 & 1  \end{bmatrix}$ then $\mu(\varphi)=\frac{1}{2}$.
\end{enumerate}
Recall that for any $x\in M$ and $t\in \mathbb{R}$, the linearization of the Reeb flow $\phi_t$ leads to a symplectic map 
$$d\phi_t(x):(\xi_x,(d\alpha)_x)\rightarrow (\xi_{\phi_t(x)},(d\alpha)_{\phi_t(x)}).$$
Let $\gamma$ be a closed Reeb orbit with period $T>0$. We fix a symplectic trivialization 
 $$\Phi:S^1\times \mathbb{R}^2\rightarrow \gamma(T\cdot)^*\xi$$
 where $S^1=\mathbb{R}/\mathbb{Z}$.
Then the Conley-Zehnder index of the orbit $\gamma$ with respect to $\Phi$ is given by
\begin{eqnarray}\label{CZoforbit}
    \mu^{\Phi}(\gamma):=\mu \left(\left\{t\mapsto \Phi^{-1}(t)\circ d\phi_{Tt}(\tilde{\gamma}(0))\circ \Phi(0)\right\}\right).
\end{eqnarray}
Let $C=[\Sigma,j,\Gamma,u]\in \mathcal{M}$.  We pick a collection $\{\Phi_{z_i^\pm}\}$ of trivializations for the asymptotic ends $\{\gamma_i^\pm\}$. Then we consider the complex line bundle $(u^*\xi,J)\rightarrow \Sigma\setminus \Gamma$. Let $U\subset \Sigma$ be an open neighbourhood of the puncture set $\Gamma$, consisting of disks centered at each puncture. We endow each such disk with cylindrical coordinates via 
$$[0,+\infty)\times S^1\rightarrow \mathbb{D},\; (s,t)\mapsto e^{-2\pi(s+it)}.$$ 
and we extend $\{\Phi_{z}\}$ to a complex trivialization $\Phi: U\times \mathbb{C}\rightarrow (u^*\xi,J)_{|U}$ and define \textit{first Chern number of $u$ relative to} $\Phi$ is defined to be the signed count 
\begin{eqnarray}\label{relativechernnumber}
c_1^\Phi(u):=\#s^{-1}(0)
\end{eqnarray}
where $s$ is a generic section of $u^*\xi$ such that $\Phi(s)=1$. Having all the ingredients at hand, the \textit{Fredholm index of} $C=[\Sigma,j,\Gamma,u]\in \mathcal{M}$  reads as 
\begin{equation}\label{Fredholmindex}
\begin{aligned}
{\rm ind}(C)=&\:2c_1^\Phi(u)+\sum\limits_{z_i^+\in \Gamma^+}\mu^{\Phi}(\gamma_i^+)-\sum\limits_{z_i^-\in \Gamma^-}\mu^{\Phi}(\gamma_i^-)\\ 
&+\frac{1}{2}\sum\limits_{z_i^\pm\in \Gamma}\dim S_i^\pm +\#\Gamma -\chi(\Sigma),  
\end{aligned}
\end{equation}
see \cite{Bourgeois} and \cite{Siefring}. We recall that ${\rm ind }(C)$ is precisely the index of the Fredholm operator associated to the curve $C$ so that once it is surjective, $\mathcal{M}$ gains a smooth structure near $C$ and the kernel of the operator defines the the tangent space $T_C\mathcal{M}$.
 \begin{center}
\emph{The index of pair of pants} 
 \end{center}
 
The aim of this section is to compute (\ref{Fredholmindex}) for any curve in the moduli space $\mathcal{M}$ of pair of pants given by (\ref{popinlens}) and (\ref{asypmofpopinlens}). We note that this make sense due to \ref{paramunparam}. Our strategy is utilize the lifting procedure here as well. 

We first fix \qq{trivializations} of $\xi_0$ and $\xi$. Using (\ref{alpha0}) we define a non-vanishing section
$$s:S^3\rightarrow \mathbb{C}^2;\;(z_1,z_2)\mapsto(\overline{z}_2,-\overline{z}_1)$$ of $\xi_0$ and get a global complex trivialization 
\begin{eqnarray}\label{Phi0}
\Phi_0:S^3\times \mathbb{C}\rightarrow \xi_0,\; ((z_1,z_2),\lambda)\mapsto \lambda s(z_1,z_2).
\end{eqnarray}
Note that on $\xi_0$, the standard almost complex structure $J_0$ coincides with $i$. As mentioned in the introduction, $\xi\rightarrow L(p,q)$ is non-trivial. Nevertheless, it is convenient to fix a section of it which vanishes along a mild subset. We construct such a section as a quotient of a section of $\xi_0$ as follows. 
We define a section
\begin{eqnarray}\label{equivsection}
k(z_1,z_2):=f(z_1,z_2)s(z_1,z_2),\; f(z_1,z_2):=z_1^{q+1}+z_2^{v+1}
\end{eqnarray}
An easy computation shows that the section $k:S^3\rightarrow \xi_0\subset TS^3$ is equivariant in the sense that 
$$k(\sigma(z_1,z_2))=d\sigma_{(z_1,z_2)}[k(z_1,z_2)]$$
where $\langle d\sigma\rangle$ is the induced action on $TS^3$ for which $\xi_0$ is invariant. We note that $k$ vanishes along a torus knot $K\in S^3$, which does not intersect with $\gamma_0$ or $\gamma_\infty$. It is not hard to see that $K$ descends to a curve that is homologous to $(q+1)[\overline{\gamma}_0]$ in $L(p,q)$. 
We define another trivialization of $\xi_0$ away from $K$
\begin{eqnarray}\label{Phi}
\Phi:(S^3\setminus K\times \mathbb{C})\rightarrow \xi_0,\; ((z_1,z_2),\lambda)\mapsto \lambda k(z_1,z_2)
\end{eqnarray}
so that $\Phi$ induces a trivialization of $\xi$ away from $\texttt{p}(K)$ given by 
\begin{eqnarray}\label{Phibar}
\overline{\Phi}:(L(p,q)\setminus \texttt{p}(K))\times \mathbb{C}\rightarrow \xi\;\;\textrm{such that}\;\; \overline{\Phi}\circ (\texttt{p}\times {\rm id})=d\texttt{p}\circ \Phi.
\end{eqnarray}
We use this trivialization along the orbits which are away from the vanishing set $\texttt{p}(K)$. 

We first compute the Conley-Zehnder indices of the positive ends given by (\ref{asypmofpopinlens}).  Due to the definition of $\overline{\Phi}$, instead of considering the flow of $R$ along $k^\infty\overline{\gamma}_{\infty}$, we consider the flow of $R_0$ along the lifted Reeb arc 
$$[0,k^\infty\theta]\rightarrow S^3,\;t\mapsto (e^{it},0);\; \theta=2\pi/p$$
together with the trivialization $\Phi$. The resulting symplectic arc $\varphi_\infty(t)$ is given by
$$\varphi_\infty(t)= \Phi^{-1}(e^{itk^\infty\theta},0)\circ d\phi_{(t k^\infty\theta)}(1,0)\circ \Phi(1,0).$$
Viewing $\varphi_\infty(t)\in \mathbb{C}$, we compute
\begin{eqnarray*}
\varphi_\infty(t)\Phi(e^{itk^\infty\theta},0) &=&d\phi_{(t k^\infty\theta)}(1,0)[f(1,0)(0,-1)]\\
\varphi_\infty(t) f(e^{itk^\infty\theta},0)s(e^{itk^\infty}\theta,0)&=&(0,-e^{itk^\infty\theta})\\
\varphi_\infty(t) e^{it(q+1)k^\infty\theta}(0,-e^{-itk^\infty}\theta)&=&(0,-e^{itk^\infty\theta})
\end{eqnarray*}
and get 
$$\varphi_\infty(t) =e^{it2\pi\frac{k^\infty(1-q)}{p}}.$$
A similar computation for $k^0\overline{\gamma}_0$ leads to the arc
$$\varphi_0(t)=e^{it2\pi\frac{k^0(1-v)}{p}}.$$
We note that the orbits $k\overline{\gamma}_{0/\infty}$ are non-degenerate  if and only if  $k\not\equiv 0$ since $v,q>1$. Hence we can compute the Conley-Zehnder indices of $k_{0/\infty}\overline{\gamma}_{0/\infty}$ using a standard computational recipe, see \cite{Gutt}. We note that 
$${\rm det}_\mathbb{R}(I-\varphi_{0/\infty}(1))>0,$$
that is $\varphi_{0/\infty}(1)$ and $-I=e^{i\pi}$ are in the same component of non-degenerate symplectic matrices.  We connect $\varphi_{\infty}(1)=e^{i2\pi\frac{k^\infty(1-q)}{p}}$ to  
$-I=e^{i2\pi\left(\lfloor{\frac{k^\infty(1-q)}{p}}\rfloor+\frac{1}{2} \right)}$ by a rotation that does not hit the Maslov cycle. Squaring the resulting loop and computing the degree leads to
\begin{eqnarray}\label{muofgammainfty}
\mu^{\overline{\Phi}}(k^\infty\overline{\gamma}_{\infty})=2\left\lfloor{\frac{k^\infty(1-q)}{p}}\right\rfloor+1.
\end{eqnarray}
A similar argument gives
\begin{eqnarray}\label{muofgamma0}
\mu^{\overline{\Phi}}(k^0\overline{\gamma}_{0})=2\left\lfloor{\frac{k^\infty(1-v)}{p}}\right\rfloor+1.
\end{eqnarray}

\begin{lem}\label{indexpantslemma} The Fredholm index of a pair of pants given by (\ref{popinlens}) and (\ref{asypmofpopinlens}) is given by
\begin{equation}\label{indexpantsformula}
{\rm ind}\;(\overline{u})=\mu^{\overline{\Phi}}(k^\infty\overline{\gamma}_{\infty})+\mu^{\overline{\Phi}}(k^0\overline{\gamma}_{0})+\frac{2}{p}\left(d+k^0v+k^\infty q\right)-2k+2
\end{equation}
where $d=k^0+k^\infty-pk$.
\end{lem}
\begin{proof} Let $\overline{u}$ be given by (\ref{popinlens}) and (\ref{asypmofpopinlens}). We can safely assume that $\overline{u}(1)=(-\infty,k\overline{\gamma})$ and $\overline{\gamma}$ is away from the vanishing set of $\overline{\Phi}$. Then by (\ref{Fredholmindex}), the index of $\overline{u}$ is given by
\begin{eqnarray}\label{indexpantsunworked}
{\rm ind}\:(\overline{u})=2c_1^{\overline{\Phi}}(\overline{u})+\mu^{\overline{\Phi}}(k^0\overline{\gamma}_{0})+\mu^{\overline{\Phi}}(k^\infty\overline{\gamma}_{\infty})-\mu^{\overline{\Phi}}(k\overline{\gamma})+2.
\end{eqnarray}
Now instead of directly computing the unknown terms, we utilize the lifting procedure again. Let $u$ be a lift of $\overline{u}$. We consider $u$ as a representative of an unparametrized curve $C$ and compute the index via the formula (\ref{Fredholmindex}) and the trivialization $\Phi$ given by (\ref{Phi}). We get 
\begin{eqnarray*}
{\rm ind}\:(C)&=&2c_1^{\Phi}(u)+\mu^{\Phi}(k^0\gamma_0)+\mu^{\Phi}(k^\infty\gamma_\infty)-\sum_{i=1}^{p} \mu^{\Phi}(k\gamma_i)\\
&&+\frac{1}{2}\sum\limits_{i=1}^{p+2}\dim S_i + (p+2) -2 \\
&=&2c_1^{\Phi}(u)+\mu^{\Phi}(k^0\gamma_0)+\mu^{\Phi}(k^\infty\gamma_\infty)-\sum_{i=1}^{p} \mu^{\Phi}(k\gamma_i)
+2p+2
\end{eqnarray*}
where $\gamma_1,...,\gamma_p$ are $p$-distinct lifts of the simple contractible orbit $\overline{\gamma}$. Due to the equivariance, we have $\mu^{\Phi}(k\gamma_i)=\mu^{\overline{\Phi}}(k\overline{\gamma})$ for all $i$ and $c_1^{\overline{\Phi}}(\overline{u})=c_1^{\Phi}(u)/p$. Hence we have 
\begin{eqnarray}\label{c1barminusmugammabar}
2c_1^{\overline{\Phi}}(\overline{u})-\mu^{\overline{\Phi}}(k\overline{\gamma})
=\frac{1}{p}\left({\rm ind}\:(C)-\mu^{\Phi}(k^0\gamma_0)-\mu^{\Phi}(k^\infty\gamma_\infty)-2p-2\right).
\end{eqnarray}
Now we compute the right hand side of the above equation. A computation similar to the one carried out for $k_{0/\infty}\overline{\gamma}_{0/\infty}$ leads to the symplectic paths
$$\psi_\infty(t)=e^{it2\pi k^\infty(1-q)},\;\psi_0(t)=e^{it2\pi k^0(1-v)}$$
for $k^\infty \gamma_\infty$ and $k^0\gamma_0$ respectively. We first observe that the constant path $\mathcal{I}(t)=I$ leads to
\begin{eqnarray*}
\mu(\mathcal{I})=\frac{1}{2}(\mu(\mathcal{I})+\mu(\mathcal{I}))
=\frac{1}{2}(\mu(\mathcal{I})+\mu(\mathcal{I}^{-1}))=0
\end{eqnarray*}
by \ref{inverse}. 
For a general symplectic path of the form $\varphi=e^{itk2\pi}$, $k\in\mathbb{Z}$, viewing the inverse path as a loop and using \ref{maslov} we get 
$$0=\mu(\mathcal{I})=\mu(\varphi^{-1}\varphi)=\mu(\varphi)+2m(\varphi^{-1})=\mu(\phi)-2k\;\Rightarrow\; \mu(\varphi)=2k.$$
Hence we get
$$\mu^{\Phi}(k^\infty\gamma_\infty)=2k^\infty(1-q),\;\mu^{\Phi}(k^0\gamma_0)=2k^0(1-v).$$
Now for the term ${\rm ind}\:(C)$, we use the trivialization $\Phi_0$ given by (\ref{Phi0}). Note that since it is induced by a non-vanishing section, we have $c_1^{\Phi_0}(u)=0$. For any closed orbit 
\begin{eqnarray*}
k\gamma(t)=(e^{it}z_1,e^{it}z_2),\;t\in[0,2\pi k]
\end{eqnarray*}
the associated symplectic path reads as
\begin{eqnarray*}
\varphi(t)\Phi_0(e^{i t k 2\pi}z_1,e^{i t k 2\pi}z_2)&=&d\phi_{(t k2\pi)}(z_1,z_2)[(\overline{z}_2,-\overline{z}_1)]\\
\varphi(t)(e^{-i t k 2\pi}\overline{z}_1,e^{-i t k 2\pi}\overline{z}_2)&=&(e^{i t k 2\pi}\overline{z}_2,e^{i t k 2\pi}-\overline{z}_1)\\
\varphi(t)&=&e^{i t (2k) 2\pi}.
\end{eqnarray*}
Hence we get $\mu^{\Phi_0}(k\gamma)=4k$ and therefore 
\begin{eqnarray*}
{\rm ind}\:(C)&=&2c_1^{\Phi_0}(u)+\mu^{\Phi_0}(k^0\gamma_0)+\mu^{\Phi_0}(k^\infty\gamma_\infty)-\sum_{i=1}^{p} \mu^{\Phi_0}(k\gamma_i)\\
&&+\frac{1}{2}\sum\limits_{i=1}^{p+2}S_i + (p+2) -2 \\
&=&4k^0+4k^\infty-p4k+2p+2.
\end{eqnarray*}
Then (\ref{c1barminusmugammabar}) leads to
\begin{eqnarray*}
2c_1^{\overline{\Phi}}(\overline{u})-\mu^{\overline{\Phi}}(k\overline{\gamma})
&=&\frac{1}{p}\left(4k^0+4k^\infty-p4k-2k^0(1-v)-2k^\infty(1-q)\right)\\
&=&\frac{2}{p}\left(d+k^0v+k^\infty q\right)-2k.
\end{eqnarray*}
Substituting the above formula in (\ref{indexpantsunworked}) leads to the formula (\ref{indexpantsformula}).
\end{proof}
\begin{lem} \label{reg.pant.} For any $\overline{u}$ given by (\ref{popinlens}) and (\ref{asypmofpopinlens}), we have 
$${\rm ind}\:(\overline{u})=4+4d^I$$
where $d^I=\lfloor{ d/p}\rfloor$.
\end{lem}
\begin{proof}
We write
$$k^\infty(1-q)=l_\infty+n_\infty p,\;\;k^0(1-v)=l_0+n_0 p\,;\; 0<l_\infty,l_0<p$$
so that 
$$\mu(k^\infty\overline{\gamma}_{\infty})=2n_\infty+1,\; \mu(k^0\overline{\gamma}_0)=2n_0+1.$$
We note that by (\ref{multiplicitiesforlens}),
$$k^\infty(1-q)\equiv k^0(1-v)\equiv k^\infty+k^0.$$
By (\ref{multipanddegree}) we also have $k^\infty+k^0=r+p(d^I+k)$ where $0<r<p$. In particular, $k^\infty+k^0\equiv r$. 
Hence we conclude that $r=l_1=l_2$. Combining all these, we get 
\begin{eqnarray*}
 {\rm ind}\:(\overline{u})&=&\frac{2}{p}\left(k_1q+k_2u+d\right)+ \mu^{\overline{\Phi}}(k^\infty\overline{\gamma}_{\infty})+\mu^{\overline{\Phi}}(k^0\overline{\gamma}_0)-2k+2\\
&=&\frac{2}{p}\left(k^\infty -r-n_\infty p+k^0 -r-n_0 p +r +pd^I\right)\\
&&\;\;\;+(2n_\infty+1)+(2n_0+1)-2k+2\\
&=&\frac{2}{p}\left(k^0+k^\infty  -r +pd^I-(n_0+n_\infty)p\right)\\
&&\;\;+2(n_0+n_\infty)-2k+4\\
&=&\frac{2}{p}\left(r+pd^I+pk -r +pd^I-(n_0+n_\infty)p\right)\\
&&\;\;\;+2(n_0+n_\infty)-2k+4\\
&=&\frac{2}{p}\left(2pd^I+pk-(n_0+n_\infty)p\right)+2(n_0+n_\infty)-2k+4\\
&=&\left(4d^I+2k-2(n_0+n_\infty)\right)+2(n_0+n_\infty)-2k+4\\
&=&4d^I+4.
\end{eqnarray*}
\end{proof}
We list the outcomes of above discussion below. 
\begin{rmk}\label{Jalpharegularforpants}(Regularity of $J_\alpha$)
As a first corollary, we conclude that any non-empty component of the moduli space (\ref{popinlens})-(\ref{asypmofpopinlens}) is cut out transversally. In fact, the identity above shows that the quotient almost complex structure 
$J_\alpha$ on $\mathbb{R}\times L(p,q)$ is regular for any $\overline{u}$ given by  (\ref{popinlens})-(\ref{asypmofpopinlens}). The holomorphic perturbations of $\overline{u}$
are in one to one correspondence with the equivariant holomorphic perturbations of the lift $u$. 
But the dimension of these perturbations is equal to the equivariant index associated to 
$\overline{u}$ and therefore to ${\rm ind}(\overline{u})$ by the above lemma. Hence the kernel of the corresponding Cauchy-Riemann
operator has the dimension equal to the index and the cokernel is trivial. 
\end{rmk}
\begin{rmk} \label{moduli.lens.}(Components of the moduli space) The equivariant index makes easy to determine the components of the moduli space of pants with 
two non-contractible ends. In the case above, when $d^I$ and the multiplicity of the contractible end are
fixed, we have $p-1$ components,
corresponding to each  value of $0<r<p$ and therefore to each non-trivial homotopy class of ends. 
We note that the minimal index for the moduli space of pair of pants is 4 and it corresponds to $d^I=0$.
Each component
may be identified with $\mathbb{C}^*\times\mathbb{C}^*$, where one factor stands for the freedom of moving $\lambda$
and the other corresponds to the rescaling of the meromorphic section.

We note that once $d^I$ is fixed, one may increase multiplicities of all ends simultaneously in such a way that 
the degree condition is satisfied and the index is unchanged. Hence, the big moduli space of pair of pants where
the asymptotics at punctures $0,\infty$ are fixed geometrically as above, can be written as 
$$\bigcup\limits_{d^I\geq 0}\bigcup\limits_{k \geq 0}\bigcup\limits_{a=1}^{p-1}\mathcal{M}_{d^I,k,a};
\; \dim \mathcal{M}_{d^I,k,a}=4+4d^I,$$
where $k$ is the multiplicity of the contractible end.
\end{rmk} 
\begin{center}
\emph{Adding more contractible ends}
\end{center}

The discussion above can be repeated with minor modifications for other curves with only two non-contractible ends.

Adding more contractible ends to the pair of pants configuration does not require any essential change. To be 
more specific, lets consider the moduli space $\mathcal{M}$ of unparametrized curves with two non-contractible
positive ends, that are fixed geometrically as above, $s^+$-many contractible positive ends and $(s^-+1)$-many negative contractible ends. As noted in Remark \ref{paramunparam}, we interpret
$\mathcal{M}$ as the set of tuples 
$$\left(\overline{u}, (z_1,..,z_{s^+}),(w_1,...,w_{s^-})\right)$$
where $\overline{u}$ has the positive ends at $0$  asymptotic to $k_2\overline{\gamma}_0$ and at $\infty$
asymptotic to $k_1\overline{\gamma}_{\infty}$, a negative end at $1$ of multiplicity $k$, positive ends at 
$z_1,..,z_{s^+}$ of multiplicities $l_1^+,...,l_{s^+}^+$, negative ends at $w_1,...,w_{s^-}$ 
of multiplicities $l_{1}^-,..., l_{s^-}^-$. That is we fix tree punctures and let the rest of the 
punctures move. Given such an object, we have a lifted base curve $c$ with $r>0$ and ${\rm deg}(c)=r+pd^I$
and 
$$-k^\infty-k^0-p\sum l_i^++pk+p\sum l_j^-=-(r+pd^I).$$
The multiplicities $k_i$ are determined in$\mod p$ as before. The dimension of the component of  
$\left(\overline{u}, (z_1,..,z_{s^+}),(w_1,...,w_{s^-})\right)$ is then 
$$4+4d^I+2s^++2s^-$$
since we added the freedom of choosing the places of zeros and poles of the meromorphic section 
other than $0,\infty$ and the lifts of $1$. Yet in the lift, lifts of the remaining zeros and poles
should be distributed invariantly. Moreover an analysis similar to above shows that 
$${\rm ind}\:(\overline{u}, (z_1,..,z_{s^+}),(w_1,...,w_{s^-}))=4+4d^I+2s^++2s^-.$$
In order to describe the components of the moduli space, one simply adds free zeros and poles to the pair of pants configurations and the multiplicities of the non-contractible
ends are adjusted to get the degree condition is satisfied. For fixed values of $d^I$, $a$ and $k$ as before,  each component of the moduli space may be identified with 
$$\mathcal{M}_{d^I,k,a}\times\, \mathfrak{S}^{(s^++s^-)}(\mathbb{C}P^1\setminus \{0,1,\infty\})$$
where $\mathcal{M}_{d^I,k,a}$ is given in Remark \ref{moduli.lens.}
\begin{center}
\emph{The cylinders}
\end{center}

For the cylinders, we first consider the parametrized cylinders and then mode out  
biholomophisms that fix the punctures. 
In our case, this corresponds to removing the contractible end from the configurations of the pair of pants with one positive and one 
negative non-contractible ends and killing the freedom in the domain. 
More concretely, lets consider the moduli space of cylinders with 
positive end asymptotic to $k^0\overline{\gamma}_0$ and negative end asymptotic to $k^\infty\overline{\gamma}_{\infty}$.  The equivariant 
index of parametrized cylinders is again $4+4d^I$ and moding out reparametrization means 
we kill the freedom of moving $\lambda$ in the definition of the lifted curve $c$, see Lemma \ref{cforlens}. Hence the equivariant index reads as $2+4d^I$. 
On the other hand, the Fredholm index reads as
$${\rm ind}([\overline{u}])=2c^{\overline{\Phi}}_1(\overline{u})+\mu^{\overline{\Phi}}(k^0\overline{\gamma}_0)-\mu^{\overline{\Phi}}(k^\infty\overline{\gamma}_{\infty}).$$
We know that a lift $u$ lies above a closed curve $c$ so that  $k^\infty-k^0=-{\rm deg}(c)$ and similar to Lemma \ref{indexpantslemma} we get
$$c_1(\overline{u})=k^0v-k^\infty q+k^\infty-k^0.$$
Combining all these, one can show that ${\rm ind}([\overline{u}])=4d^I+2$ and 
therefore we have transversality for cylinders as well. 
Concerning the big moduli space, we have two parameters, namely $d^I$ and $a$ that index the components.

\section{An application}
In this section, we carry out a neck-stretching procedure, which is initiated by a positive contactomorphism. We perturb $J_\alpha$- holomorphic pair of pants with non-contractible positive ends in $\mathbb{R}\times L(p,q)$, which come from a particular component of the moduli space and at the end we would like to obtain a very particular holomorphic building, see Figure \ref{outline}.
\begin{figure}[h]
\includegraphics[scale=0.22]{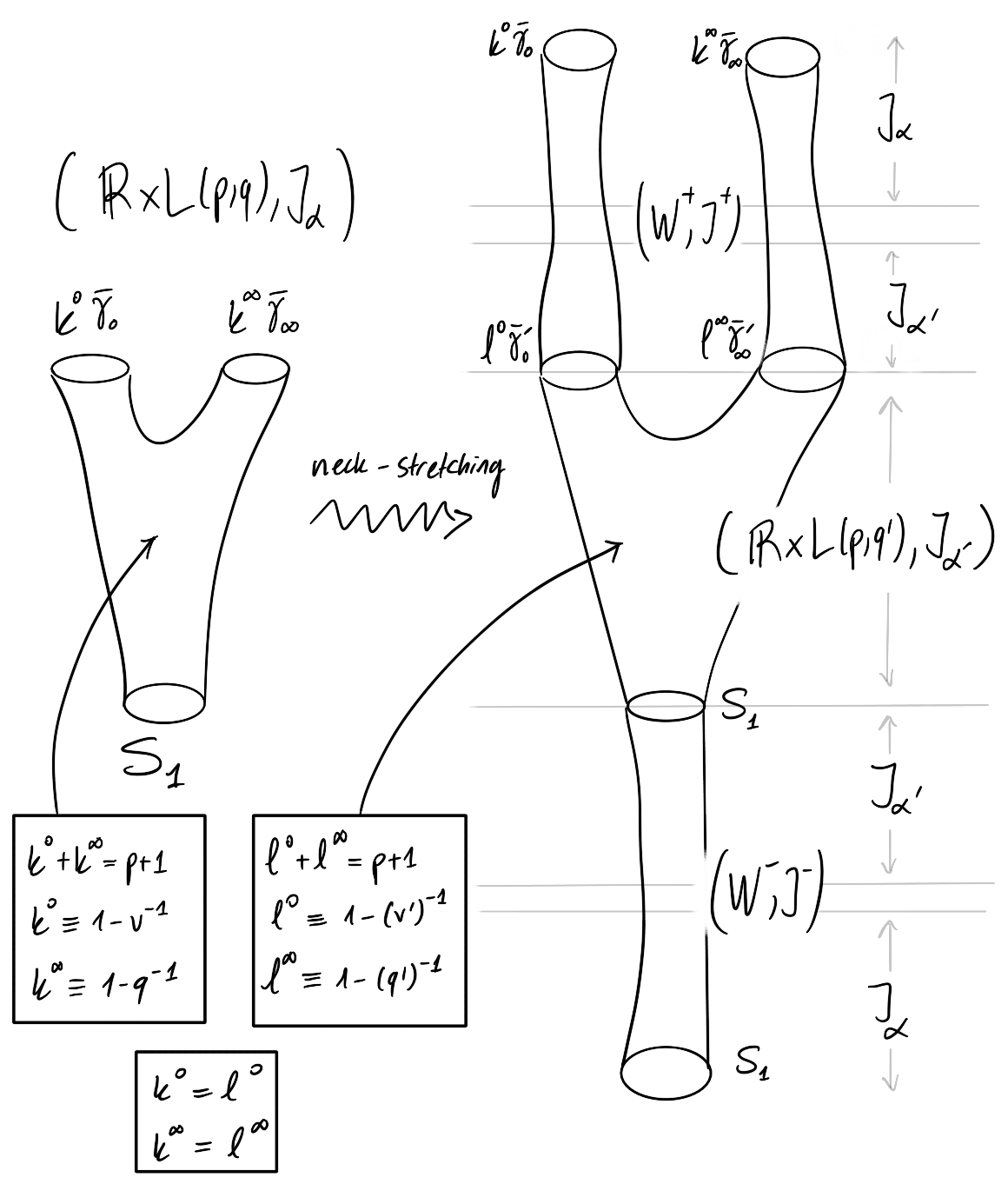}
\centering
\caption{The outcome of the neck-stretching procedure.} 
\label{outline}
\end{figure}

In general it is very unlikely to achieve such a picture since there are the common issues of transversality and compactness. It turns out that the study of the index behaviour of multiples of non-contractible Reeb orbits does not give enough control to handle these issues. The crucial observation is that one can apply well-established $4$-dimensional methods to the lifts of punctured curves in the symplectic cobordisms that show up along the way, after extending them to closed curves in certain completions of these cobordisms. The outcome of this observation is that given any holomorphic curve in a cobordism, the total action of its positive ends is greater or equal to the total action of its negative ends. We note that this property holds in general for symplectizations and for cobordisms where the contact forms at ends are suitably rescaled. But here we achieve this property for the cobordisms for which contact forms are not rescaled and this provides a strong control over the components of the limiting buildings that emerge along the neck-stretching procedure. As a result we prove the following. 
\begin{thm}\label{teo2} Let $p$ be prime and $1<q,q'<p-1$.
Suppose that there is a positive contactomorphism 
\begin{equation}\label{contacto}
\varphi: \left(L(p,q),\xi=\ker \alpha\right)\rightarrow \left(L(p,q'),\xi'=\ker \alpha'\right).    
\end{equation}
Then $q\equiv (q')^{\pm1}\mod p$.
\end{thm}
\subsection{The proof of the statement}
We take two lens spaces $L(p,q)$ and $L(p,q')$, where $p,q$ and $q'$ satisfies the assumptions of Theorem \ref{teo2}, with the contact structures induced by contact forms $\alpha$ and $\alpha'$,
which are quotients of $\alpha_0$,
see the introduction.

We consider the moduli space of $J_\alpha$- holomorphic pair of pants with non-contractible positive ends in the symplectization of $L(p,q)$, which has the minimal index. By the consideration of equivariant curves in $\mathbb{R}\times S^3\cong L^*$, see section 3,
we know that minimal index for this problem is 4 and 
we have a moduli space with countably many components that can be collected in $p-1$ groups each group  
being determined by the degree of the underlying closed curve, see Remark \ref{moduli.lens.}. 

We consider the component, denoted by $\mathcal{M}$, 
associated to the underlying closed curve with degree one and 
minimal multiplicity of non-contractible ends. More precisely,
$\mathcal{M}$ is the moduli space of pair of pants 
$$u=(a,v):\mathbb{C}P^1\setminus \Gamma \rightarrow \mathbb{R}\times L(p,q)\cong\overline{L}^*$$
with punctures $\Gamma=\{0,1,\infty\}$ and asymptotics
\begin{eqnarray}\label{pantsforproof}
u(0)=(+\infty, k^0 \overline{\gamma}_0),\,u(\infty)=(+\infty,k^\infty \overline{\gamma}_{\infty}),\,u(1)\in\{-\infty\}\times S_1 
\end{eqnarray}
where $S_1$ is the orbit space of the contractible orbits of action $2\pi$. We know from the previous section that 
\begin{eqnarray}\label{pantsforproofmultiplicities}
  k^0\equiv (1-v)^{-1},\;k^\infty\equiv (1-q)^{-1},\; k^0+k^\infty=p+1.
\end{eqnarray}
Next, we  impose a 4-dimensional constraint on $\mathcal{M}$ as follows. 
We pick a point $x_0\in L(p,q)$ away from the non-contractible orbits and with the property that $\varphi(x_0)$ is also away from the non-contractible orbits in $L(p,q')$. Given the contactomorphism in (\ref{contacto}), we have a positive function 
\begin{equation}\label{f}
  f:L(p,q)\rightarrow (0,+\infty)\;\;{\rm s.t.}\;\;  \varphi^*\alpha'=f\alpha. 
\end{equation}
Now consider the evaluation map
\begin{eqnarray}\label{evlens}
{\rm ev}:\mathcal{M}\rightarrow \mathbb{R}\times L(p,q),\; {\rm ev}(u)=u(2). 
\end{eqnarray}

We cut out a 0-dimensional submanifold of $\mathcal{M}$ via
\begin{equation}\label{cutout_J0}
\mathcal{M}^0:={\rm ev}^{-1}\left((\log f(x_0),x_0)\right). \end{equation}
By the description of the curves in $\mathcal{M}$ given in the previous section, it is easy to see that $\mathcal{M}^0$ is cut out transversely and consists of a 
single curve. In fact one notes that $x_0$ lies on some contractible orbit, say $\gamma$ and considering the equivariant lifts of the curves in $\mathcal{M}$, $\gamma$ as a point in the orbit space is the image of the point $2$ under the underlying base curve and this choice fixes the parameter $\lambda\in \mathbb{C}$ given in Lemma \ref{cforlens}. Then the choice of $x_0\in \gamma$ and the quantity $f(x_0)$ fixes the freedom over the equivariant meromorphic section, which is due to the $\mathbb{C}^*$-action on $L$ . 
\begin{rmk}\label{choiceofx0} The conditions on the point $x_0$ are  easily satisfied and play a significant role in ruling out certain unpleasant configurations at the end of the neck-stretching argument, see Lemma \ref{needfancyx0}. 
\end{rmk}  
We consider the following exact symplectomorphism
\begin{align}\label{exactsymplecto}
\Phi : \left(\mathbb{R}\times L(p,q'), d(e^t\alpha')\right)&\rightarrow \left(\mathbb{R}\times L(p,q), d(e^s\alpha)\right),\\ 
 (t,x)&\mapsto \left(t+\log f\circ\varphi^{-1}(x),\varphi^{-1}(x)\right).\nonumber
\end{align}
Let $\Sigma:=\Phi(\{0\}\times L(p,q'))$ be the contact type hypersurface in $\mathbb{R}\times L(p,q)$. Then we have 
\begin{eqnarray}\label{defnofNs}
N^+\cup N^-:=\left(\mathbb{R}\times L(p,q)\right)\setminus\Sigma
\end{eqnarray}
Where $N^+$ is the upper and $N^-$ is the lower connected component. We fix $\varepsilon>0$ and put 
\begin{align}\label{defnofU}
U:= \Phi\left([-\varepsilon,\varepsilon]\times L(p,q')\right)\subset\mathbb{R}\times L(p,q).   
\end{align}
We fix an open subset $V\subset \mathbb{R}\times L(p,q)$ such that $\overline{U}\subset V$.
For the later purposes we set
\begin{align}\label{defnofV}
V^+\sqcup V^-:=V\setminus \overline{U}    
\end{align}
where $V^+$ and $V^-$ are the upper and lower components respectively. 
Let $n$ be a positive integer. Following \cite{foliations}, we construct the symplectic cobordism $W^n$ as follow. We remove $\Phi\left((-\varepsilon/2,\varepsilon/2)\times L(p,q')\right)$ from $\mathbb{R}\times L(p,q)$ we glue $[-\varepsilon-n,n+\varepsilon]\times L(p,q)$ in the middle via the following identifications
\begin{align}\label{defnofWn}
\begin{split}
   [-\varepsilon-n,-n-\varepsilon/2]\times L(p,q')\ni(t,x)\sim \Phi(t+n,x)\in U,\\
   [n+\varepsilon/2,n+\varepsilon]\times L(p,q')\ni(t,x)\sim \Phi(t-n,x)\in U.
\end{split}   
\end{align}
We consider a smooth function $\phi_n:[-\varepsilon-n,n+\varepsilon]\rightarrow [-\varepsilon,\varepsilon]$ such that
\begin{itemize}
    \item $\phi_n'>0$,
    \item $\phi_n(t)=t+n$ for $t\in [-\varepsilon-n,-\varepsilon/2-n]$,
    \item $\phi_n(t)=t-n$ for $t\in [n+\varepsilon/2,n+\varepsilon]$,
    \item $\phi_n(0)=0$.
\end{itemize}
Such a function leads to a diffeomorphism
\begin{equation}\label{Phi_n}
\Phi_n:W^n\rightarrow \mathbb{R}\times L(p,q)
\end{equation}
where $\Phi_n={\rm id}$ on $\left(\Phi\left((-\varepsilon/2,\varepsilon/2)\times L(p,q)\right)\right)^c$ and 
$$\Phi_n(t,x)=\left(\phi_n(t)+\log f\circ\varphi^{-1}(x),\varphi^{-1}(x)\right)$$
on $[-\varepsilon-n,n+\varepsilon]\times L(p,q')$. For later purposes we note that 
\begin{equation}\label{evaluation_makes_sense_for_Jn}
\Phi_n(0, \varphi(x_0))=\left(\phi_n(0)+\log f\circ \varphi^{-1}(\varphi(x_0)), \varphi^{-1}(\varphi(x_0)\right)=\left(\log f(x_0), x_0\right).    
\end{equation}
We consider the exact symplectic form
\begin{align}\label{omega_n}
\omega_n:=\Phi^*_nd(e^s\alpha)=d\left(\Phi^*_n(e^s\alpha)\right)    
\end{align}
on $W^n$, which reads as
\begin{align}\label{omega_nreadsas}
 \omega_n=\left\{
\begin{array}{ll}   
d(e^s\alpha)&{\rm on}\; U^c\\
d(\phi_n\alpha')& {\rm on }\;[-\varepsilon-n,n+\varepsilon]\times L(p,q'). 
\end{array} 
\right.
\end{align}
We note that the standard almost complex structure $J_{\alpha'}$ is compatible with the symplectic form $d(\phi_n\alpha')$. Next we consider an almost complex structure $J_n$ on $\mathbb{R}\times L(p,q)$ with the following properties:
\begin{enumerate}[label=\textrm{(Jn\arabic*)}]
    \item \label{JnonVc}$J_n=J_\alpha$ on $V^c$,
    \item \label{Jnmiddle}$(d\Phi_n)^{-1}\circ J_n\circ d\Phi_n=J_{\alpha'}$ on $[-\varepsilon-n,n+\varepsilon]\times L(p,q')$,
    \item \label{Jnregular}$J_n$ is regular for any simple curve passing through the open subset $V^+\sqcup V^-$ and if relevant, satisfying 
    \begin{eqnarray}\label{cutout_Jn}
    {\rm ev}(u)\in \left\{(\log f(x_0),x_0)\right\}\subset \mathbb{R}\times L(p,q),
    \end{eqnarray}
    where the evaluation map ${\rm ev}$ is given by (\ref{evlens}).
   \item\label{Jnupperlowerregular}
When viewed as an almost complex structure on 
$$(N^+\setminus U)\cup \left((-\infty,n+\varepsilon]\times L(p,q')\right),$$ 
$J_n$ is regular for any simple curve passing through $V^+$ and satisfying an priori constraint. The corresponding statement holds for simple curves passing through $V^-$ when $J_n$ is viewed as an almost complex structure on 
$$\left([-n-\varepsilon,+\infty)\times L(p,q')\right)\cup (N^-\setminus U).$$
    \item \label{Jncompatible} $J_n$ is compatible with $d(e^s\alpha)$.
\end{enumerate}
\begin{figure}[h]
\includegraphics[scale=0.21]{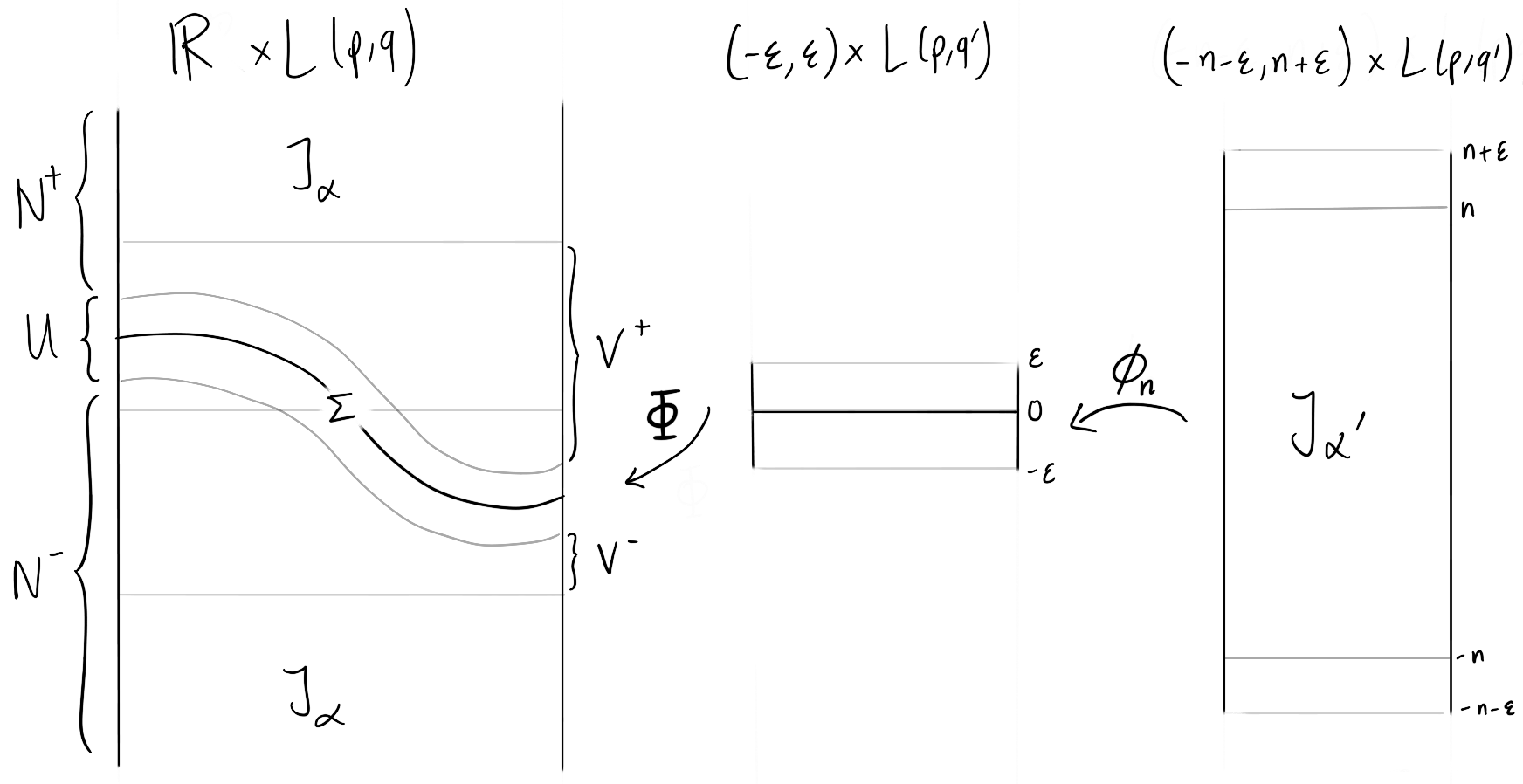}
\centering
\caption{The construction of $(W^n,J_n)$.}
\label{setting}
\end{figure}
We note that one can find an almost complex structure $J_n$ without the condition \ref{Jnregular} being satisfied. In fact, we can construct $J_n$ on by interpolating between $J_\alpha$ and the restriction of $d\Phi_n\circ J_{\alpha'}\circ (d\Phi_n)^{-1}$ to $U$ on the region $V^+\sqcup V^-$ via almost complex structures compatible with $d(e^s\alpha)$. Since $J_{\alpha'}$ is compatible with $d(\phi\alpha')$, $d\Phi_n\circ J_{\alpha'}\circ (d\Phi_n)^{-1}$ is compatible with $d(e^s\alpha)$.  Once such a reference almost complex structure is fixed, after a generic perturbation of it on the open set $V^+\sqcup V^-$, we get an almost complex structure $J_n$, for which any simple curve passing through the above open set is Fredholm regular, see Theorem 7.1 in \cite{wendlSFT}. Moreover, via an argument similar to the one in Lemma 2.5 of \cite{seidel}, we may assume that the evaluation map in (\ref{evlens}) is transverse for any simple parametrized curve passing through the the above open set. In fact, one can impose the transversality condition of the evaluation map at level of universal moduli space and show that the universal moduli space of simple curves with constraints is a smooth manifold and then one gets a generic perturbation of the almost complex structure. 

Now we consider the moduli space $\mathcal{M}^0(J_n)$ of $J_n$- holomorphic pair of pants with the asymptotics given by (\ref{pantsforproof}) and (\ref{pantsforproofmultiplicities}), satisfying (\ref{cutout_Jn}). 
\begin{lem}\label{popsaresimple} Any $J_n$- holomorphic pair of pants satisfying (\ref{pantsforproof}) and (\ref{pantsforproofmultiplicities}) is simple.
\end{lem}
\begin{proof}
We note that since the negative end of a such a pair of pants $u$ has action $2\pi$, it is a simple orbit unless without loss of generality it is $p\overline{\gamma}_{0}$. Assume that $u$ factors through a branched covering $\psi$ and a simple curve $v$. By Riemann-Hurwitz formula $v$ is rational. Since $p$ is prime, the degree of $\psi$ is $p$ and the negative end of $v$ is $\overline{\gamma}_{0}$. It is clear that $v$ has at most two positive ends. Note that $v$ can not be a cylinder since the positive ends of $u$ are geometrically distinct. Hence $v$ is also a pair of pants. In particular, $\psi$ maps punctures of $u$ to punctures of $v$ respectively. Hence the ramification number of $\psi$ at each puncture of $u$ has to be $p$. In particular $p$ divides both $k^0$ and $k^\infty$. But this is not possible since $k^0+k^\infty=p+1$.
\end{proof}
By the lemma above and \ref{Jnregular}, we conclude that $\mathcal{M}^0(J_n)$ is a 0-dimensional manifold. 

Our aim is to show that $\mathcal{M}^0(J_n)$ is not empty. To this end, we choose a generic homotopy $(J_t)_{t\in[0,1]}$ of almost complex structures with the following properties:
\begin{enumerate}[label=\textrm{(Jt\arabic*)}]
    \item \label{Jtatends} $J_0=J_\alpha$ and $J_1=J_n$,
    \item \label{JtonVc}$J_t=J_\alpha$ on $V^c$ for all $t$,
    \item \label{Jtcompatible} $J_t$ is compatible with $d(e^s\alpha)$ for all $t$,
    \item \label{Jtregular} $(J_t)_{t\in[0,1]}$ is a regular homotopy of almost complex structures for any relevant moduli problem concerning simple curves together with the evaluation condition given by (\ref{cutout_Jn}) and for which $J_n$ and $J_\alpha$ are regular.   
\end{enumerate}
The existence of such a homotopy follows from the parametric version of the geometric transversality statement used for $J_n$, see Remark 7.4 in \cite{wendlSFT}. In what follows, we need the regular homotopy property of $(J_t)_{t\in[0,1]}$ for only finitely many moduli problems. Hence $\ref{Jtregular}$ is safely assumed. 

We recall that $J_0=J_\alpha$ is regular for the pair of pants we consider and Lemma \ref{popsaresimple} applies to $J_t$. Hence the moduli space 
    $$\bigcup_{t\in[0,1]}\mathcal{M}^0(J_t)$$ 
of  $J_t$- holomorphic pair of pants with asymptotics given by (\ref{pantsforproof}) and (\ref{pantsforproofmultiplicities}) together with the evaluation condition is a one-dimensional cobordims between $\mathcal{M}^0$ and $\mathcal{M}^0(J_n)$. We note that if the cobordism $\bigcup\mathcal{M}^0(J_t)$ is compact then $\mathcal{M}^0(J_n)$ is not empty since $\mathcal{M}^0$ consists of an odd number of points. 
\begin{center}
\emph{Compactness of the cobordism} 
\end{center}

Let $t_0\in [0,1]$ and $(u_{n})$ be a sequence of $J_{t_n}$- holomorphic pair of pants such that $t_n$ converges to $t_0$. Then by the SFT compactness theorem,
there is a subsequence, again denoted by $(u_n)$, converging to a holomorphic building $u_\infty$. A priori the holomorphic building $u_\infty$ is a collection of curves that lie in $\mathbb{R}\times L(p,q)$, which are either $J_\alpha$- holomorphic or $J_{t_0}$- holomorphic. These components fit together along their asymptotic ends and lead to the level structure of the building, see \cite{Bourgeois}, \cite{SFT-compact}.  

Since the complex structure on the domain of $u_n$'s is fixed, the components of the building $u_\infty$ emerge only out of bubbling-off. In particular, there is a finite set $P\subset \mathbb{C}P^1\setminus \Gamma$ such that on $\mathbb{C}P^1\setminus (\Gamma\cup P)$ the sequence $(u_n)$ has a uniform gradient bound. Hence there exist a component 
$$u_0:\mathbb{C}P^1\setminus (\Gamma\cup P)\rightarrow \mathbb{R}\times L(p,q)$$ 
of $u_\infty$ such that one of the followings hold:
\begin{itemize}
\item $u_0$ is $J_\alpha$- holomorphic and $u_n$ converges to $u_0$ in $C^\infty_{loc}$ on $\mathbb{C}P^1\setminus (\Gamma\cup P)$ after a sequence of shifts in $\mathbb{R}$-direction,
\item $u_0$ is $J_{t_0}$- holomorphic and $u_n$ converges to $u_0$ in $C^\infty_{loc}$ on $\mathbb{C}P^1\setminus (\Gamma\cup P)$.
\end{itemize}
We note that the signs of the punctures of  $u_0$ that are in $\Gamma$ do not have to match with the signs of the punctures of $u_n$'s but the homotopy classes of corresponding asymptotics do have to match. In particular, $u_0$ is non-constant.  The remaining structure of the building is given by so called a bubble tree. Instead of describing the a priori structure of the bubble tree, we immediately utilize a particular control over the components, which is due to the following fact. 
\begin{prop}\label{actiondecrease} Given any component of the limiting building, the total action of its positive ends is greater or equal to the total action of its negative ends. 
\end{prop}
We note that for the components of $u_\infty$ that lie in upper or lower translation invariant levels the above statement is trivial. The non-trivial part of the statement is about $J_{t_0}$- holomorphic components, namely the ones in the middle level and the proof is given in the next section. 

The first implication of Proposition \ref{actiondecrease} is the absence of holomorphic planes. We note that a holomorphic plane requires a positive end of action at least $2\pi$. Together with the action of the very bottom end of the building, a finite energy plane forces the total action of the very top end of the building to be at least $4\pi$. But we know that the total action at the top is $2\pi(1+1/p)$.

Now a priori there may be components of the  building that are bubbled off at points in $P$. But any collection of such components associated to a given bubble point in $P$ must contain a finite energy plane. Since such planes are ruled out we conclude that $P=\emptyset$. Next we consider the components that are bubbled off at the punctures in $\Gamma$ and note that any such component is cylindrical since the domain of the building has arithmetic genus zero and there are no holomorphic planes. Hence we have the essential component with the puncture set $\Gamma$ (possibly with different signs compared to $u_n$'s) and the remaining components are cylindrical (possibly with two positive ends), which can be grouped into collections associated to the punctures in $\Gamma$. 

Concerning the signs of the punctures of components, we first note that the puncture 1 is has to stay negative since turning it into a positive puncture requires a cylindrical component associated to the puncture 1, which has two negative punctures. Consequently only one of the punctures among $0$ and $\infty$ may change sign. But it is easy to see that by Proposition \ref{actiondecrease} such a configuration is not possible. Hence the building consists of the essential component $u_0$, which is an honest pair of pants and we have honest cylinders. 

Concerning the level structure we note the following. Since there is no bubbling off at the marked point 2, we have $C^\infty_{loc}$-convergence of $u_n$ to $u_0$. Combining this with the fact that $u_n$'s satisfy the evaluation condition (\ref{evlens}) we conclude that the $u_0$ lies in the middle layer. 

Now we need to show that there are no non-trivial cylinders in upper and lower levels. To this end we let $\mathcal{A}^\pm$ denote the total action of the positive/negative ends in the middle layer. We note that minimal action of a closed Reeb orbit is $2\pi/p$ if the orbit is non-contractible and $2\pi$ if the orbit is contractible. By Proposition \ref{actiondecrease}, we have following possibilities regarding $\mathcal{A}^\pm$:
\begin{enumerate}
\item $\mathcal{A}^+=2\pi(1+1/p)$ and $\mathcal{A}^-=2\pi$,
\item $\mathcal{A}^+=2\pi(1+1/p)$ and $\mathcal{A}^-=2\pi(1+1/p)$,
\item $\mathcal{A}^+=2\pi$ and $\mathcal{A}^-=2\pi$. 
\end{enumerate} 
We first assume that $\mathcal{A}^+=2\pi(1+1/p)$ and $\mathcal{A}^-=2\pi$. In this case any cylinder in an upper or lower level has trivial $d\alpha$-energy. Hence any such cylinder is trivial. Now it remains to rule out last two cases.
\begin{lem} The case of $\mathcal{A}^+=2\pi(1+1/p)$ and $\mathcal{A}^-=2\pi(1+1/p)$ is not possible.
\end{lem}
\begin{proof} In this case there is no non-trivial cylinder in the upper levels and there has to be a non-trivial cylinder in a lower level. But such a cylinder has to have a contractible positive end with action $2\pi(1+1/p)$ and this is not possible. 
\end{proof}
\begin{lem} The case of $\mathcal{A}^+=2\pi$ and $\mathcal{A}^-=2\pi$ is not possible.
\end{lem}
\begin{proof} We note that there is no non-trivial component in the lower level and there is only one non-trivial cylinder in the upper level since the minimal period of Reeb orbits is $2\pi/p$ and this is precisely the action difference between $\mathcal{A}^+$ and the total action at the top of the building. The index of the non-trivial cylinder is at least 2 (in fact it is precisely 2 in this case) and therefore the index of $u_0$ is at most 2. Now if $u_0$ is simple then we get a contradiction since $(J_t)_{t\in[0,1]}$ is a regular homotopy and $u_0$ satisfies a $4$-dimensional constraint. 
\begin{figure}[h]
\includegraphics[scale=0.15]{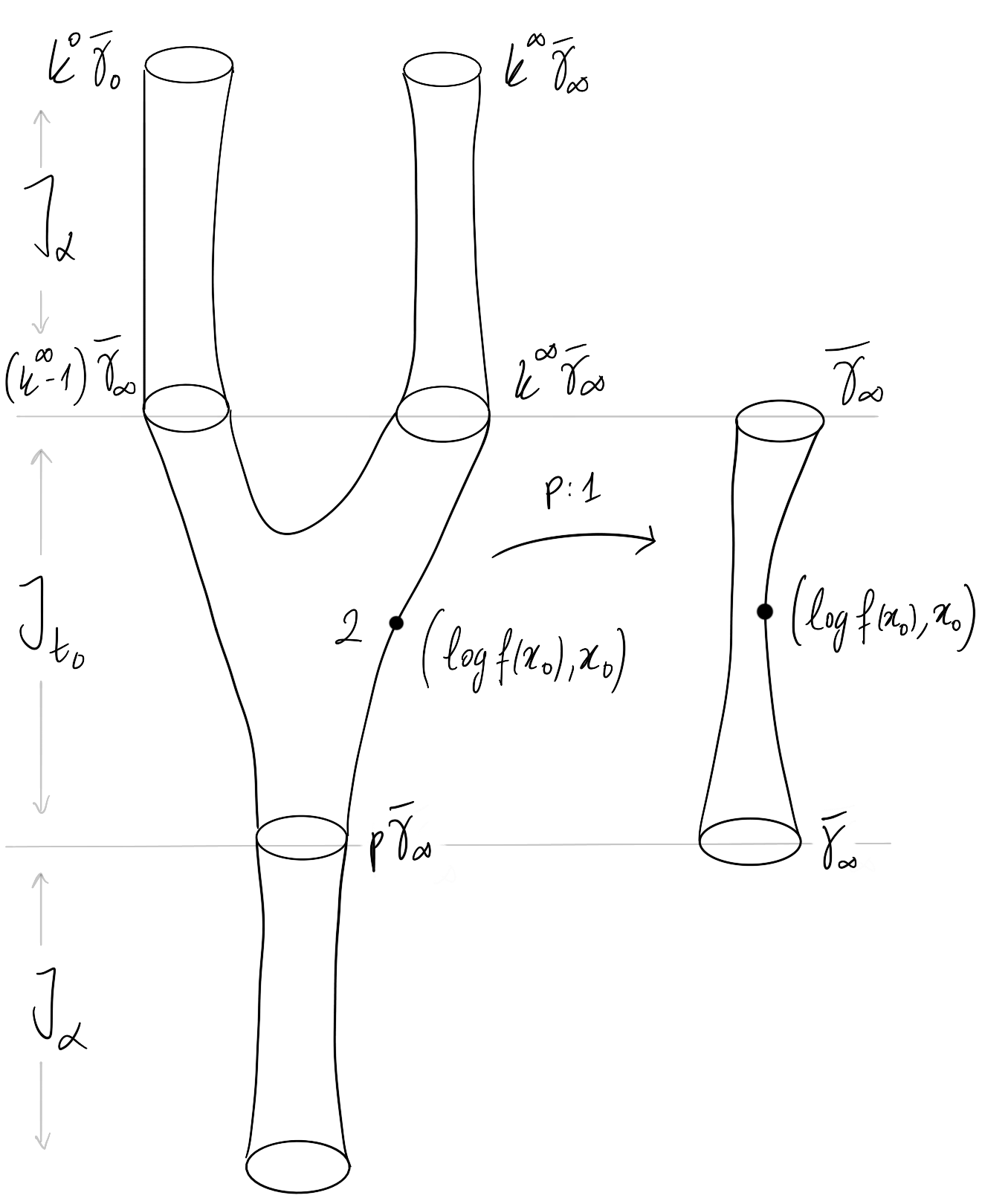}
\centering
\caption{The case $\mathcal{A}^+=2\pi$ and $\mathcal{A}^-=2\pi$.}
\label{2pi2picobordism}
\end{figure}

Now we assume that $u_0$ is multiply covered. Without loss of generality we further assume that the positive end of the non-trivial cylinder is  $k_0\overline{\gamma}_0$. Then negative end of the non-trivial cylinder is $(k_0-1) \overline{\gamma}_\infty$ and the positive ends of  $u_0$ are given by $(k_0-1) \overline{\gamma}_\infty$ and $k_\infty \overline{\gamma}_\infty$. Note that the negative end of $u_0$ can not be simple. Hence it is either $p\overline{\gamma}_0$ or $p\overline{\gamma}_\infty$. Now let $$v:\Sigma\setminus \Gamma\rightarrow \mathbb{R}\times L(p,q)$$ be the underlying simple curve and $\psi:\mathbb{C}P^1\rightarrow \Sigma$ be the branched covering so that $u_0=v\circ \psi$ and $\psi^{-1}(\Gamma)=\{0,1,\infty\}$. Let $N>1$ be the degree of $\psi$. By Riemann-Hurwitz formula, $\Sigma$ is a sphere. 
Since $u_0$ has only one negative end so does $v$ and this negative end has multiplicity $p/N$. Since $p$ is prime, we get $N=p$. It is clear that $v$ has at most two positive ends. 

We first assume that $v$ has one positive end, namely $l\overline{\gamma}_{\infty}$. Let $r_{0/\infty}$ be the ramification number of $\psi$ at the point $0/\infty$. Then we get $k_0-1=r_0l$ and $k_\infty=r_\infty l$. Since $k_0-1+k_\infty=r_0+r_\infty=p$, we get $l=1$. By Proposition \ref{actiondecrease} and the fact that $q\not \equiv \pm 1$, the negative end of $v$ is also $\overline{\gamma}_\infty$. Hence the Fredholm index of $v$ is 0 as an unparametrized curve. Fixing a parametrization and viewing $v$ as a parametrized curve it has Fredholm index 2 but it satisfies a 4-dimensional constraint induced by $u_0$ and $\psi$. Since $(J_t)_{t\in[0,1]}$ is a regular homotopy and $v$ is simple, this is not possible.

Now we assume that $v$ has two positive ends. In this case the ramification number of the points $0$ and $\infty$ are both $p$. Hence $p$ divides both $k_0-1$ and $k_\infty$. But this is not possible since $k_0+k_\infty-1=p$.
\end{proof}
Hence we conclude that the only non-trivial component of the limiting building $u_\infty$ is a pair of pants $u_0$ having the asymptotics of $(u_n)$ and satisfying the evaluation condition. This finishes the proof of the compactness of the cobordism.
\begin{center}
Stretching the neck
\end{center}
Knowing that the cobordism $\cup_{t\in [0,1]}\mathcal{M}^0(J_t)$ is compact, we get $J_n$- holomorphic pair of pants $u_n$ for each $n$. 
Now we look at the limit of the sequence $(u_n)$ as $n\rightarrow\infty$. We know that a subsequence of $(u_n)$, again denoted by $(u_n)$, converges to  a holomorphic building $u_\infty$. A priori the limiting building has $J_\alpha$- holomorphic components in upper and lower levels, $J_{\alpha'}$- holomorphic components in middle levels and finally some components in the upper and the lower connecting
levels. The connecting levels has the following description. 
Note that the hypersurface $\Sigma=\Phi(\{0\}\times L(p,q'))$ divides $\mathbb{R}\times L(p,q)$ into two components $N^+$ and $N^-$ each admitting $\Sigma$ as its boundary. The upper connecting level can be seen as the manifold
\begin{equation}\label{W^+}
W^+:=N^+\cup (-\infty,\varepsilon)\times L(p,q')
\end{equation}
where the neighbourhood $\Phi([0,\varepsilon)\times L(p,q'))$ of $\Sigma\subset N^+$ is identified with $[0,\varepsilon)\times L(p,q')$ via the symplectomorphism $\Phi$ that is given by (\ref{exactsymplecto}). $W^+$ is endowed with an almost complex structure $J^+$ such that $J^+=J_\alpha$ on $N^+\setminus V^+$ and $J^+=J_{\alpha'}$ on $(-\infty,0]\times L(p,q')$. 
Moreover $J^+$ is regular for any simple curve that passes through the open set $V^+$ together with an a priori constraint if applies. This follows from the fact that $J^+$ coincides with some $J_n$ on $V^+$ and the assumption \ref{Jnupperlowerregular}.  Similarly the lower connecting level is given by 
\begin{equation}\label{W^-}
W^-:=N^-\cup (-\varepsilon,+\infty)\times L(p,q')
\end{equation}
together with the almost complex structure $J^-$ such that $J^-=J_\alpha$ on $N^-\setminus V^-$ and $J^-=J_{\alpha'}$ on $[0,+\infty)\times L(p,q')$ and and by \ref{Jnupperlowerregular} it is regular for any simple curve passing through $V^-$ and satisfying an a priori constraint. 

As before, since the domains of $u_n$'s are fixed, we have a bubble tree structure on the limiting building. We have an essential component $u_0$ with domain $\mathbb{P}\setminus (\Gamma\cup P)$, which in this case can be $J_{\alpha}$- holomorphic or $J_{\alpha'}$- holomorphic or $J^\pm$- holomorphic. In order to rule out unpleasant components of the building, we want to argue in terms of the actions of the asymptotic ends as before. Similar to Proposition \ref{actiondecrease}, we have an a priori control over the actions of asymptotic ends of $J^{\pm}$- holomorphic components. 
\begin{prop}\label{actiondecreasefinal} For any component in $(W^\pm,J^\pm)$ the total action at its positive ends is greater or equal to the total action at its negative ends.  
\end{prop} 
We postpone the proof of this statement to the next section and continue with the proof of the main statement. We note that replacing  Proposition \ref{actiondecrease} with Proposition \ref{actiondecreasefinal}, the previous arguments  apply to $u_\infty$ word by word since in $L(p,q')$, the minimal action of the Reeb orbits  $2\pi/p$ and the action of a contractible orbit is $2\pi$. Hence we have a pair of pants $u_0$ and bunch of cylinders in the building. Moreover, $u_n$ converges to $u_0$ in $C^\infty_{loc}$   near the marked point 2 and in the light of the equation (\ref{evaluation_makes_sense_for_Jn}), $u_0$ lies in a middle layer since $u_n$'s satisfies (\ref{cutout_Jn}). Note that since the contactomorphism (\ref{contacto}) induces an isomorphism on the fundamental group, $u_0$ has two positive non-contractible ends and one negative contractible end. 

Due to the action window, $u_0$ is the only non-trivial component in the middle layer. We let $\mathcal{A}^+$ to be the total action of the positive ends of $u_0$ and $\mathcal{A}^-$ be the total action of the negative end of $u_0$. By Proposition \ref{actiondecreasefinal}, we have following possibilities regarding $\mathcal{A}^\pm$:
\begin{enumerate}
\item $\mathcal{A}^+=2\pi(1+1/p)$ and $\mathcal{A}^-=2\pi$,
\item $\mathcal{A}^+=2\pi(1+1/p)$ and $\mathcal{A}^-=2\pi(1+1/p)$,
\item $\mathcal{A}^+=2\pi$ and $\mathcal{A}^-=2\pi$. 
\end{enumerate} 
We first discuss the unpleasant cases. 
\begin{lem} The case of $\mathcal{A}^+=2\pi(1+1/p)$ and $\mathcal{A}^-=2\pi(1+1/p)$ is not possible.
\end{lem}
\begin{proof}
Note that the negative end of $u_0$ has action $2\pi(1+1/p)$ but such an orbit can not be contractible. 
\end{proof}
\begin{lem}\label{needfancyx0} The case of $\mathcal{A}^+=2\pi$ and $\mathcal{A}^-=2\pi$ is not possible.
\end{lem}
\begin{proof}
In this case $u_0$ has trivial $d\alpha$-energy and this is possible only if $u_0$ is a cover of a trivial cylinder. In this case, the negative end of $u_0$ is either $p\overline{\gamma}_0'$ or $p\overline{\gamma}_\infty'$. Without loss of generality we assume that it is $p\overline{\gamma}_0'$. Then $u_0$ is the $p$-fold cover of the trivial cylinder over $p\overline{\gamma}_0'$. But this is not possible since the image of $u_0$ misses the point $(0,\varphi(x_0))$ in $\mathbb{R}\times L(p,q')$ due to the choice of the point $x_0$, see Figure \ref{2pi2piatNS}. 
\end{proof}
\begin{figure}[h]
\includegraphics[scale=0.15]{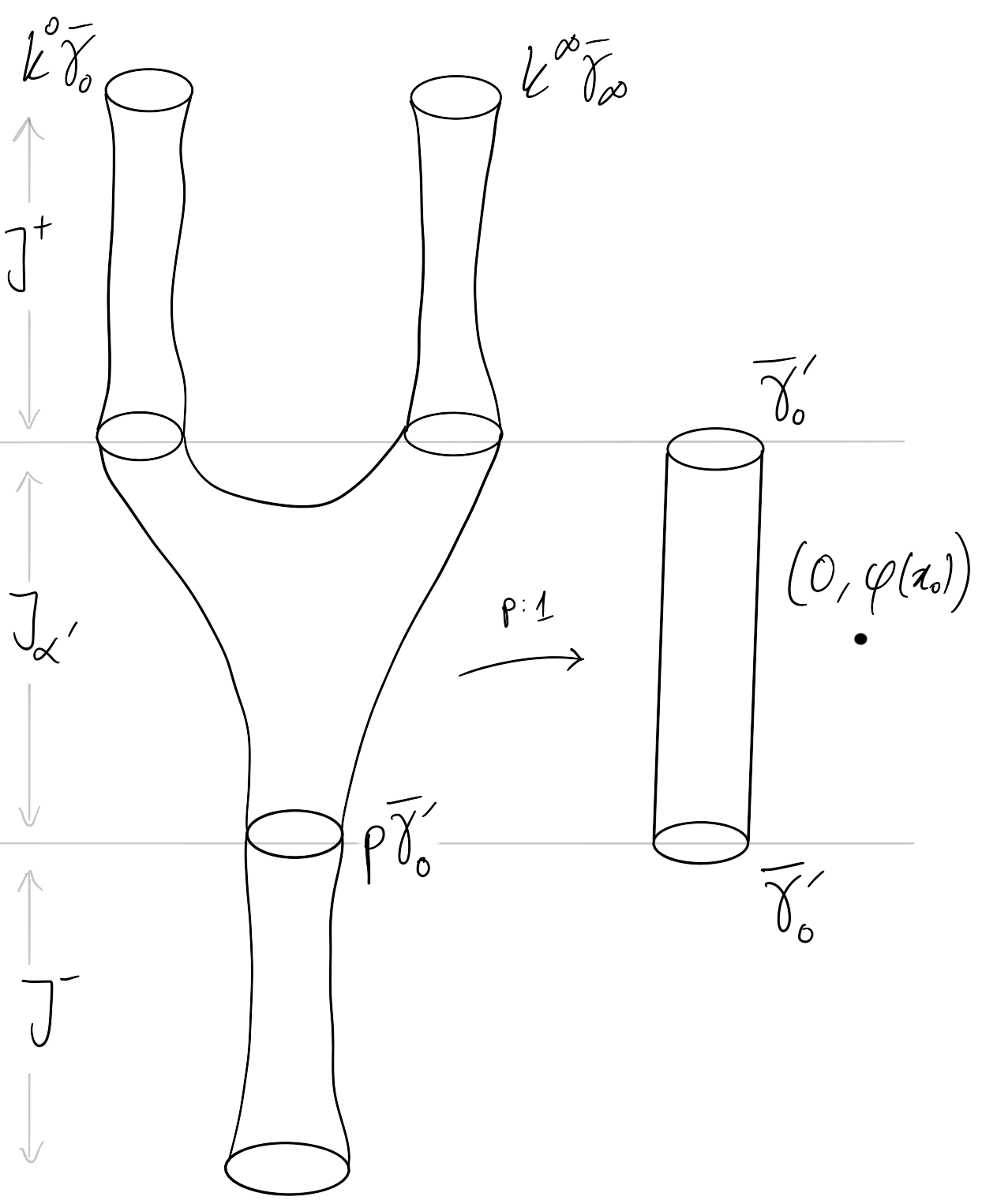}
\centering
\caption{The case $\mathcal{A}^+=2\pi$ and $\mathcal{A}^-=2\pi$.}
\label{2pi2piatNS}
\end{figure}
Now consider the case $\mathcal{A}^+=2\pi(1+1/p)$ and $\mathcal{A}^-=2\pi$. In this case $u_0$ is a $J_{\alpha'}$- holomorphic pair of pants. Moreover we have cylinders $C_1$ and $C_2$ in $W^+$  such that $C_1$ has the positive end $k^0\overline{\gamma}_0$ and   $C_2$ has the positive end $k^\infty\overline{\gamma}_\infty$.
Following the discussions given in the previous sections we conclude that $u_0$ is pair of pants with a base curve of degree 1. We have two possible profiles for the positive ends of $u_0$:
\begin{enumerate}
\item $u_0(0)=(+\infty, l^0\overline{\gamma}_0')$ and $u_0(\infty)=(+\infty, l^\infty\overline{\gamma}_\infty')$,
\item $u_0(0)=(+\infty, l^\infty\overline{\gamma}_\infty')$ and $u_0(\infty)=(+\infty, l^0\overline{\gamma}_0')$.
\end{enumerate}
In both cases we have $l^0+l^\infty=p+1$. Following Lemma \ref{localrem} we have 
$$l^0\equiv (1-v')^{-1} \;\;{\rm and }\;\; l^{\infty}\equiv (1-q')^{-1}$$
in the first case. Here the integer $v'$ is given by $1<v'<p-1$ and  $v'q'\equiv 1$. Now the negative end of $C_1$ is $l^0\overline{\gamma}_0'$ and the negative end of $C_2$ is $l^\infty\overline{\gamma}_\infty'$. We apply Proposition \ref{actiondecreasefinal} to $C_1$ and $C_2$ and use the fact that $l^0+l^\infty=k^0+k^\infty=p+1$ to conclude $l^0 =k^0$ and $l^\infty=k^\infty$. We finally get 
$$(1-v')^{-1}\equiv l^0=k^0\equiv (1-v)^{-1}\;\Rightarrow\; v'\equiv v\;\Rightarrow\; q'\equiv q.$$
In the second case, we have
$$l^0\equiv (1-q')^{-1} \;\;{\rm and }\;\; l^{\infty}\equiv (1-v')^{-1}.$$
Considering the cylinder $C_1$ again we get
$$(1-q')^{-1}\equiv l^0=k^0\equiv (1-v)^{-1}\;\Rightarrow\; q'\equiv v\;\Rightarrow\; q'\equiv q^{-1}.$$
This concludes the proof the theorem.
\subsection{Action control via 4-dimensional tools}
In this section we prove Proposition \ref{actiondecrease} and Proposition \ref{actiondecreasefinal}. The method is to apply intersection theory for closed holomorphic curves to the lifts of relevant punctured curves in symplectic cobordisms associated to lens spaces. Once these curves are lifted to the cobordisms associated to $S^3$, we compactify these cobordisms and extend the punctured curves to closed curves in order to apply the intersection theory.
\subsubsection{Basics of intersection theory}
We briefly recall the basics of the intersection theory of closed holomorphic curves. For the details and proofs
of the following statements, we refer to \cite{MW} and \cite{Wendl-low}.

Let $W$ be a closed oriented 4-manifold, $\Sigma$ and $\Sigma'$ be closed oriented surfaces, and let 
$u:\Sigma\rightarrow W$,
$v:\Sigma'\rightarrow W$ be smooth maps. An intersection $u(z)=v(w)=p$ is \textit{transverse} if  
$du(T_z\Sigma)\oplus dv(T_w\Sigma')=T_pW$. We say this intersection is \textit{positive}  if the direct
sum of the orientations of the surfaces coincides with the orientation of $W$ and say \textit{negative}
otherwise. We define \textit{the local intersection index} $\imath(u,z;v,w)$ as +1 if the intersection
is positive and -1 otherwise. We note that if an intersection is transverse then it is isolated. Hence
if all intersections of $u$ and $v$ are transverse, there are finitely many of them and we can define
the \textit{total intersection number} 
$$[u]\cdot[v]=\sum_{u(z)=v(w)}\imath(u,z;v,w).$$
It turns out that $[u]\cdot[v]$ depends only on the homology classes  $[u],[v]\in H_2(W)$. Moreover, it 
defines a bilinear symmetric form on $H_2(W)$, which is non-degenerate.

If an intersection $u(z)=v(w)=p$ is not transverse but still isolated, one can still define a local 
intersection index as follows. One localizes the intersection via closed discs $D_z$ and $D_w$ 
around $z$ and $w$. Then one picks $C^\infty$-small perturbation $u_\epsilon$ of $u$ so that 
when restricted to  $D_z$ and $D_w$, $u_\epsilon$ and $v$ have only transverse intersections and
$u_\epsilon(\partial D_z)\cap v(D_w)=\emptyset$. Then one defines 
$$\imath(u,z;v,w)=\sum_{u_\epsilon(z')=v(w')}\imath(u_\epsilon,z';v,w')$$
where the sum is taken for $(z',w')\in D_z\times D_w$.

Now we assume that $W$ is equipped with an almost complex structure $J$ and it is oriented via $J$.
We also assume that $\Sigma$ and $\Sigma'$ carry 
complex structures $j$ and $j'$ respectively, and they are oriented via $j$ and $j'$. Finally, we assume that
$u$ and $v$ are closed $J$- holomorphic curves, that is $du\circ j=J\circ du$ 
and $dv\circ j=J\circ dv$. A well known fact is that any intersection  $u(z)=v(w)=p$ of 
two such curves is either isolated or there are neighbourhoods $z\in U$ and $w\in V$ such that $u(U)=v(V)$.
We note that this phenomenon is independent of the dimension of $W$. 

The special features of the 
case $\dim W=4$ are as follows. It is clear that for any transverse intersection $u(z)=v(w)=p$,  
we have $\imath(u,z;v,w)=+1$. The non-trivial fact is that if an intersection is isolated 
then $\imath(u,z;v,w)\geq 1$, with equality if and only if the intersection is transverse. This phenomenon
is referred as \textit{local positivity of intersections} and has the following global consequence. 
By the principle of unique continuation, one can show that if $u$ and $v$ have infinitely many intersections
then ${\rm im}(u)={\rm im}(v)$, that is either one is a reparametrization of the other or they are multiple 
covers of the same simple curve. Hence, if ${\rm im}(u)\neq {\rm im}(v)$, there are finitely many intersections
and we have
$$[u]\cdot [v]\geq \# \{(z,w)\in \Sigma\times \Sigma'\,|\, u(z)=v(w)\},$$
with equality if and only if all the intersections are transverse. In particular, $[u]\cdot [v]=0$ if 
and only if ${\rm im}(u)\cap {\rm im}(v)=\emptyset$. We refer this fact as 
\textit{global positivity of intersections}.

Finally, we want to address the question of self-intersections of a holomorphic curve. Let 
$u:\Sigma\rightarrow W$ be a closed $J$- holomorphic curve. If $u$ is not simple, then it is clear that
it has infinitely many double points. But if $u$ is simple, one has only finitely many intersections. 
For a simple curve $u$, one defines 
$$\delta(u)=\frac{1}{2}\sum_{u_\epsilon(z)=u_\epsilon (w),\, z\neq w} \imath(u_\epsilon,z;v_\epsilon,w),$$
where $u_\epsilon$ is some $C^\infty$-small perturbation of $u$ that is immersed. It turns out that $\delta(u)$
is non-negative for any such perturbation $u_\epsilon$, and vanishes if and 
only if $u$ is embedded. In fact, $\delta(u)$ is independent of the chosen 
perturbation since it satisfies  the \textit{adjunction formula}:
$$[u]\cdot[u]=2\delta +c_N(u),$$
where the\textit{ normal Chern number} $c_N(u)$ is defined as $c_N(u)=c_1([u])-\chi(\Sigma)$. We note that 
$c_N(u)=c_1(N_u)$ when $u$ is immersed, where $N_u$ is the normal bundle of $u$. 

\subsubsection{Intersections in the completions of the symplectizations and symplectic cobordisms} We recall the biholomorphic identification between $(\mathbb{R}\times S^3,J_0)$ and $L^*$, where the latter space is the total space of the tautological line bundle without its zero section. Let $$u:\Sigma\setminus (\Gamma^+\cup\Gamma^-)\rightarrow \mathbb{R}\times S^3\cong L^*$$
be a finite energy $J_0$- holomorphic curve, 
where $k^+_i$'s/$k^-_j$'s are the multiplicities of positive/negative ends of $u$. We recall from the previous sections that 
$$c:=\pi\circ u:\Sigma\setminus (\Gamma^+\cup\Gamma^-)\rightarrow \mathbb{C}P^1,$$
extends to a closed curve $c:\Sigma\rightarrow \mathbb{C}P^1$ with degree $d$ 
and $u$ is identified with a meromorphic section of the bundle $c^*L\rightarrow \Sigma$, zeros corresponding the negative ends and poles corresponding 
to the positive ends, with 
$$ \#{\rm zeros} - \#{\rm poles}=-d.$$
In particular $\sum_i k^+_i- \sum_j k^-_j=d\geq 0$. We note that this inequality holds for any punctured curve in $\mathbb{R}\times S^3$ that is holomorphic with respect to  
a translation invariant almost complex structure induced by $\alpha_0$ due to the positivity of $d\alpha$-energy. Namely one has 
$$\sum_i 2\pi k^+_i- \sum_j 2\pi k^-_j\geq 0.$$
Now we want to interpret this property in terms of positivity of intersections so that it generalizes to the settings with almost complex structures that coincide with $J_0$ near the ends. 

We now consider the completion $\widehat{L}$ of $L$, where we compactify 
each fibre of $L$ by turning it into $\mathbb{C}P^1$. We get the complex manifold 
$\widehat{L}$ with
the complex structure $\hat{J}_0$ that extends $J_0$. In fact $\widehat{L}$ is nothing but $\mathbb{C}P^2$ blown up at one point but we do not need this description in what follows so we stick with our notation. Note that $\widehat{L}$ is a sphere bundle over $\mathbb{C}P^1$ such that each fibre is holomorphic. Moreover $\widehat{L}$ contains two holomorphically embedded spheres $S_0$ and $S_{\infty}$, first being the zero section of
 $L$ and second being \qq{the section at infinity}. It is easy to see that 
$$H_2(\widehat{L},\mathbb{Z})=\mathbb{Z}\cdot [S_0]\oplus \mathbb{Z}\cdot [S_{\infty}].$$
We observe that $[S_0]\cdot [S_0]=c_1(N_0)=-1$. 
Here $N_0$ is the normal bundle of $S_0$ and it can be identified with the bundle $L$. 
Similarly, $[S_{\infty}]\cdot [S_{\infty}]=c_1(N_{\infty})=1$ since the normal bundle
$N_{\infty}$ of $S_{\infty}$ can be identified with the dual bundle of $L$. It is also clear that $[S_{0}]\cdot [S_{\infty}]=0$. 

Now we extend $u$ to $\widehat{L}$ by extending the corresponding meromorphic section over the zeros and poles using the local holomorphic coordinates. We note that the extension is unique and 
we end up with a closed curve, denoted by $\hat{u}$ in $\widehat{L}$. We observe that $[\hat{u}]\cdot[S_0]=K^-$ and $[\hat{u}]\cdot[S_{\infty}]=K^+$, where $K^-:=\sum k^-_j$ and $K^+:=\sum k^+_i$. In fact, if $u$ has a negative end with multiplicity $k$ at a puncture $z$ and $c$ has the ramification number $r$ at $z$, then one can 
locally write $u(z)=(z^r,z^k)$. In order to compute the intersection number with $[S_0]$ at $z$, 
one needs to compute the local intersection number between $u$ and $v(w)=(w,0)$. Since the intersection at 0 is not transverse, we perturb $u$, and
put $u_\epsilon(z)=(z^r,z^k+\epsilon)$. We see that an intersection $(z,w)$ is a solution of the system $(z^r,z^k+\epsilon)=(w,0)$ and the second 
coordinates produce $k$-distinct roots $z_1,...,z_k$ of $-\epsilon$ and for each $z_i$, we have $w_i=z_i^r$. We note that $v'(w_i)=(1,0)$ and 
$u'(z_i)=(z_i^r, kz_i^{k-1})$ has a non-vanishing second coordinate and hence the intersections are transversal. Applying the same argument at each isolated intersection that appears at each negative end leads to the claim. Note that the same argument applies to the positive ends. 

We write the homology class $[\hat{u}]=m[S_0]+n[S_\infty]$ for some $m,n\in \mathbb{Z}$ and compute
$$K^-=[\hat{u}]\cdot[S_0]=\left(m[S_0]+n[S_\infty]\right)\cdot[S_0]=m[S_0]\cdot [S_0]+n[S_\infty]\cdot [S_0]=-m$$
$$K^+=[\hat{u}]\cdot[S_\infty]=\left(m[S_0]+n[S_\infty]\right)\cdot[S_\infty]=m[S_0]\cdot [S_\infty]+n[S_\infty]\cdot [S_\infty]=n.$$
Hence we get 
\begin{equation}\label{classofuhat}
[\hat{u}]=-K^-[S_0]+K^+[S_\infty]
\end{equation}
and therefore 
$$[\hat{u}]\cdot[\hat{u}]=(-K^-)^2[S_0]\cdot[S_0]+(K^+)^2[S_\infty]\cdot[S_\infty]=(K^+)^2-(K^-)^2.$$
We also that 
$$c_1([S_\infty])=[S_\infty]\cdot [S_\infty]+\chi(S_\infty)=1+2=3,$$
$$c_1([S_0])=[S_0]\cdot [S_0]+\chi(S_0)=-1+2=1.$$
Hence we get
\begin{equation}\label{chernofuhat}
c_1([\hat{u}])=-K^-c_1([S_0])+K^+c_1([S_\infty])=3K^+-K^-.
\end{equation} 
Next we assume that $u$ is simple. In this case $\hat{u}$ is simple as well and by the adjunction formula we get 
$$(K^+)^2-(K^-)^2=2\delta(\hat{u})+3K^+-K^--(2-2g)$$ 
where $g$ is the genus of $\Sigma$. In particular
$$(K^+)^2-(K^-)^2-3K^++K^-+2=2\delta(\hat{u})+2g\geq 0.$$
We note that $K^+\geq 1$. This is due to the fact that the almost complex structure is tamed by an exact symplectic form and therefore a curve without positive ends cannot exist. Consequently $-K^+\geq -3K^++2$ and 
$$(K^+-K^-)(K^++K^--1)=(K^+)^2-(K^-)^2-K^++K^-\geq 0.$$
Hence if $K^++K^->1$ then $K^+\geq K^-$ and if $K^++K^-=1$ then $K^-=0$ and $K^+\geq K^-$.

Now if $u$ is not simple and say $v$ is the underlying simple curve with $L^+$/$L^-$ being the total multiplicity of its positive/negative ends, then the above argument says that $L^+\geq L^-$. It is easy to see that $K^+=NL^+$ and $K^-=NL^-$, where $N$ is the degree of the underlying branched covering. Hence we get $K^+\geq K^-$.

We claim that the above discussion applies if one considers an almost complex structure with cylindrical ends. In fact if $J$ is an almost complex structure on $\mathbb{R}\times S^3$ which coincides with $J_0$ outside of a compact set then it extends to an almost complex structure $\hat{J}$ on $\widehat{L}$ and we have embedded $\hat{J}$- holomorphic spheres $S_0$ and $S_\infty$ with the properties that $[S_0]\cdot[S_0]=-1$, $[S_\infty]\cdot[S_\infty]=1$ and $[S_0]\cdot[S_\infty]=0$. Now let $u$ be a finite energy $J$- holomorphic curve and $z$ be a negative puncture with the asymptotic end of multiplicity $k$. Note that the projection $L\rightarrow \mathbb{C}P^1$ is not any more $J$- holomorphic. Nevertheless we consider a punctured disk neighbourhood $D^*$ of $z$ in $\Sigma$ and consider $u:D^*\rightarrow L^*$. This map is $J_0$- holomorphic and leads to a holomorphic map $c:D^*\rightarrow \mathbb{C}P^1$. Since $u$ has finite energy, $c$ extends over the origin and one gets a section $f:D^*\rightarrow c^*L$. Clearly the section $f$ is holomorphic and with an isolated zero of order $k$ at the origin due to the asymptotic behaviour of $u$. A similar argument applies to the positive punctures and we conclude that $u$ extends to a closed curve $\hat{u}$ in $(\widehat{L},\hat{J})$. Moreover near the intersection points the local models of $\hat{u}$ and $S_{0/\infty}$ are as before and we get
$[\hat{u}]\cdot [S_\infty]=K^+$ and $[\hat{u}]\cdot [S_0]=K^-$, where $K^+/K^-$ is the sum of the multiplicities of the positive/negative ends. Consequently (\ref{classofuhat}) holds and the rest of the above computation goes through.
\subsubsection{Proofs of Proposition \ref{actiondecrease} and Proposition \ref{actiondecreasefinal}}
We want to apply the above discussion to the lifts of curves that are given in Proposition \ref{actiondecrease} and Proposition \ref{actiondecreasefinal}. To this end we need to specify coverings of the cobordisms at hand. We fix an equivariant lift $\tilde{\varphi}:S^3\rightarrow S^3$ of the contactomorphism (\ref{contacto}). We put $\beta_0:=\tilde{f}\alpha_0=\tilde{\varphi}^*\alpha_0$. Here $\tilde{f}=f\circ\texttt{p}:S^3\rightarrow (0,+\infty)$ is the lift of (\ref{f}). We let 
$$\tilde{\Phi}: \mathbb{R}\times S^3\rightarrow \mathbb{R}\times S^3$$ be the corresponding lift of (\ref{exactsymplecto}). Here the actions of $\sigma$ and $\sigma'$ are extended to be invariant under the translation along $\mathbb{R}$-directions and $\tilde{\Phi}$ is also equivariant. We put $\widetilde{\Sigma}:=\tilde{\Phi}(\{0\}\times S^3)$. We note that for any $n$, the construction of $W_n$ given by (\ref{defnofWn}) lifts via $\tilde{\Phi}$. Namely we have the covering space $\widetilde{W}_n\rightarrow W_n$, where the group of Deck transformations is $\mathbb{Z}_p$ and for each $n$ we have an equivariant diffeomorphism 
$$\tilde{\Phi}_n:\widetilde{W}_n\rightarrow \mathbb{R}\times S^3$$
that lifts (\ref{Phi_n}). Consequently  $J_n$ lifts to an invariant almost complex structure $\tilde{J}_n$ on $\widetilde{W}_n$. Looking at the picture on the other side, $\tilde{J}_n$ is a $\sigma$-invariant almost complex structure on $\mathbb{R}\times S^3$ and we have a homotopy of invariant almost complex structures $(\tilde{J}_{t})_{t\in [0,1]}$ on $\mathbb{R}\times S^3$, which connects $J_0$ and $\tilde{J}_n$ and  lifts the the path $(J_{t})_{t\in [0,1]}$. 
\begin{figure}[h]
\includegraphics[scale=0.2]{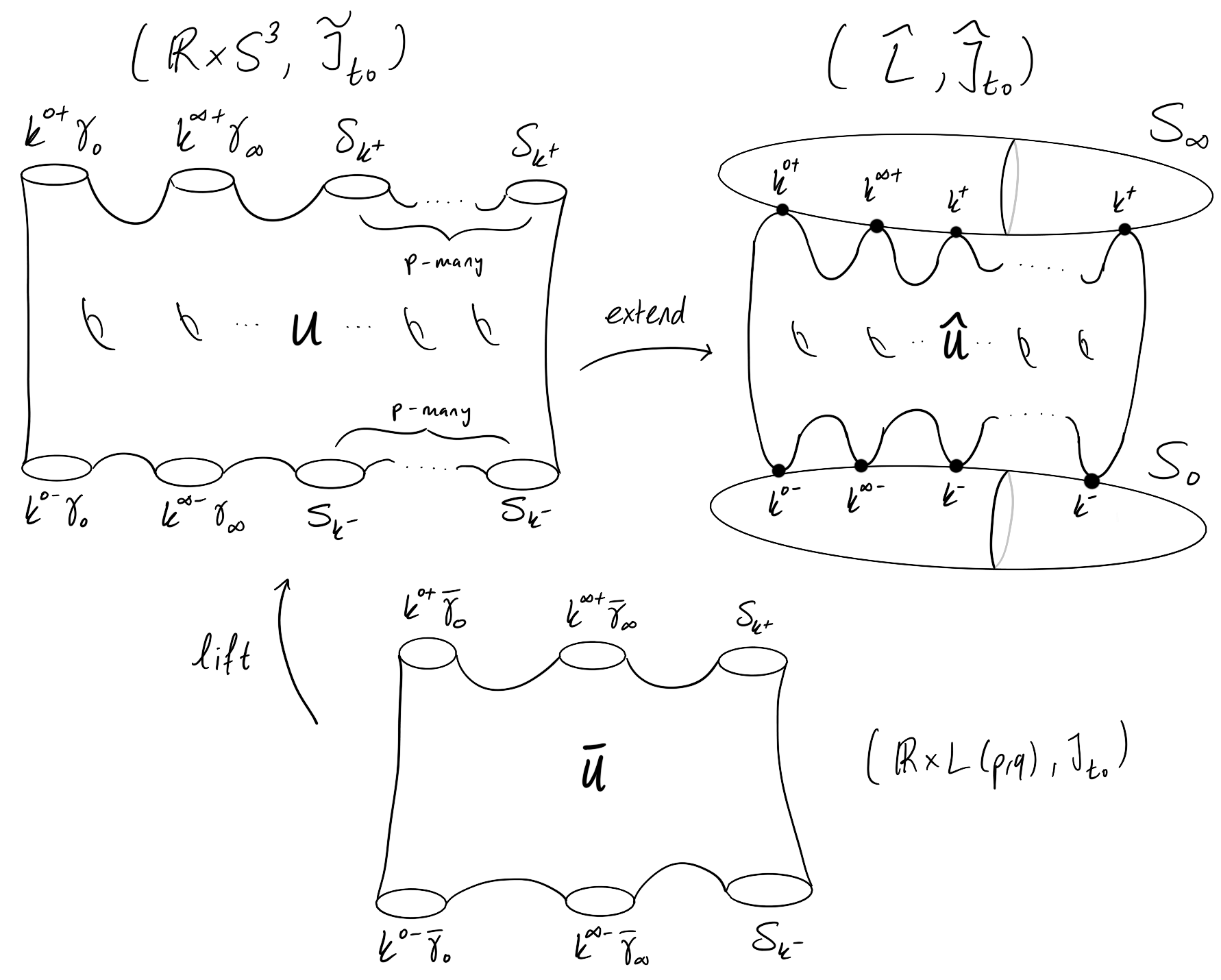}
\centering
\caption{Intersections after lifting and extending.}
\label{extendLhat}
\end{figure}

Let $J_{t_0}$ be the almost complex structure on $\mathbb{R}\times L(p,q)$ given in Proposition \ref{actiondecrease} and let $\overline{u}:\mathbb{C}P^1\setminus \Gamma\rightarrow \mathbb{R}\times L(p,q)$ be a $J_{t_0}$- holomorphic curve. First we assume that $\overline{u}$ has a non-contractible end and the asymptotic profile of $\overline{u}$ is as follows:
\begin{itemize}
\item positive ends:
    \begin{itemize}
    \item non-contractible ones: $k^{0,+}_i\overline{\gamma}_0$ where $k^{0,+}_i\not\equiv 0$ for $i=1,...,n^+_0$ and $k^{\infty,+}_i\overline{\gamma}_{\infty}$ where $k^{\infty,+}_i\not\equiv 0$ for $i=1,...,n^+_\infty$,     
    \item contractible ones: in the orbit spaces $S^+_{k^+_i}$ of action $2\pi k^+_i$ for  $i=1,...,n^+_c$
    \end{itemize}
\item negative ends:
    \begin{itemize}
    \item non-contractible ones: $k^{0,-}_i\overline{\gamma}_0$ where $k^{0,-}_i\not\equiv 0$ for $i=1,...,n^-_0$ and $k^{\infty,-}_i\overline{\gamma}_{\infty}$ where $k^{\infty,-}_i\not\equiv 0$ for $i=1,...,n^-_\infty$,
    \item contractible ones: in the orbit spaces $S^-_{k^-_i}$ of action $2\pi k^-_i$ for  $i=1,...,n^-_c$. 
     \end{itemize}     
\end{itemize}
Then using the scheme given in previous section, which is purely topological, after precomposing it with a suitable $p$-fold covering 
$\mathfrak{p}:\Sigma\setminus \tilde{\Gamma}\rightarrow \mathbb{C}P^1\setminus \Gamma$, we get a lifted $\tilde{J}_{t_0}$- holomorphic curve 
$u:\Sigma\setminus \tilde{\Gamma}\rightarrow \mathbb{R}\times S^3$. Recall that once $\mathfrak{p}$ is extended over the punctures, the punctures in $\Gamma$ with non-contractible ends are precisely the branch points of $\mathfrak{p}$ and each of these branch points has a unique preimage in $\tilde{\Gamma}$. Hence at each such preimage the ramification number is $p$. On the other hand the punctures in $\Gamma$ with contractible ends are all regular points hence each such puncture has precisely $p$-many preimages. As a result the asymptotics of $u$ are given as follows.
\begin{itemize}
\item positive ends:
    \begin{itemize}
    \item $k^{0,+}_i\gamma_0$ for $i=1,...,n^+_0$ and  $k^{\infty,+}_i\gamma_{\infty}$ for $i=1,...,n^+_\infty$,
    \item for each $i\in\{1,...,n^+_c\}$, \textbf{$p$-many}
     positive ends in the orbit space $S^+_{k^+_i}$ of multiplicity $k^+_i$. 
    \end{itemize}     
\item  negative ends:
    \begin{itemize}
    \item $k^{0,-}_i\gamma_0$ for $i=1,...,n^-_0$ and $k^{\infty,-}_i\gamma_{\infty}$ for $i=1,...,n^-_\infty$,
    \item for each $i\in\{1,...,n^-_c\}$, \textbf{$p$-many} negative ends in the orbit space $S^+_{k^-_i}$ of multiplicity $k^-_i$. 
    \end{itemize}      
 \end{itemize}
Since $\tilde{J}_{t_0}$ coincides with $J_0$ near the ends, the discussion above applies to $u$ and we get
\begin{equation}\label{ineqforlift}
\sum\limits^{n^+_{0}}_{i=1} k^{0,+}_i+\sum\limits^{n^+_{\infty}}_{i=1} k^{\infty,+}_i+p\sum\limits^{n_c^+}_{i=1}k^+_i\geq \sum\limits^{n^-_{0}}_{i=1} k^{0,-}_i+\sum\limits^{n^-_{\infty}}_{i=1} k^{\infty,-}_i+p\sum\limits^{n_c^-}_{i=1}k^-_i.
\end{equation} 
Recall that the action of the orbit $\overline{\gamma}_{0/\infty}$ is given by $2\pi/p$. Hence dividing both sides of (\ref{ineqforlift}) by $p$ and multiplying with $2\pi$ shows that the total action at the positive ends of $\overline{u}$ is not less than the total action at its negative ends. Now if $\overline{u}$ does not have any non-contractible end, it lifts without being pre-composed with a covering and the result again follows. This concludes the proof of Proposition \ref{actiondecrease}.

For Proposition \ref{actiondecreasefinal}, we repeat the same reasoning for the lifts of corresponding curves. More precisely, we let $\widetilde{W}^+\rightarrow W^+$ denote the covering space, which is determined by the lifting scheme given above and let $\tilde{J}^+$ be the lift of $J^+$. Let $\overline{u}$ be a given $J^+$- holomorphic component in $W^+$. Then the asymptotic ends of $\overline{u}$ are given as above, where for the negative non-contractible ends $\overline{\gamma}_{0/\infty}$'s are replaced with $\overline{\gamma}'_{0/\infty}$'s. Since the contactomorphism $\varphi$ induces an isomorphism on the fundamental groups $\overline{\gamma}'_{0/\infty}$'s also generate the fundamental group of $W^+$ and our lifting scheme applies the covering $\widetilde{W}^+\rightarrow W^+$ as well. Consequently we have a $\tilde{J}^+$-holomorphic lift $u:\Sigma\setminus\tilde{\Gamma}\rightarrow \widetilde{W}^+ $ with the asymptotics given as before.

Now we extend $u$ to a closed curve as follows. Topologically we view $\widetilde{W}^+$ as the the upper connected component $\widetilde{N}^+$ of $(\mathbb{R}\times S^3)\setminus \widetilde{\Sigma}$. We compactify $\widetilde{W}^+$ by collapsing its positive end, namely $S^3$, to the 2-sphere $S_\infty$ via the Reeb flow of $\alpha_0$ and by collapsing its negative end, namely $\tilde{\Sigma}$, to the 2-sphere $S_0$ via the Reeb flow of $\beta_0$. When we view $\widetilde{W}^+$ as an almost complex manifold equipped $\tilde{J}^+$, the negative end is given by $S^3$ and it is collapsed via the Reeb flow of $\alpha_0$ since $\tilde{J}^+$ coincides with $J_0$ near the negative end. We denote the resulting almost complex manifold by $(\widehat{W}^+,\hat{J}^+)$. It is not hard to see that $H_2(\widehat{W}^+)$ is generated again by two spherical classes $[S_0]$ and $[S_\infty]$ and clearly the representatives $S_0$ and $S_\infty$ are $\hat{J}^+$- holomorphic. Since $\hat{J}^+$ coincides with $\hat{J_0}$ also on a tubular neighbourhoods of $S_\infty$ in $\widehat{W}^+$, we get same intersection properties of $S_0$ and $S_\infty$ and due to the asymptotic behaviour of $u$ the equation (\ref{classofuhat}) holds and the rest of the computation goes through. Hence the total multiplicities of the positive ends of $u$ is greater or equal to the total multiplicities of the negative ends of $u$. Relating this inequality to total actions of positive and negative ends of $\overline{u}$ leads to the claim. 

We note that the same argument applies to the components in $(W^-,J^-)$ and this finishes the proof of Proposition \ref{actiondecreasefinal}.

\end{document}